\documentclass[12pt,oneside,openany,article]{memoir}
\usepackage{mempatch}

\nouppercaseheads
\usepackage[amsthm,thmmarks,hyperref]{ntheorem}
\usepackage[noTeX]{mmap}
\usepackage{etex}

\usepackage[english]{babel}
\usepackage[leqno]{mathtools}

\usepackage{mathrsfs}
\usepackage{upgreek}
\usepackage[compress,square,comma,numbers]{natbib}
\usepackage{hypernat} 

\usepackage{sfmath}

\usepackage{microtype}

\usepackage{subfig}
\usepackage{wrapfig}
\usepackage{array}

\usepackage{graphicx,color}
\definecolor{gray}{gray}{0}
\pagecolor{white}

\usepackage{hyperxmp}
\usepackage{xr-hyper}
\usepackage{nameref}
\usepackage[pdftex,bookmarks,pdfnewwindow,plainpages=false,unicode]{hyperref}

\usepackage{bookmark}

\usepackage{enumitem}

\usepackage{amssymb} 

\hypersetup{
colorlinks=true,
linkcolor=black,
citecolor=black,
urlcolor=blue,
pdfauthor={Victor Ivrii},
pdftitle={Operators with Periodic Hamiltonian Flows in Domains with the Boundary},
pdfsubject={Sharp Spectral Asymptotics},
pdfkeywords={Microlocal Analysis,  Sharp Spectral Asymptotics, Periodic flows},
bookmarksdepth={4}
}

\numberwithin{equation}{chapter}

\theoremstyle{plain}
\newtheorem{theorem}{Theorem}[section]

\newtheorem{proposition}[theorem]{Proposition}
\newtheorem{corollary}[theorem]{Corollary}

\theoremstyle{definition}

\newtheorem{Problem}[theorem]{Problem}
\theoremstyle{remark}
\newtheorem{remark}[theorem]{Remark}

\newtheorem{problem}[theorem]{Problem}

\numberwithin{equation}{section}

\DeclareMathAlphabet{\mathpzc}{OT1}{pzc}{m}{it}

 \newcommand{\cE}{\mathcal{E}}

 \newcommand{\cL}{\mathcal{L}}

 \newcommand{\cX}{\mathcal{X}}
 \newcommand{\cY}{\mathcal{Y}}
 \newcommand{\cZ}{\mathcal{Z}}

 \newcommand{\sC}{\mathscr{C}}

 \newcommand{\sH}{\mathscr{H}}

 \newcommand{\sL}{\mathscr{L}}

\newcommand{\E}{{\mathsf{E}}}
\newcommand{\TF}{{\mathsf{TF}}}

\newcommand{\sfH}{{\mathsf{H}}}

\newcommand{\Scott}{{\mathsf{Scott}}}
\newcommand{\Dirac}{{\mathsf{Dirac}}}
\newcommand{\Schwinger}{{\mathsf{Schwinger}}}

\newcommand{\D}{{\mathsf{D}}}

\newcommand{\N}{{\mathsf{N}}}

\newcommand{\W}{{{\mathsf{W}}}}

\newcommand{\y}{{\mathsf{y}}}

\newcommand{\const}{{\mathsf{const}}}
\newcommand{\dist}{{{\mathsf{dist}}}}

\newcommand{\bC}{{\mathbb{C}}}

\newcommand{\bR}{{\mathbb{R}}}

\newcommand{\bZ}{{\mathbb{Z}}}

\newcommand{\fH}{{\mathfrak{H}}}

\newcommand{\boldupsigma}{{\boldsymbol{\upsigma}}}

\def\1{\boldsymbol {|}}
\newcommand{\3}{{|\!|\!|}}
\newcommand{\Def}{\mathrel{\mathop:}=}

\newcommand{\Hess}{\operatorname{Hess}}

\newcommand{\mes}{\operatorname{mes}}

\renewcommand{\Re}{\operatorname{Re}}       

\newcommand{\Spec}{\operatorname{Spec}}

\newcommand{\supp}{\operatorname{supp}}

\newcommand{\tr}{\operatorname{tr}}
\newcommand{\Tr}{\operatorname{Tr}}

\newenvironment{claim}[1][{\textup{(\theequation)}}]{\refstepcounter{equation}\vglue10pt
\begin{trivlist}
\item[{\hskip\labelsep#1}]}{\vglue10pt\end{trivlist}}

\newenvironment{claim*}[1][{}]{\vglue10pt
\begin{trivlist}
\item[{\hskip\labelsep#1}]}{\vglue10pt\end{trivlist}}

\newenvironment{phantomequation}[1][]{\refstepcounter{equation}}{}
\newcounter{note}

\DeclareTextCommand{\textbeta}{PU}{\83\262}
\DeclareTextCommand{\textmu}{PU}{\80\265}
\DeclareTextCommand{\texttau}{PU}{\83\304}

\DeclareTextCommand{\textlesssim}{PU}{\9042\162}
\DeclareTextCommand{\textgtrsim}{PU}{\9042\163}

\DeclareTextCommand{\textpartial}{PU}{\9042\002}
\DeclareTextCommand{\texttwosuperior}{PU}{\80\262}
\DeclareTextCommand{\textGamma}{PU}{\83\223}
\DeclareTextCommand{\textxinferior}{PU}{\9040\223}
\DeclareTextCommand{\textiinferior}{PU}{\9035\142}
\DeclareTextCommand{\textjinferior}{PU}{\9054\174}

\DeclareTextCommand{\textge}{PU}{\9042\145}
\DeclareTextCommand{\textle}{PU}{\9042\144}
\DeclareTextCommand{\texthat}{PD1}{\136}

\begin{document}
\title{Asymptotics of the ground state energy of heavy atoms and molecules in combined magnetic field}
\author{Victor Ivrii}

\maketitle
{\abstract%
We consider asymptotics of the ground state energy of heavy atoms  in the combined magnetic field and derive it including Scott, and in some cases even Schwinger and Dirac corrections (if magnetic field is not too strong).

In the next versions we will consider also molecules and  related topics:  an excessive negative charge, ionization energy and excessive positive charge when atoms can still bind into molecules.
\endabstract}

\chapter{Introduction}
\label{sect-27-1}

In this Chapter instead of Schr\"odinger operator without magnetic field as in Chapter~\ref{book_new-sect-24}, or with a constant magnetic field as in Chapter~\ref{book_new-sect-25}, or with a self-generated magnetic field as in Chapter~\ref{book_new-sect-26} we consider Schr\"odinger operator (\ref{book_new-26-1-1})
with unknown magnetic field $A$ but then we add to the ground state energy of the atom (or molecule) the energy of the self-generated magnetic field (see selected term in (\ref{27-1-1}) thus arriving to
\begin{equation}
\E(A)= \inf \Spec (\sfH_{A,V} ) +
\underbracket{\alpha^{-1} \int |\nabla \times (A-A^0)|^2\,dx}
\label{27-1-1}
\end{equation}
where $A^0=\frac{1}{2} B (-x_2,x_1,0)$ is a constant \index{magnetic field!external}\emph{external magnetic field\/}.

Then finally
\begin{equation}
\E ^*=\inf_{A-A^0\in \sH^1_0} \E(A)
\label{27-1-2}
\end{equation}
defines a ground state energy with a \index{magnetic field!combined}\emph{combined magnetic field\/} $A$ while $A'\Def A-A^0$ is \index{magnetic field!self-generated}\emph{self-generated magnetic field\/}.

Note that
\begin{equation}
\int |\nabla \times (A-A^0)|^2\,dx =
\int \bigl(|\nabla \times A|^2-|\nabla \times A^0|^2\bigr) \,dx
\label{27-1-3}
\end{equation}
which seems to be a more ``physical'' definition.

\subsection*{Plan of the Chapter}
\label{sect-27-1-a}

First of all, we are lacking so far a \emph{semiclassical local theory\/} and we are developing it in Sections~\ref{sect-27-2}--\ref{sect-27-4} where we consider one-particle quantum Hamiltonian (\ref{27-2-1}) with an external constant magnetic field $A^0$ of intensity $\beta$, $h\ll 1$ and a self-generated magnetic field $(A-A^0)$. Here theory significantly depends if
$\beta h\lesssim 1$ or $\beta h\gtrsim 1$ and with self-generated magnetic field.

While Section~\ref{sect-27-2} is preparatory and rather functional-analytical, Sections~\ref{sect-27-3} and~\ref{sect-27-4} are microlocal; they cover cases $\beta h\lesssim 1$ and $\beta h\gtrsim 1$ respectively. These three sections are similar to a single Section~\ref{book_new-sect-26-2}. However in Sections~\ref{sect-27-3} and~\ref{sect-27-4} various non-degeneracy assumptions play a very significant role, especially for large $\beta$.

Then in Section~\ref{sect-27-5} we consider a global theory if a potential has Coulomb singularities and (in some statements) behaves like (magnetic) Thomas-Fermi potential both near singularities and far from them.

Finally, in Section~\ref{sect-27-6} we apply these results to our original problem of the ground state energy so far assuming that the number of nuclei is $1$. One can recover the same results as $M\ge 2$ but the external magnetic field $B$ is weak enough. No surprise that the theory is different in the cases $B\le Z^{\frac{4}{3}}$ and $Z^{\frac{4}{3}}\le B\ll Z^3$ (see Chapter~\ref{book_new-sect-25} where this difference appears). Since as $M=1$ the strongest non-degeneracy assumption is surely achieved and as $M\ge 2$ much weaker non-degeneracy assumption is achieved in the border zone (see Chapter~\ref{book_new-sect-25}) our remainder estimates for large $B$ significantly differ in the atomic and molecular cases.

In Appendix~\ref{sect-27-A} we first generalize Lieb-Loss-Solovej estimate to the case of the combined magnetic field (which is necessary as
$\beta h\gtrsim 1$), then establish electrostatic inequality in the current settings and finally study very special pointwise spectral expressions for a Schr\"odinger operator in $\bR^3$ with linear magnetic and scalar potentials (we considered such operators already in Section~\ref{book_new-sect-16-5})

\subsection*{Unfinished business}
\label{sect-27-1-b}

One can apply these results to estimates of the excessive positive and negative charges (latter--as $M\ge 2$ in the free nuclei framework) and estimates or asymptotics of the ionization energy in the same manner as we did it in Chapters~\ref{book_new-sect-24}--\ref{book_new-sect-26}; however there are no new ideas but rather tedious calculations and we leave it to those readers who decide to explore these topics, which is clearly serious task.

\chapter{Local semiclassical trace asymptotics: Preparation}
\label{sect-27-2}

\section{Toy-model}
\label{sect-27-2-1}

Let us consider operator (\ref{book_new-26-1-4})
\begin{equation}
H=H _{A,V}=\bigl((h D -A)\cdot \boldupsigma \bigr) ^2-V(x)
\label{27-2-1}
\end{equation}
in $\bR^3$ where $A,V$ are real-valued functions and
$V\in \sC^{\frac{5}{2}}$, $A-A^0\in \sH^1_0$. Then operator $H_{A,V}$ is self-adjoint. We are interested in $\Tr^- (H_{A,V})= \Tr^- (H_{A,V}^-)$ (the sum of all negative eigenvalues of this operator). Let
\begin{gather}
\E^*=\E ^*_\kappa\Def \inf_{A-A^0\in \sH^1_0}\E(A),\label{27-2-2}\\
\shortintertext{where}
\E(A)=\E_\kappa (A)\Def \Bigl( \Tr^-(H_{A,V} ) +
\kappa^{-1} h^{-2}\int |\partial (A-A^0)|^2\,dx\Bigr)
\label{27-2-3}
\end{gather}
with $\partial A=(\partial_i A_j)_{i,j=1,2,3}$ a matrix. Recall that $A^0$ is a linear potential, $A^0(x)=\frac{1}{2}\beta (-x_2,x_1,0)$. We consider rather separately cases\begin{phantomequation}\label{27-2-4}\end{phantomequation}
\begin{equation}
\beta h\lesssim 1 \qquad \text{and} \qquad \beta h\gtrsim 1
\tag*{$\textup{(\ref*{27-2-4})}_{1,2}$}\label{27-2-4-*}
\end{equation}
of the \index{magnetic field!external!moderate}\emph{moderate\/} and \index{magnetic field!external!strong}\emph{strong\/} external magnetic field.

To deal with the described problem we need to consider first a formal semiclassical approximation.

\section{Formal semiclassical theory}
\label{sect-27-2-2}

\subsection{Semiclassical theory: \texorpdfstring{$\beta h\lesssim 1$}{\textbeta h \textlesssim 1}}
\label{sect-27-2-2-1}

Let us replace trace expression $\Tr (H^-_{A,V}\psi )$ by its magnetic semiclassical approximation
$-h^{-3}\int P_{B h}(V)\psi \,dx$ where
$B=|\nabla \times A|$ is a scalar intensity of the magnetic field and $P_*(.)$ is a pressure. Then
$\E(A)\approx \cE(A)$ with
\begin{equation}
\cE(A)=\cE_\kappa (A)\Def
-h^{-3}\int P_{Bh}(V)\psi \,dx + \frac{1}{\kappa h^2} \int |\partial A'|^2\,dx.
\label{27-2-5}
\end{equation}
Assuming that $|\partial A'|\ll \beta$, $A'=(A-A^0)$ we find out that

\noindent\begin{minipage}{1.1\linewidth}
\begin{multline}
-h^{-3}\int P_{B h}(V)\psi \,dx
\approx -h^{-3}\int \Bigl(P_{\beta h}(V)-
\partial_\beta P_{\beta h} (V)(B-\beta)\psi\Bigr) \,dx \\
\begin{aligned}
&\approx -h^{-3}\int P_{\beta h}(V)\psi\,dx \\
&-h^{-3}\int \Bigl[\partial_{x_2} \bigl(
\partial_\beta P_{\beta h}(V)\psi\bigr)\cdot A'_1-
\partial_{x_1} \bigl(
\partial_\beta P_{\beta h}(V)\psi\bigr) \cdot A'_2\Bigr]\,dx
\end{aligned}
\label{27-2-6}
\end{multline}\end{minipage}

\bigskip\noindent
where we used that $B\approx \beta - \partial_{x_2}A'_1+\partial_{x_1}A'_2$ and integrated by parts. Then $\cE (A) \approx \bar{\cE}(A)$ with
\begin{multline}
\bar{\cE}(A)=\bar{\cE}_\kappa (A)\Def
-h^{-3}\int P_{\beta h}(V)\psi\,dx \\
-h^{-3}\int \Bigl[\partial_{x_2} \bigl(
\partial_\beta P_{\beta h}(V)\psi\bigr) \cdot A'_1-
\partial_{x_1} \bigl(
\partial_\beta P_{\beta h}(V)\psi\bigr) \cdot A'_2\Bigr]\,dx +
\frac{1}{\kappa h^2} \int |\partial A'|^2\,dx
\label{27-2-7}
\end{multline}
and replacing approximate equalities by exact ones and optimizing with respect to $A'$ we arrive to
\begin{multline}
\Delta A'_1= -\frac{1}{2}\kappa h^{-1}
\partial_{x_2}\bigl(\partial_\beta P_{\beta h}(V)\psi\bigr), \quad
\Delta A'_2=
\frac{1}{2} \kappa h^{-1}\partial_{x_1} \bigl(\partial_\beta P_\beta(V)\psi\bigr),\\
\Delta A'_3=0
\label{27-2-8}
\end{multline}
and
\begin{equation}
\cE^*_\kappa\Def \inf _{A:\, A-A^0\in \sH^1_0} \cE_\kappa(A) \approx \bar{\cE}^*_\kappa \Def \inf _{A:\, A-A^0\in \sH^1_0} \bar{\cE}_\kappa(A).
\label{27-2-9}
\end{equation}

To justify our analysis we need to justify approximate equality
\begin{multline}
-h^{-3}\int P_{Bh}(V)\,dx + \frac{1}{\kappa h^2} \int |\partial A'|^2\,dx \\
\approx -h^{-3}\int P_{\beta h}(V)\psi dx-
h^{-3}\int \Bigl[\partial_\beta P_{\beta h}(V)\psi
(-\partial_{x_2} A'_1+ \partial_{x_1} A'_2)\Bigr]\,dx+\\
\frac{1}{\kappa h^2} \int |\partial A'|^2\,dx
\label{27-2-10}
\end{multline}
and estimate an error when we minimize the right-hand expression instead of the left-hand one. To do this observe that (even without assumptions $Bh\lesssim 1$,
$\beta h\lesssim 1$)
\begin{multline}
|P_{Bh}(V)-P_{\beta h}(V)- \partial_{\beta h} P_{\beta h} (V) \cdot (B-\beta)h |\le\\
C(B-\beta)^2h^2+ C|B-\beta|^{\frac{3}{2}}\beta h^{\frac{5}{2}}.
\label{27-2-11}
\end{multline}
Indeed, one can prove it easily recalling that
\begin{equation}
P_{\beta h} (V)=\varkappa_0 \sum_{j\ge 0} (1-\frac{1}{2}\updelta_{j0})
(V-2j \beta h)_+^{\frac{3}{2}} \beta h
\label{27-2-12}
\end{equation}
and considering cases $\beta h \gtrless 1$, $B h \gtrless 1$,
$|B -\beta | \gtrless \beta h$, and considering different terms in (\ref{27-2-12}) and observing that the last term in (\ref{27-2-11}) appears only in the case $\beta h\lesssim 1$, $|B-\beta|\lesssim \beta $.

Then since $|B-\beta |\le |B'|$ (where $B'=|\partial (A-A^0)|$) we conclude that the left-hand expression of (\ref{27-2-10}) is greater than
\begin{equation*}
-h^{-3}\int P_{\beta h}(V)\psi dx- C\|B'\| \beta h^{-1} - C\|B'\|^{\frac{3}{2}}\beta h^{-\frac{1}{2}} + \kappa^{-1}h^{-2}\|B'\|^2
\end{equation*}
where we used that
\begin{equation}
| \partial_{\beta h} P_{\beta h} (V) |\le C\beta h\qquad \text{as\ \ } V\le c;
\label{27-2-13}
\end{equation}
then a minimizer for the left-hand expression of (\ref{27-2-10}) must satisfy
\begin{equation}
\|B'\|\le C\kappa \beta h
\label{27-2-14}
\end{equation}
and one can observe easily that the same is true and for the minimizer for the right-hand expression as well.

Also observe that
\begin{equation}
B=\beta + \partial_{x_1}A'_2-\partial_{x_2}A'_1+ O(\beta^{-1}|B'|^2).
\label{27-2-15}
\end{equation}
Then for both minimizers the difference between the left-hand and right-hand expressions of (\ref{27-2-10}) does not exceed $C\kappa^{\frac{5}{2}}\beta^{\frac{5}{2}}$ and therefore
\begin{equation}
|\cE^* -\bar{\cE}^*|\le C\kappa^{\frac{5}{2}}\beta^{\frac{5}{2}}.
\label{27-2-16}
\end{equation}

One can calculate easily the minimizer for
\begin{equation}
-h^{-3}\int \Bigl[\partial_\beta P_{\beta h}(V)\psi
(-\partial_{x_2} A'_1+ \partial_{x_1} A'_2)\Bigr]\,dx+
\frac{1}{\kappa h^2} \int |\partial A'|^2\,dx
\label{27-2-17}
\end{equation}
and conclude that $A'_j=\kappa \beta h a_j$ with $a_3=0$ and
\begin{equation}
\Delta a_1=- (\beta h)^{-1}\partial{x_2}\partial _{\beta h} P_{\beta h}(V),\quad
\Delta a_2= (\beta h)^{-1}\partial{x_1}\partial _{\beta h} P_{\beta h}(V)
\label{27-2-18}
\end{equation}
and the minimum is negative and $O(\kappa \beta^2)$; we call it \emph{correction term\/}; in fact,
$\|\partial A'\|\asymp \kappa \beta h$ and the minimum is
$\asymp -\kappa \beta^2$ in the generic case.

Then a minimum of the left-hand expression of (\ref{27-2-10}) is equal to the minimum of the right-hand expression modulo
$O(\kappa^{\frac{3}{2}}\beta^{\frac{5}{2}} h)$.

\begin{remark}\label{rem-27-2-1}
\begin{enumerate}[label=(\roman*), fullwidth]
\item\label{rem-27-2-1-i}
One can improve this estimate under non-degeneracy assumptions (\ref{27-3-60}) or (\ref{27-3-65}). However even in the general case observe that
\begin{multline*}
\hskip-.1in
\cE(A'')-\cE (A') \\
\ge -C\beta h^{-\frac{1}{2}}\|B'-B''\|^{\frac{3}{2}} -
C\beta h^{-\frac{1}{2}}\|B'\|^{\frac{1}{2}}\cdot \|B'-B''\| +
2\epsilon_0\kappa^{-1}h^{-2}\|\partial (A'-A'')\|^2\\[3pt]
\ge
 -C\kappa^2 \beta^{2}h +
\epsilon_0\kappa^{-1}h^{-2}\|\partial (A'-A'')\|^2\hskip-.15in
\end{multline*}
if $A'$ is the minimizer for $\bar{\cE}$ and therefore since
$\|B'\|\le C\kappa \beta h$ we conclude that
\begin{gather}
\cE^* \ge \cE (A') -C\kappa^2 \beta^{2}h
\label{27-2-19}
\shortintertext{and}
\|\partial (A'-A'')\|\le C\kappa^{\frac{3}{2}}\beta h^{\frac{1}{2}}
\label{27-2-20}
\end{gather}
if $A''$ is an almost-minimizer for $\cE(A'')$.

\item\label{rem-27-2-1-ii}
Observe, that picking up $A'=0$ and applying arguments of Chapter~\ref{book_new-sect-18} we can derive an upper estimate
\begin{equation*}
\E^* \le  -\int P_{\beta h}(V)\psi\,dx + O(h^{-1});
\end{equation*}
however this estimate is not sharp as
$\kappa \beta^2\gg h^{-1}$ as $\cE^* $ is less than the main term here with a gap $\asymp \kappa \beta^2$. As $\kappa \asymp 1$ it gives us a proper upper estimate only as $\beta \le h^{-\frac{1}{2}}$.

Therefore as $\kappa \beta^2\gg h^{-1}$ an upper estimate is not as trivial as in Chapter~\ref{book_new-sect-26}; in the future we pick up as $A'$ a minimizer for $\bar{\cE}(A)$ (mollified by $x$ as this minimizer is not smooth enough).
\end{enumerate}
\end{remark}

\subsection{Semiclassical theory: \texorpdfstring{$\beta h\gtrsim 1$}{\textbeta h \textgtrsim 1}}
\label{sect-27-2-2-2}

Consider $\beta h\ge 1$. Without any loss of the generality one can assume that $\|V\|_{\sL^\infty} \le \beta h$ and
$\|\partial A'\|_{\sL^\infty} \le \frac{1}{2}\beta$. Then in the definition (\ref{27-2-12}) of $P_{\beta h}(V)$ (etc) remains only term with $j=0$:
\begin{equation}
P_{\beta h} (V)= \frac{1}{2}\varkappa_0
V_+^{\frac{3}{2}} \beta h
\tag*{$\textup{(\ref*{27-2-12})}'$}\label{27-2-12-'}
\end{equation}
which leads to simplification of $\cE(A')$ and $\bar{\cE}(A')$; both of them become
\begin{multline}
\frac{1}{2}\varkappa_0 \beta h^{-2} \int V_+^{\frac{3}{2}}\psi\,dx \\+
\frac{1}{2}\varkappa_0 h^{-2}
\int V_+^{\frac{3}{2}}\bigl(\partial_{x_1}A'_2-\partial_{x_2}A'_1\bigr)\psi\,dx + \kappa^{-1}h^{-2}\|\partial A'\|^2
\tag*{$\textup{(\ref*{27-2-7})}'$}\label{27-2-7-'}
\end{multline}
modulo $O(\beta^{-1}h^{-2}\|B'\|^2 )$ and equations to the minimizer become
\begin{equation}
\Delta A'_1= -\frac{1}{2}\varkappa_0\kappa
\partial_{x_2}\bigl(V_+^{\frac{3}{2}}\psi\bigr), \quad
\Delta A'_2=
\frac{1}{2}\varkappa_0\kappa
\partial_{x_1}\bigl(V_+^{\frac{3}{2}}\psi\bigr),\quad
\Delta A'_3=0.
\tag*{$\textup{(\ref*{27-2-8})}'$}\label{27-2-8-'}
\end{equation}
Then $\|B'\| \asymp \kappa $ and a correction term is negative and
$\asymp -\kappa h^{-2}$ in the generic case. An error $O(\beta^{-1}h^{-2}\|B'\|^2 )$ becomes $O(\kappa \beta^{-1}h^{-2})$ (and thus not exceeding microlocal error $O(\beta)$).

\section{Estimate from below}
\label{sect-27-2-3}

\subsection{Basic estimates}
\label{sect-27-2-3-1}

Let us estimate $\E(A)$ from below. First we need the following really simple

\begin{proposition}\footnote{\label{foot-27-1} Cf. Proposition~\ref{book_new-prop-26-2-1}.}\label{prop-27-2-2}
Consider operator $H_{A,V}$ defined on $\sH^2(B(0,1))\cap \sH^1_0 (B(0,1))$\,\footnote{\label{foot-27-2} I.e. on $\sH^2(B(0,1))$ with the Dirichlet boundary conditions.}. Let $V\in \sL^4$.
\begin{enumerate}[label=(\roman*), fullwidth]
\item\label{prop-27-2-3-i}
Let $\beta h \le 1$. Then
\begin{gather}
\E^*\ge -C h^{-3}
\label{27-2-21}\\
\shortintertext{and either}
\frac {1}{\kappa h^2} \int |\partial (A-A^0)|^2\,dx \le C h^{-3}
\label{27-2-22}
\end{gather}
or $\E (A) \ge ch^{-3}$;
\item\label{prop-27-2-3-ii}
Let $\beta h\ge 1$. Then
\begin{gather}
\E^*\ge -C\beta h^{-2}- C\kappa^{\frac{1}{3}}\beta^{\frac{4}{3}} h^{-\frac{4}{3}}
\label{27-2-23}\\
\shortintertext{and either}
\frac {1}{\kappa h^2} \int |\partial A'|^2\,dx \le
C\beta h^{-2}+ C\kappa^{\frac{1}{3}}\beta^{\frac{4}{3}} h^{-\frac{4}{3}}
\label{27-2-24}
\end{gather}
or $\E (A) \ge
C\beta h^{-2}+ C\kappa^{\frac{1}{3}}\beta^{\frac{4}{3}} h^{-\frac{4}{3}}$;
\item\label{prop-27-2-3-iii}
Furthermore, if $\beta h\ge 1$ and
\begin{gather}
\kappa \beta h^2\le c\label{27-2-25}\\
\shortintertext{then}
\E^*\ge -C\beta h^{-2}\label{27-2-26}\\
\shortintertext{and either}
\frac {1}{\kappa h^2} \int |\partial A'|^2\,dx \le C\beta h^{-2}
\label{27-2-27}
\end{gather}
or $\E (A) \ge C\beta h^{-2}$;
\end{enumerate}
\end{proposition}

\begin{proof}
Using estimate (\ref{27-A-2})\,\footnote{\label{foot-27-3} Magnetic Lieb-Thirring inequality (5) of E.~H.~Lieb, M.~Loss,~M. and J.~P.~Solovej~\cite{lieb:loss:solovej}) would be sufficient as $\beta h\le 1$ but will lead to worse estimate than we claim as $\beta h\ge 1$.} we have
\begin{multline}
\E(A) \ge -C(1+\beta h) h^{-3}\\
\shoveright{-
C\beta h^{-\frac{3}{2}} \Bigl(\int |\partial A'|^2\,dx\Bigr)^{\frac{1}{4}}-
C h^{-\frac{3}{2}} \Bigl(\int |\partial A'|^2\,dx\Bigr)^{\frac{3}{4}}}
+\frac{1}{\kappa h^2}\Bigl(\int |\partial A'|^2\,dx\Bigr)
\label{27-2-28}
\end{multline}
which implies both Statements~\ref{prop-27-2-3-i}--\ref{prop-27-2-3-ii}.
\end{proof}

\begin{remark}\label{rem-27-2-3}
\begin{enumerate}[label=(\roman*), fullwidth]
\item\label{rem-27-2-3-i}
Definitely we would prefer to have an estimate
\begin{multline}
\E(A) \ge -C(1+\beta h) h^{-3}\\-
C h^{-2} \Bigl(\int |\partial A'|^2\,dx\Bigr)^{\frac{1}{2}}
+\frac{1}{\kappa h^2}\Bigl(\int |\partial A'|^2\,dx\Bigr)
\label{27-2-29}
\end{multline}
from the very beginning but we cannot prove it without some smoothness conditions to $A$ and they will be proven only later under the same assumption (\ref{27-2-25});
\item\label{rem-27-2-3-ii}
This assumption (\ref{27-2-25}) in a bit stronger form \ref{27-2-25-*} will be required for our microlocal analysis in Section~\ref{sect-27-3}.
\end{enumerate}
\end{remark}

\begin{remark}\label{rem-27-2-4}
\begin{enumerate}[label=(\roman*), fullwidth]
\item \label{rem-27-2-4-i}
Proposition \ref{book_new-prop-26-2-2} (existence of the minimizer) remains valid;
\item \label{rem-27-2-4-ii}
As in Remark~\ref{book_new-rem-26-2-3} we do not know if the minimizer is unique. From now on until further notice let $A$ be a minimizer; we also assume that $V$ is sufficiently smooth ($V\in \sC^{2+\delta}$);

\item \label{rem-27-2-4-iii}
Proposition \ref{book_new-prop-26-2-4} (namely, equation (\ref{book_new-26-2-14}) to a minimizer) remains valid for both $A$ and $A'$.
\end{enumerate}
\end{remark}

\begin{proposition} \label{prop-27-2-5}
\begin{enumerate}[fullwidth, label=(\roman*)]
\item \label{prop-27-2-5-i}
Let $\beta h\le 1$, $0<\kappa \le (1-\epsilon_0)\kappa^*$ and
\begin{gather}
\E^*(\kappa^*,\beta,h) \ge \cE - CM, \label{27-2-30}\\[3pt]
\E^*(\kappa,\beta,h) \le \cE + CM\label{27-2-31}
\intertext{with the same number $\cE$ and with
$M\ge C h^{-1}+C\kappa^* \beta^2$. Then for this $\kappa$}
\int |\partial A'|^2\,dx \le C_1\kappa h^2M;
\label{27-2-32}
\end{gather}
\item \label{prop-27-2-5-ii}
Let $\beta h\ge 1$, $\kappa^* \beta h \le c$,
$0<\kappa \le (1-\epsilon_0)\kappa^*$ and \textup{(\ref{27-2-30})}--\textup{(\ref{27-2-31})} be fulfilled with the same number $\cE$ and with $M\ge C \beta +C\kappa^* h^{-2}$. Then for this $\kappa$
estimate \textup{(\ref{27-2-32})} holds.
\end{enumerate}
\end{proposition}

\begin{proof}
Proof is obvious\footnote{\label{foot-27-4} Cf. Proposition~\ref{book_new-prop-26-2-5}.}.
\end{proof}

\subsection{Estimates to minimizer: \texorpdfstring{$\beta h\lesssim 1$}{\textbeta h \textlesssim 1}}
\label{sect-27-2-3-2}

Consider first simpler case $\beta h\le 1$.

\begin{proposition}\label{prop-27-2-6}
Let $\beta h\le 1$. Then as $\mu =\|\partial A'\|_\infty$
\begin{gather}
|e(x,y,\tau)| \le
C\bigl(1+ \mu^{\frac{1}{2}}h^{\frac{1}{2}}\bigr)h^{-3}
\label{27-2-33}\\
\shortintertext{and}
|((hD-A)_x\cdot \boldupsigma )e(x,y,\tau)|  \le
C\bigl(1+ \mu^{\frac{1}{2}}h^{\frac{1}{2}}\bigr)h^{-3}.
\label{27-2-34}
\end{gather}
\end{proposition}

\begin{proof}
Without any loss of the generality one can assume that $\mu \ge h^{-1}$.
Consider
$\mu ^{-\frac{1}{2}}h^{\frac{1}{2}}$ element in $\bR^3_x$. Without any loss of the generality one can assume that $A^0=A'=0$ in its center $z$.

Since both operators $E(\tau)$ and $((hD-A)_x\cdot \boldupsigma )E(\tau)$ have their operator norms bounded by $c$ in $\sL^2$ one can prove easily that operators $\phi D^\alpha E(\tau)$ and
$\phi D^\alpha ((hD-A)_x\cdot \boldupsigma )E(\tau)$ have their operator norms bounded by $C \zeta^{|\alpha|}$ with $\zeta= \mu^{\frac{1}{2}}h^{-\frac{1}{2}}$ as $\alpha \in \{0,1\}^3$ and $\phi $ is supported in the mentioned element.

Then operator norms of operators $\upgamma_x E(\tau)$ and
$\upgamma_x ((hD-A)_x\cdot \boldupsigma )E(\tau)$ from $\sL^2$ to $\bC^q$ do not exceed $C_0\zeta_1\zeta_3^{\frac{1}{2}}$ and therefore the same is true for adjoint operators; here $\upgamma_z$ is operator of restriction to $x=z$.

Since $E(\tau)^*=E(\tau)^2=E(\tau)$ we conclude that the left-hand expressions in (\ref{27-2-33}) and (\ref{27-2-34}) do not exceed $C\zeta ^3$ which is exactly the right-hand expressions.
\end{proof}

Then from equation (\ref{book_new-26-2-14}) which remains valid (see Remark~\ref{rem-27-2-4}\ref{rem-27-2-4-iii}) we conclude that
\begin{gather}
\|\Delta A'\|_{\sL^\infty} \le C\kappa
\bigl(1+ \mu^{\frac{1}{2}}h^{\frac{1}{2}}\bigr)h^{-1}\label{27-2-35}\\
\shortintertext{and therefore}
\|\partial ^2 A'\|_{\sL^\infty}\le C\kappa|\log h |
 \bigl(1+ \mu^{\frac{1}{2}}h^{\frac{1}{2}}\bigr)h^{-1}.
\label{27-2-36}
\end{gather}

Further, combining (\ref{27-2-36}) with (\ref{27-2-22})  and standard inequality
$\|\partial A'\|_{\sL^\infty}\le \|\partial ^2 A'\|_{\sL^\infty}^{\frac{3}{5}} \cdot \|\partial A'\|^{\frac{2}{5}}$ we conclude that
\begin{gather}
\mu \le C(\kappa|\log h|)^{\frac{4}{5}}
\bigl(1+ \mu^{\frac{1}{2}}h^{\frac{1}{2}}\bigr) ^{\frac{4}{5}} h^{-\frac{4}{5}}\notag\\
\shortintertext{and then}
\|\partial A'\|_{\sL^\infty} \le C (\kappa |\log h|) ^{\frac{4}{5}} h^{-\frac{4}{5}}
\label{27-2-37}\\
\intertext{and therefore due to (\ref{27-2-36})}
\|\partial A'\|_{\sL^\infty} \le C\kappa |\log h| h^{-1}
\label{27-2-38}
\end{gather}
(where for a sake of simplicity we slightly increase power of logarithm) thus arriving to

\begin{proposition}\label{prop-27-2-7}
Let $\beta h\le 1$ and $\kappa \le \kappa^*$. Then estimates \textup{(\ref{27-2-37})} and \textup{(\ref{27-2-38})} hold.
\end{proposition}

Furthermore, the standard scaling arguments applied to the results of Section~\ref{book_new-sect-26-2} imply that in fact the left-hand expressions of estimates (\ref{27-2-34}) and (\ref{27-2-22}) do not exceed
$C(1+\beta +\mu)h^{-2}$ and $C(1+\beta +\mu)^2h^{-1}$ respectively and then
$\|\partial ^2 A'\|_{\sL^\infty}$ does not exceed
$C\kappa|\log h |(1+\beta +\mu)$ and $\|\partial A'\|$ does not exceed $C\kappa^{\frac{1}{2}}(1+\beta +\mu)h^{\frac{1}{2}}$ and then
$\mu \le
C(\kappa|\log h |)^{\frac{4}{5}} (1+\beta +\mu)^{\frac{4}{5}} h^{\frac{1}{5}}$
which implies $\mu
\le C(\kappa|\log h |)^{\frac{4}{5}} (1+\beta )^{\frac{4}{5}} h^{\frac{1}{5}}$
and we arrive to

\begin{proposition}\label{prop-27-2-8}
Let $\beta h\le 1$ and $\kappa \le \kappa^*$. Then
\begin{gather}
\|\partial A'\|_{\sL^\infty} \le
C (\kappa |\log h|) ^{\frac{4}{5}} (1+\beta )^{\frac{4}{5}} h^{\frac{1}{5}}
\label{27-2-39}\\
\intertext{and}
\|\partial^2 A'\|_{\sL^\infty} \le C\kappa |\log h| (1+\beta).
\label{27-2-40}
\end{gather}
\end{proposition}

\subsection{Estimates to minimizer: \texorpdfstring{$\beta h\gtrsim 1$}{\textbeta h \textgtrsim 1}}
\label{sect-27-2-3-3}

Consider more complicated case $\beta h\ge 1$.

\begin{proposition}\label{prop-27-2-9}
Let $\beta h\ge 1$. Then as $\mu =\|\partial A'\|_\infty$
\begin{gather}
|e(x,y,\tau)| \le C\bigl(\beta +\mu\bigl)
\bigl(1+ \mu^{\frac{1}{2}}h^{\frac{1}{2}}\bigr)h^{-2}
\label{27-2-41}\\
\shortintertext{and}
|((hD-A)_x\cdot \boldupsigma )e(x,y,\tau)|  \le
C\bigl(\beta +\mu\bigl)
\bigl(1+ \mu^{\frac{1}{2}}h^{\frac{1}{2}}\bigr)h^{-2}
\label{27-2-42}
\end{gather}
\end{proposition}

\begin{proof}
Without any loss of the generality one can assume that $\mu \le \beta$.
Consider $(\beta ^{-\frac{1}{2}}h^{\frac{1}{2}},
\beta ^{-\frac{1}{2}}h^{\frac{1}{2}},(\mu+1)^{-\frac{1}{2}}h^{\frac{1}{2}})$-box in $\bR^3_x$. Without any loss of the generality one can assume that $A^0=A'=0$ in its center $z$.

Since both operators $E(\tau)$ and $((hD-A)_x\cdot \boldupsigma )E(\tau)$ have their operator norms bounded by $c$ in $\sL^2$ one can prove easily that operators
$\phi D^\alpha E(\tau)$ and
$\phi D^\alpha ((hD-A)_x\cdot \boldupsigma )E(\tau)$ have their operator norms bounded by $C \zeta^\alpha$ with
$\zeta_1 =\zeta_2= \beta^{\frac{1}{2}}h^{-\frac{1}{2}}$ and
$\zeta_3=(h^{-1}+\mu^{\frac{1}{2}}h^{-\frac{1}{2}})$ as $\alpha \in \{0,1\}^3$ and $\phi $ is supported in the mentioned cube.

Then operator norms of operators $\upgamma_x E(\tau)$ and
$\upgamma_x ((hD-A)_x\cdot \boldupsigma )E(\tau)$ from $\sL^2$ to $\bC^q$ do not exceed $C_0\zeta_1\zeta_3^{\frac{1}{2}}$ and therefore the same is true for adjoint operators; recall that $\upgamma_z$ is operator of restriction to $x=z$.

Since $E(\tau)^*=E(\tau)^2=E(\tau)$ we conclude that the left-hand expressions in (\ref{27-2-41}) and (\ref{27-2-42}) do not exceed $C\zeta_1^2\zeta_3$ which is exactly the right-hand expressions.
\end{proof}

Then from equation (\ref{book_new-26-2-14}) which remains valid (see Remark~\ref{rem-27-2-4}\ref{rem-27-2-4-iii}) we conclude that
\begin{gather}
\|\Delta A'\|_{\sL^\infty} \le C\kappa \bigl(\beta +\mu\bigl)
\bigl(1+ \mu^{\frac{1}{2}}h^{\frac{1}{2}}\bigr)\label{27-2-43}\\
\shortintertext{and therefore}
\|\partial ^2 A'\|_{\sL^\infty}\le C\kappa|\log \beta |
 \bigl(\beta +\mu\bigl)\bigl(1+ \mu^{\frac{1}{2}}h^{\frac{1}{2}}\bigr).
\label{27-2-44}
\end{gather}

Let (\ref{27-2-25})) be fulfilled. Then combining (\ref{27-2-44}) with (\ref{27-2-27})  and
$\|\partial A'\|_{\sL^\infty}\le \|\partial ^2 A'\|_{\sL^\infty}^{\frac{3}{5}} \cdot \|\partial A'\|^{\frac{2}{5}}$ we conclude that
\begin{gather*}
\mu \le C(\kappa|\log \beta|)^{\frac{3}{5}}
\bigl(\beta +\mu\bigl)^{\frac{3}{5}}
\bigl(1+ h^{\frac{1}{2}}\mu^{\frac{1}{2}}\bigr) ^{\frac{3}{5}}
\times \kappa ^{\frac{1}{5}} \beta ^{\frac{1}{5}}\\
\shortintertext{and then either}
1\le \mu h \le C (\kappa \beta h|\log \beta|) ^{\frac{6}{7}}
(\kappa \beta h^2) ^{\frac{2}{7}}
\end{gather*}
or $\mu h\le 1$. In the former case
\begin{gather}
\|\partial A'\|_{\sL^\infty}\le
C (\kappa \beta |\log \beta|) ^{\frac{8}{7}} h^{\frac{3}{7}}
\label{27-2-45}\\
\shortintertext{and}
\|\partial^2 A'\|_{\sL^\infty}\le
C (\kappa \beta |\log \beta|) ^{\frac{11}{7}}h^{\frac{5}{7}}.
\label{27-2-46}
\end{gather}
Observe that the right-hand expression of (\ref{27-2-45}) is less than $C\beta$ under assumption \begin{equation}
\kappa \beta h^2 |\log \beta|^K\le c
\tag*{$\textup{(\ref*{27-2-25})}^*$}\label{27-2-25-*}
\end{equation}
with sufficiently large $K$; however the right-hand expression of (\ref{27-2-46}) is not necessarily less than $C\beta$ under this assumption and we need more delicate arguments.

Without any loss of the generality we can assume that $\partial_3 A_3 (z)=0$ (we can reach it by a gauge transformation).

Then considering
$(\beta^{-\frac{1}{2}}h^{\frac{1}{2}}, \beta^{-\frac{1}{2}}h^{\frac{1}{2}},
\nu^{-\frac{1}{3}}h^{\frac{1}{3}})$-box in $\bR^3$ we can replace factor $(1+\mu^{\frac{1}{2}}h^{\frac{1}{2}})h^{-1}$ by
$(1+\nu^{\frac{1}{3}}h^{\frac{2}{3}})h^{-1}$ in all above estimates with
$\nu = \|\partial^2 A'\|_{\sL^\infty}$ and therefore (\ref{27-2-44}) is replaced by
\begin{gather}
\nu \le C\kappa |\log \beta|\beta \bigl(1+\nu^{\frac{1}{3}}h^{\frac{2}{3}}\bigr)
\notag\\
\intertext{and then under assumption \ref{27-2-25-*}}
\nu= \|\partial^2 A'\|_{\sL^\infty}\le
C\kappa |\log \beta|\beta\label{27-2-47}\\
\shortintertext{which implies}
\|\partial A'\|_{\sL^\infty}\le
C(\kappa |\log \beta|\beta)^{\frac{4}{5}}.\label{27-2-48}
\end{gather}

Thus we have proven
\begin{proposition}\label{prop-27-2-10}
Let $\beta h\ge 1$, $\kappa \le \kappa^*$ and \ref{27-2-25-*} be fulfilled. Then estimates \textup{(\ref{27-2-47})}--\textup{(\ref{27-2-48})} hold.
\end{proposition}

\chapter{Microlocal analysis unleashed: \texorpdfstring{$\beta h\lesssim 1$}{\textbeta h \textlesssim 1}}
\label{sect-27-3}

\section{Rough estimate to minimizer}
\label{sect-27-3-1}

Recall equation (\ref{book_new-26-2-14}) to a minimizer $A$ of $\E (A)$:
\begin{multline}
\frac{2}{\kappa h^2} \Delta A_j (x)  = \Phi_j(x)\Def\\
-\Re\tr \upsigma_j\Bigl( \bigl( (hD-A)_x \cdot \boldupsigma e (x,y,\tau)+
 e (x,y,\tau)\,^t(hD-A)_y \cdot \boldupsigma\bigr)\Bigr) |_{y=x}
\tag{\ref*{book_new-26-2-14}}\label{26-2-14x}
\end{multline}
where $e(x,y,\tau)$ is the Schwartz kernel of the spectral projector
$\uptheta (\tau-H_{A,V})$. Indeed, this equation should be replaced by
\begin{equation}
\frac{2}{\kappa h^2} \Delta \bigl(A_j (x) -A^0_j(x)\bigr)= \Phi_j(x)
\label{27-3-1}
\end{equation}
with $\Phi_j(x)$ defined above but since $\Delta A^0_j=0$ these two equations are equivalent. We assume in this section that $\beta h\lesssim 1$.

\begin{proposition}\label{prop-27-3-1}
Let $\beta h\lesssim 1$ and let
\begin{equation}
 \|\partial A'\|_{\sL^\infty}\le \mu \le h^{-1}.
\label{27-3-2}
\end{equation}
Then as $\theta \in [1,2]$
\begin{multline}
\| \Phi_j \|_{\sL^\infty}+ \|h\partial \Phi_j \|_{\sL^\infty} \\[3pt]
\le Ch^{-2} \bigl(1+ \beta^{\frac{3}{2}} h^{\frac{1}{2}} + |\log h|+
(\beta h)^{\frac{\theta-1}{\theta+1}}
\|\partial A'\|_{\sC^{\theta}}^{\frac{1}{\theta+1}}+
(\beta h)^{\frac{\theta-1}{\theta+1}}
\|\partial V\|_{\sC^{\theta}}^{\frac{1}{\theta+1}}
\bigr).
\label{27-3-3}
\end{multline}
\end{proposition}

\begin{proof}
\hypertarget{proof-27-3-1-i}{(i)} Assume first that $V\asymp 1$,
\begin{equation}
\beta h^{\frac{1}{3}}\lesssim 1\qquad \text{and}\qquad \mu = 1.
\label{27-3-4}
\end{equation}
Note that we need to consider only case $\beta\ge 2$ as otherwise estimate has been proven in Section~\ref{book_new-sect-26-2} (see Proposition~\ref{book_new-prop-26-2-16}). Then the contribution of the zone $\cZ'_\rho\Def \{|\xi_3- A_3(x)|\le \rho\}$ with $\rho\ge \rho_*\Def C_0\beta^{-1}$ to the Tauberian remainder with
$T=T_*\Def \epsilon \beta^{-1}$ does not exceed
\begin{equation}
Ch^{-2} \rho\bigl(\beta +
h^{\frac{1}{2}(\theta-1)}\|\partial A'\|^{\frac{1}{2}}_{\sC^{\theta}}+
h^{\frac{1}{2}(\theta-1)}\|\partial V\|^{\frac{1}{2}}_{\sC^{\theta}}
\bigr).
\label{27-3-5}
\end{equation}
Indeed, if $Q$ is $h$-pseudo-differential operator supported in this zone then exactly as in the proof of (\ref{book_new-26-2-47}) for $T\le T_*$
\begin{gather*}
|F_{t\to h^{-1}\tau}
\bar{\chi}_{T} (t) \Gamma_x \bigl((hD)^ k Q_x U_\varepsilon \bigr) |\le
C\rho h^{-2}\\
\shortintertext{and}
|F_{t\to h^{-1}\tau}
\bar{\chi}_{T} (t) \Gamma_x \bigl((hD)^ k Q_x (U-U_\varepsilon) \bigr) |\le
C\rho h^{-4}\vartheta T^2
\end{gather*}
where $U$, $U_\varepsilon$ are Schwartz kernels of $e^{-ih^{-1}tH_{A,V}}$ and
$e^{-ih^{-1}tH_{A_\varepsilon,V_\varepsilon}}$ and $\vartheta$ is an operator norm of perturbation $\bigl(H_{A_\varepsilon,V_\varepsilon}-H_{A,V}\bigr)$, $A_\varepsilon$ and $V_\varepsilon$ are $\varepsilon$-mollification of $A$ and $V$ respectively and $\varepsilon\ge h$; then
\begin{equation*}
|F_{t\to h^{-1}\tau} \bar{\chi}_{T} (t) \Gamma_x \bigl((hD)^ k Q_x U\bigr) |\le
C\rho (h^{-2}+h^{-4}\vartheta T^2)
\end{equation*}
and therefore the Tauberian error does not exceed
$C\rho (h^{-2}T^{-1}+h^{-4}\vartheta T)$.

Optimizing by $T\le T_*$ we get
$C\rho (h^{-2}\rho^{-1} + h^{-3}\vartheta ^{\frac{1}{2}})$ with $\varepsilon =h$ and $\vartheta= \varepsilon ^{\theta+1}\|\partial A'\|_{\sC^\vartheta}$ which is exactly the (\ref{27-3-5}).

Further, following arguments of Section~\ref{book_new-sect-26-2} we conclude that an error when we pass from the Tauberian expression to the Weyl expression does not exceed
\begin{equation}
C\rho h^{-2} \bigl(1+
h^{\frac{1}{2}(\theta-1)}\|\partial A'\|^{\frac{1}{2}}_{\sC^{\theta}}+
h^{\frac{1}{2}(\theta-1)}\|\partial V\|^{\frac{1}{2}}_{\sC^{\theta}}
\bigr).
\label{27-3-6}
\end{equation}

\medskip\noindent
\hypertarget{proof-27-3-1-ii}{(ii)} On the other hand, the contribution of zone
$\cZ_\rho =\{|\xi_3- A_3(x)|\asymp \rho \}$ with
$\rho \ge \rho_*= C_0\beta^{-1}$ to the Tauberian remainder with
$T=T^*\Def \epsilon \rho$ does not exceed
\begin{equation}
Ch^{-2} \bigl(1+
\rho^{\frac{1}{2}(1-\theta)} h^{\frac{1}{2}(\theta-1)}
\|\partial A'\|^{\frac{1}{2}}_{\sC^{\theta}}+
\rho^{\frac{1}{2}(1-\theta)} h^{\frac{1}{2}(\theta-1)}
\|\partial V\|^{\frac{1}{2}}_{\sC^{\theta}}
\bigr).
\label{27-3-7}
\end{equation}
Indeed, if $Q$ is $h$-pseudo-differential operator supported in this zone then for $T\le T^*$
\begin{gather*}
|F_{t\to h^{-1}\tau}
\bar{\chi}_{T} (t) \Gamma_x \bigl((hD)^ k Q_x U_\varepsilon \bigr) | \le
C\rho h^{-2},\\
\intertext{and for $T_*\le T\le T^*$}
|F_{t\to h^{-1}\tau} \chi_T(t) \Gamma_x \bigl((hD)^ k Q_x U_\varepsilon \bigr) |\le
C\rho h^{-2} (h/T\rho^2)^s,\\
\shortintertext{and then}
|F_{t\to h^{-1}\tau} \bigl(\bar{\chi}_{T} (t) - \bar{\chi}_{T_*} (t)\bigr)
\Gamma_x \bigl((hD)^ k Q_x U_\varepsilon\bigr) |\le
C\rho h^{-2} (\beta h/\rho^2)^s\\
\shortintertext{and}
|F_{t\to h^{-1}\tau}
\bar{\chi}_{T} (t) \Gamma_x \bigl((hD)^ k Q_x U_\varepsilon\bigr) |\le
C\rho h^{-2}
\end{gather*}
while approximation error is estimated in the same way as before but with $\varepsilon = h\rho^{-1}$ and thus $\vartheta $ acquires factor $\rho^{-1-\theta}$.

Then the Tauberian error is estimated and optimized by $T\le T^*$ and it does not exceed  (\ref{27-3-7}).

Following arguments of Section~\ref{book_new-sect-26-2} we conclude that an error when we pass from the Tauberian expression to the Weyl expression does not exceed (\ref{27-3-6}).

Then summation of the Tauberian error with respect to $\rho$ ranging from $\rho=\rho'$ to $\rho=1$ (where in what follows we use $\rho$ instead of $\rho'$ notation) returns
\begin{equation}
Ch^{-2} \bigl(1+|\log \rho|+
\rho^{-\frac{1}{2}(\theta-1)} h^{\frac{1}{2}(\theta-1)}
\|\partial A'\|^{\frac{1}{2}}_{\sC^{\theta}}+
\rho^{-\frac{1}{2}(\theta-1)}
h^{\frac{1}{2}(\theta-1)}\|\partial V\|^{\frac{1}{2}}_{\sC^{\theta}}
\bigr)
\label{27-3-8}
\end{equation}
and adding contribution of zone $\cZ'_\rho$ we conclude that the total Tauberian remainder does not exceed
\begin{equation}\\[3pt]
Ch^{-2} \bigl(\beta \rho +|\log \rho|+
\rho^{-\frac{1}{2}(\theta-1)} h^{\frac{1}{2}(\theta-1)}
\|\partial A'\|^{\frac{1}{2}}_{\sC^{\theta}}+
\rho^{-\frac{1}{2}(\theta-1)}
h^{\frac{1}{2}(\theta-1)}\|\partial V\|^{\frac{1}{2}}_{\sC^{\theta}}
\bigr).
\label{27-3-9}\\[3pt]
\end{equation}

Meanwhile summation of the Tauberian-to-Weyl error with respect to $\rho$ returns (\ref{27-3-9}) albeit without logarithmic term. Optimizing with respect to $\rho\ge \rho_*$ we arrive to
\begin{equation}
Ch^{-2} \bigl(1 +|\log h|+
(\beta h)^{\frac{\theta-1}{\theta+1}}
\|\partial A'\|_{\sC^{\theta}}^{\frac{1}{\theta+1}}+
(\beta h)^{\frac{\theta-1}{\theta+1}}
\|\partial V\|_{\sC^{\theta}}^{\frac{1}{\theta+1}}
\bigr).
\label{27-3-10}
\end{equation}

Furthermore, observe that
\begin{claim}\label{27-3-11}
If in $\epsilon$-vicinity of $x$ inequality $|\nabla V|\le \zeta$ holds (with $\zeta\ge |\log h|^{-1}$) then we can pick up
$T^*= \epsilon\min (\zeta ^{-1}\rho,1) $.
\end{claim}

Indeed, we can introduce
\begin{equation*}
p'_3 \Def \xi_3-A_3 -
\beta^{-1} \alpha_1 (\xi_1-A_1)-\beta^{-1} \alpha_2 (\xi_2-A_2)
\end{equation*}
such that $\{H, p'_3\} = V_{\xi_3} + O(\nu\beta^{-1})$ with
$\nu \Def\|\partial^2A\|_{\sC^\infty}$.

Therefore, in this case remainder estimate does not exceed
\begin{equation}
Ch^{-2} \bigl(1 +\zeta |\log h|+
(\beta h)^{\frac{\theta-1}{\theta+1}}
\|\partial A'\|_{\sC^{\theta}}^{\frac{1}{\theta+1}}+
(\beta h)^{\frac{\theta-1}{\theta+1}}
\|\partial V\|_{\sC^{\theta}}^{\frac{1}{\theta+1}}
\bigr).
\label{27-3-12}
\end{equation}

\medskip\noindent
\hypertarget{proof-27-3-1-iii}{(iii)}
Finally, observe that Weyl expression for $\Phi_j$ is just $0$. Therefore under assumption (\ref{27-3-4}) slightly improved estimate (\ref{27-3-3}) has been proven: $\| \Phi_j \|_{\sL^\infty}+ \|h\partial \Phi_j \|_{\sL^\infty}$ does not exceed expression (\ref{27-3-12}).

\medskip\noindent
\hypertarget{proof-27-3-1-iv}{(iv)}
To get rid off assumption (\ref{27-3-4}) we scale $x\mapsto x\gamma^{-1}$, $h\mapsto h\gamma^{-1}$, $\beta\mapsto \beta \gamma$ and pick up
$\gamma= \min \bigl((\beta h^{\frac{1}{3}})^{-\frac{3}{2}},\mu^{-1}\bigr)$ ; then $\beta h\mapsto \beta h$, and $Ch^{-2}\mapsto Ch^{-2}\gamma^{-1}=
C\beta^{\frac{3}{2}}h^{-\frac{3}{2}}+ C\mu h^{-2}$ (as we need to multiply by $\gamma^{-3}$) and  both $\|\partial A'\|_{\sC^{\theta}}^{\frac{1}{\theta+1}}$
and $\|\partial V\|_{\sC^{\theta}}^{\frac{1}{\theta+1}}$ acquire factor $\gamma$.

\medskip
Observing that we can take $\zeta = C\gamma$ and that factor $\gamma$ also pops up in all other terms (except $1$) in (\ref{27-3-12}) we arrive to estimate (\ref{27-3-3}).

\medskip
Furthermore, to get rid of assumption $V\asymp 1$ we also can scale with
$\gamma = \epsilon |V|+h^{\frac{2}{3}}$ and multiply operator by $\gamma^{-1}$; then $h\mapsto h\gamma^{-\frac{3}{2}}$,
$\beta\mapsto \beta \gamma^{\frac{1}{2}}$ and estimate (\ref{27-3-3}) does not deteriorate; we need to multiply by $\gamma^{\frac{1}{2}}$ which does not hurt.
\end{proof}

\begin{remark}\label{rem-27-3-2}
\begin{enumerate}[label=(\roman*), fullwidth]

\item\label{rem-27-3-2-i}
We can use $\theta'\ne \theta$ for norm of $V$;

\item\label{rem-27-3-2-ii}
If $V$ is smooth enough we can skip the related term (details later);

\item\label{rem-27-3-2-iii}
We can take $\theta=1$ but in this case factor $\rho^{-\frac{1}{2}(\theta-1)}h^{\frac{1}{2}(\theta-1)}$ in (\ref{27-3-8}) (i.e. after summation) and in (\ref{27-3-9}) is replaced by $|\log \rho|$; then taking into account \ref{rem-27-3-2-i} we replace (\ref{27-3-10}) by
\begin{equation}
Ch^{-2} \bigl(1 +|\log h|+
|\log h|\cdot
\|\partial A'\|_{\sC^{1}}^{\frac{1}{2}}+
(\beta h)^{\frac{\theta'-1}{\theta'+1}}
\|\partial V\|_{\sC^{\theta'}}^{\frac{1}{\theta'+1}}
\bigr)
\tag*{$\textup{(\ref*{27-3-10})}'$}\label{27-3-10-'}
\end{equation}
and similarly we deal with (\ref{27-3-12}) and (\ref{27-3-3}):
\begin{multline}
\| \Phi_j \|_{\sL^\infty}+ \|h\partial \Phi_j \|_{\sL^\infty} \le\\
Ch^{-2} \bigl(1+ \beta^{\frac{3}{2}} h^{\frac{1}{2}} + |\log h|+
|\log h| \|\partial A'\|_{\sC^{1}}^{\frac{1}{2}}+
(\beta h)^{\frac{\theta'-1}{\theta'+1}}
\|\partial V\|_{\sC^{\theta'}}^{\frac{1}{\theta'+1}}
\bigr).
\tag*{$\textup{(\ref*{27-3-3})}'$}\label{27-3-3-'}
\end{multline}

\item\label{rem-27-3-2-iv}
From the very beginning we could assume that $\mu \le \beta$; otherwise we could rescale as above with $\gamma=\beta^{-1}$ and apply arguments of Section~\ref{book_new-sect-26-2} simply ignoring external field.
\end{enumerate}
\end{remark}

\begin{corollary}\label{cor-27-3-3}
Let in the framework of Proposition~\ref{prop-27-3-1} $A'$ be a minimizer. Then as $\theta,\theta' \in [1,2]$
\begin{multline}
|\log h|^{-1}\|\partial A'\|_{\sC^{1}} +
h^{\theta-1}\|\partial A'\|_{\sC^{\theta}} \\
\le C\kappa \bigl(\beta^{\frac{3}{2}}h^{\frac{1}{2}} + |\log h| + \mu  +
(\beta h)^{\frac{1}{2}(\theta'-1)}\|V\|^{\frac{1}{2}}_{\sC^{\theta'+1}}
\bigr)\\
+C\kappa^2 |\log h| |\log h|^2+C \|\partial A'\|.
\label{27-3-13}
\end{multline}
\end{corollary}

\begin{proof}
Indeed, the left-hand expression of (\ref{27-3-13}) does not exceed
\begin{equation*}
\|\Delta A'\|_{\sL^\infty} +\| h\partial \Delta A'|_{\sL^\infty}+
C \|\partial A'\|
\end{equation*}
while for a minimizer
$\|\Delta A'\|_{\sL^\infty} +\| h\partial \Delta A'|_{\sL^\infty}$ does not exceed the right-hand expression of (\ref{27-3-3}) multiplied by $C\kappa h^2$.
\end{proof}

\section{Microlocal analysis}
\label{sect-27-3-2}

As long as $\beta \le C_0 h^{-\frac{1}{3}}$ we are rather happy with our result here but we want to improve it otherwise. First we will prove that singularities propagate along magnetic lines; however since we do not know self-generated magnetic field we just consider all possible lines which will be in the cone
$\{(x,y):\, |x'-y'|\le C_0\mu \beta^{-1} T\}$ where $1\le \mu \le \beta$.

\begin{proposition}\label{prop-27-3-4}
Assume that $\beta h\lesssim 1$,
\begin{gather}
\| V\|_{\sC^1(B(0,2))}\le C_0
\label{27-3-14}\\
\shortintertext{and}
\|\partial A'\|_{\sC(B(0,2))}\le \mu \qquad (1\le \mu \le \epsilon \beta)
\label{27-3-15}
\end{gather}
with sufficiently small constant $\epsilon >0$.

Let $U(x,y,t)$ be the Schwartz kernel of $e^{ih^{-1}tH_{A,V}}$. Then
\begin{enumerate}[label=(\roman*), fullwidth]

\item\label{prop-27-3-4-i}
For $T\asymp 1$ estimate
\begin{equation}
\|F_{t\to h^{-1}\tau} \bar{\chi}_T(t)
\psi_1 (x)\psi_2 (y) U\| \le C h^s
\label{27-3-16}
\end{equation}
holds for all $\psi_1,\psi_2 \in \sC_0^\infty(B(0,1))$, such that
$\dist(\supp \psi_1, \supp \psi_2)\ge C_0T$ and $\tau\le c_0$; here $\|.\|$ means an operator norm from $\sL^2$ to $\sL^2$ and $s$ is arbitrarily large;

\item\label{prop-27-3-4-ii}
For $\bar{\rho}\le \rho \lesssim 1$ with $\bar{\rho}\Def C_0\mu \beta^{-1}$ and
$T\asymp \rho $ estimate
\begin{multline}
\|F_{t\to h^{-1}\tau} \bar{\chi}_T(t)
\varphi_1 (\rho^{-1}p_{3x})\varphi_2 (\rho^{-1}p_{3y})\psi_1 (x)\psi_2 (y)U\| \\[3pt]
\le C \rho^{1-3s}h^s+
C\rho^{2-\theta}h^{\theta}\bigl(\3 A'\3_{\theta+1}+\3 V\3_{\theta+1}\bigr)
\label{27-3-17}
\end{multline}
holds for all all $\varphi_1,\varphi_2 \in \sC_0^\infty$,
$\psi_1,\psi_2 \in \sC_0^\infty(B(0,1))$, such that\newline
$\dist(\supp \varphi_1, \supp \varphi_1)\ge C_0$, and $\tau\le c_0$; here and below $p_j= hD_j-A_j$, $p^0_j= hD_j-A^0_j$;

\item\label{prop-27-3-4-iii}
For $\bar{\rho}\le \rho \lesssim 1$ and $T\asymp \rho \lesssim 1$ estimate
\begin{multline}
\|F_{t\to h^{-1}\tau} \bar{\chi}_T(t)
\varphi_1 (\rho^{-1}p_{3x})\varphi_2 (\rho^{-1}p_{3y}) \psi_1 (\gamma^{-1}x)\psi_2 (\gamma^{-1}y) U\| \\[3pt]
\le C\rho^{1-s}\gamma^{-s} h^s+
Ch^{\theta}\gamma \rho^{-\theta}\bigl(\3 A'\3_{\theta+1}+\3 V\3_{\theta+1}\bigr)
\label{27-3-18}
\end{multline}
holds  for all $\varphi_1,\varphi_2 \in \sC_0^\infty$,
$\psi_1,\psi_2 \in \sC_0^\infty$, such that
$\dist(\supp \psi_1, \supp \psi_2)\ge C_0$, $\gamma = \rho T\ge \beta^{-1}$, $\rho \gamma \ge h$ and $\tau\le c_0$.
\end{enumerate}
\end{proposition}

\begin{proof}
Statement~\ref{prop-27-3-4-i} claims that the general propagation speed is bounded by $C_0$, Statement~\ref{prop-27-3-4-ii} claims that the propagation on distances $\ge C_0\bar{\rho}$ speed with respect to $p_3$ is also bounded by $C_0$, and Statement~\ref{prop-27-3-4-iii} claims that on distances
$\ge C_0\bar{\rho}$ the propagation speed with respect to $x$ is bounded by
$C_0\rho $. Note that from Corollary~\ref{cor-27-3-3} we know that
$\mu \lesssim \beta$.

\medskip\noindent
\hypertarget{proof-27-3-4-a}{(a)} Proof follows the proof of Proposition~\ref{book_new-prop-26-2-11} in the framework of a strong magnetic field. Namely, proof of Statement~\ref{prop-27-3-4-i} is a straightforward repetition of the proof of Proposition~\ref{book_new-prop-26-2-11}\ref{book_new-prop-26-2-11-i}. Since here we do not apply at this stage operators $(hD_x)^\alpha$ and $(hD_y)^{\alpha'}$, no assumption to the smoothness of $A$ is needed.

\medskip\noindent
\hypertarget{proof-27-3-4-b}{(b)} Assume that $A_3\equiv 0$ (we will get rid off this assumption on the next step). After Statement \ref{prop-27-3-4-i} has been proven we rescale $t\mapsto t/T$,
$x_3\mapsto x_3/\gamma$ with $\gamma=\rho T$ (since $\varphi_l$ depend only on $\xi_3$ other coordinates are rather irrelevant),
$h\mapsto \hbar=h/(\rho \gamma)$, $T\mapsto 1$. Then we apply the arguments used in the proof of Proposition~\ref{book_new-prop-26-2-11}\ref{book_new-prop-26-2-11-ii} and conclude that the left-hand expression of (\ref{27-3-17}) does not exceed
\begin{equation*}
T\Bigl( \hbar^s + C\hbar^\theta \gamma^{\theta+1}
\bigl(\3 A'\3_{\theta+1}+\3 V\3_{\theta+1}\bigr)\Bigr)
\end{equation*}
where factor $T$ is due to the scaling in the Fourier transform and $\gamma^{\theta+1}$ is due to the scaling in $\3\cdot\3$-norms. Plugging
$\hbar, T$, and $\gamma=\rho^2$ we get the right-hand expression of (\ref{27-3-17}).

Then if $\psi_l$ depend only on $x_3$, $y_3$ we can follow the proof of Proposition~\ref{book_new-prop-26-2-11}\ref{book_new-prop-26-2-11-iii} and prove statement~\ref{prop-27-3-4-iii}.

\medskip\noindent
\hypertarget{proof-27-3-4-c}{(c)}
Let $A_3$ be not necessarily identically $0$. To consider $\psi_l$ depending only on $x_1$ or $x_2$ we should introduce (in the standard magnetic Schr\"odinger manner) $x'_1\Def x_1+ \beta^{-1}p^0_2$ or
$x'_2\Def x_2- \beta^{-1}p^0_1$ respectively; recall that $p^0_j= hD_j-A^0_j$ and $p_j= hD_j-A_j$, $j=1,2,3$.

Then $[p^0_1, p^0_2]=ih\beta^{-1}$, $[p^0_j, x_k]=-ih\updelta_{jk}$,
$[p^0_j, x'_k]=0$ and $[p_j, x'_k]=O(\beta^{-1}\mu h)$ (for any
$\mu: 1\le \mu \le \beta$) as $j=1,2,3$ and $k=1,2$. Now we can apply the same arguments as above as long as $\rho \ge \bar{\rho}$.

\medskip\noindent
\hypertarget{proof-27-3-4-d}{(d)}
Next we need to recover Statement~\ref{book_new-prop-26-2-11-ii} in the general case. Without any loss of the generality we may consider a vicinity of point $z$ where $A'(z)=0$ and also $\partial A'(z)=0$. Indeed we can achieve the former by the gauge transformation and the latter by a rotation of coordinates in which case increment of $p^0_3$ will be $O(\bar{\rho})$.

In this case we just repeat the same arguments of Part~\hyperlink{proof-27-3-4-b}{(b)} of our proof.

\medskip\noindent
\hypertarget{proof-27-3-4-e}{(e)}
Finally, the proof of Statement~\ref{book_new-prop-26-2-11-iii} as $\psi_l$ depend only on $x_3$ follows from Statement~\ref{book_new-prop-26-2-11-ii}.

\medskip
We leave all easy details to the reader.
\end{proof}

\begin{proposition}\label{prop-27-3-5}
Let $\beta h\lesssim 1$ and assumptions \textup{(\ref{27-3-14})} and \textup{(\ref{27-3-15})} be fulfilled. Then
\begin{enumerate}[label=(\roman*), fullwidth]

\item\label{prop-27-3-5-i}
For $h\le T\le 1$ estimate
\begin{equation}
| F_{t\to h^{-1}\tau} \bar{\chi}_T(t) p_x^\alpha p_y^{\alpha'} U | \le C T^{-s}h^{-3+s}
\label{27-3-19}
\end{equation}
holds for all $\alpha:|\alpha|\le 2$, $\alpha':|\alpha'|\le 2$, and all $x,y\in B(0,1)$, such that $|x-y|\ge C_0T$ and $\tau\le c_0$;

\item\label{prop-27-3-5-ii}
In the framework of Proposition~\ref{prop-27-3-4}\ref{prop-27-3-4-ii} the following estimate holds for all $\alpha:|\alpha|\le 2$,
$\alpha':|\alpha'|\le 2$, and $\tau\le c$:
\begin{multline}
|F_{t\to h^{-1}\tau} \bar{\chi}_T(t)
p_x^\alpha p_y^{\alpha'} \varphi_1 (\rho^{-1}p_{3x})\varphi_2 (\rho^{-1}p_{3y})\psi_1 (x)\psi_2 (y)U | \\[3pt]
\le C \rho h^{-1}\bigl(\beta h^{-1}+ \rho^2 h^{-2})
\Bigl(\rho^{1-3s}h^s+
\rho^{2-\theta}h^{\theta}\bigl(\3 A'\3_{\theta+1}+\3 V\3_{\theta+1}\bigr)\Bigr);
\label{27-3-20}
\end{multline}

\item\label{prop-27-3-5-iii}
In the framework of Proposition~\ref{prop-27-3-4}\ref{prop-27-3-4-iii} the following estimate holds for all $\alpha:|\alpha|\le 2$,
$\alpha':|\alpha'|\le 2$, and $\tau\le c$:
\begin{multline}
|F_{t\to h^{-1}\tau} \bar{\chi}_T(t)
p_x^\alpha p_y^{\alpha'}\varphi_1 (\rho^{-1}p_{3x})\varphi_2 (\rho^{-1}p_{3y}) \psi_1 (\gamma^{-1}x)\psi_2 (\gamma^{-1}y) U| \\[3pt]
\le C \rho h^{-1} \bigl(\beta h^{-1}+ \rho^2 h^{-2}\bigr)
\Bigl( \rho^{1-s}\gamma^{-s} h^s+
h^{\theta}\gamma \rho^{-\theta}\bigl(\3 A'\3_{\theta+1}+\3 V\3_{\theta+1}\bigr) \Bigr).
\label{27-3-21}
\end{multline}
\end{enumerate}
\end{proposition}

\begin{proof}
Observe that estimates (\ref{27-3-16})--(\ref{27-3-18}) hold if one applies
operator $p_x^\alpha p_y^{\alpha'}$ under the norm (this follows from equations for $U$ by $(x,t)$ and dual equations by $(y,t)$). Then estimates (\ref{27-3-19})--(\ref{27-3-21}) hold with $\alpha=\alpha'=0$.

Really, without any loss of the generality one can assume that $A'=0$ at some point of $\supp (\psi_1)$; then estimates (\ref{27-3-16})--(\ref{27-3-18}) hold if one applies operator $p_x^{0\alpha} p_y^{0\alpha'}$ instead. Then estimate (\ref{27-3-19}) holds with $\alpha=\alpha'=0$; further, estimates (\ref{27-3-20})--(\ref{27-3-21}) also hold with $\alpha=\alpha'=0$ if one applies an extra factor $\bar{\varphi}_1(\rho^{-1}p^0_{3x})\bar{\varphi}_2(\rho^{-1}p^0_{3y})$ under the norm (this follows from the properties of $p^0_j$, $j=1,2,3$, in particular, canonical form). However if $\bar{\varphi}_l=1$ in $\epsilon$-vicinity of $\supp(\varphi_l)$ then we can skip this factor.

Finally, appealing to equations for $U$ by $(x,t)$ and $(y,t)$ again we recover estimates (\ref{27-3-19})--(\ref{27-3-21}) with $|\alpha|\le 2$,
$|\alpha'|\le 2$.
\end{proof}

\begin{proposition}\label{prop-27-3-6}
Let $\beta h\lesssim 1$ and \textup{(\ref{27-3-14})} and \textup{(\ref{27-3-15})} be fulfilled. Let $z\in B(0,1)$.
Then:

\begin{enumerate}[fullwidth, label=(\roman*)]
\item\label{prop-27-3-6-i}
The following estimate
\begin{multline}
|F_{t\to h^{-1}\tau} \bar{\chi}_T(t)
p_x^\alpha p_y^{\alpha'}\varphi_1 (\rho^{-1}p_{3x})\varphi_2 (\rho^{-1}p_{3y}) U|_{x=y=z}| \\[3pt]
\le C \rho T h^{-1} \bigl(\beta h^{-1} +\rho^2 h^{-2}\bigr)
\label{27-3-22}
\end{multline}
holds for all $\alpha:|\alpha|\le 2$, $\alpha' :|\alpha'|\le 2$, and
$\tau\le c$;

\item\label{prop-27-3-6-ii}
Let $A_z(x)= A(z) + \langle x-z,\nabla _z\rangle A(z)$,
$V_z(x)= V(z) + \langle x-z,\nabla _z\rangle V(z)$ be linear approximations to $A$ and $V$ at $z$; let $H_{z}= H_{A_z,V_z}$, $U_z(x,y,t)$ be its Schwartz kernel. Then for $h^{1-\delta}\le T\le C_0$, $\rho\le C_0$ estimate
\begin{multline}
|F_{t\to h^{-1}\tau} \bar{\chi}_T(t)
p_x^\alpha p_y^{\alpha'}\varphi_1 (\rho^{-1}p_{3x})\varphi_2 (\rho^{-1}p_{3y}) (U-U_z)|_{x=y=z}| \\[3pt]
\le C \rho T ^2 h^{-2} \bigl(\beta h^{-1}+\rho^2 h^{-2}\bigr) \nu \gamma^2
\label{27-3-23}
\end{multline}
holds with
\begin{gather}
\gamma = \gamma(\rho,T)\Def C_0(\rho+T)T + C_0h\rho^{-1},
\label{27-3-24}\\
\shortintertext{and}
\nu= \bigl(\3 A'\3_{2}+\3 V\3_{2}\bigr)+\mu^2.
\label{27-3-25}
\end{gather}
\end{enumerate}
\end{proposition}

\begin{proof}
(i) Proof of Statement~\ref{prop-27-3-6-i} is easy and left to the reader.

\medskip\noindent
(ii) To prove Statement~\ref{prop-27-3-6-ii} observe that
\begin{multline}
e^{ith^{-1}H}= e^{ith^{-1}H_z} +
ih^{-1}\int_0^t e^{i(t-t')h^{-1}H }(H-H_z) e^{it'h^{-1}H_z}\,dt'=\\
e^{ith^{-1}H_z} +
\sum_{0\le k\le K} \underbracket{ih^{-1}\int_0^t e^{i(t-t')h^{-1}H }(H-H_z) \psi_k e^{it'h^{-1}H_z}\,dt'}
\label{27-3-26}
\end{multline}
where $\psi_0$ is a $\gamma$-admissible function supported in $B(z, 2\gamma)$ and $\psi_k$ are $\gamma_k$-admissible functions supported in
$B(z, \gamma_k)\setminus B(z, \frac{1}{2}\gamma_k)$ with $\gamma_k=2^k \gamma$. Plugging (\ref{27-3-26}) into the left-hand expression of (\ref{27-3-23}) we note that the first term is cancelled and we have the sum with respect to $k:0\le k\le K$ obtained from this expression when we replace $(U-U_z)$ by the Schwartz kernel of the selected above term.

Further, observe that the term with $k=0$ does not exceed the right-hand expression of (\ref{27-3-23}).

Furthermore, terms with $k:1\le k\le K$ do not exceed the right-hand expression of (\ref{27-3-21}) multiplied by $CTh^{-1}\min(\nu \gamma_k^2,1)$; indeed, we just replace $\rho$ by $T$ if needed. After summation with respect to
$k:0\le k\le K$ we get
\begin{equation*}
C\rho h^{-1}T \bigl(\beta h^{-1}+\rho^2 h^{-2}\bigr) \times \Bigl(\rho^{-s}\gamma^{-s}h^s \min (\nu \gamma^2,1)+
h^2 \rho^{-2} \nu \min (\nu,1)\Bigr)
\end{equation*}
which again does not exceed the right-hand expression of (\ref{27-3-23}).
\end{proof}

\begin{remark}\label{rem-27-3-7}
Actually Statement \ref{prop-27-3-6-ii} is better than Statement \ref{prop-27-3-6-i} only if $\nu \gamma^2 Th^{-1}\le 1$.
\end{remark}

\section{Advanced estimate to a minimizer}
\label{sect-27-3-3}

Now we are going to apply the results of the previous Subsection~\ref{sect-27-3-2} to the right-hand of (\ref{26-2-14x}).

\subsection{Tauberian estimate}
\label{sect-27-3-3-1}

Consider different zones (based on magnitude of $|p_3|$). Recall that $\bar{\rho}= \beta^{-1}$\,\footnote{\label{foot-27-5} As we assume that $\mu= 1$; otherwise $\bar{\rho}=\mu \beta^{-1}$.} and
$\rho^*=(\beta h)^{\frac{1}{2}}$.

\subsubsection{Zone $\{\rho'\lesssim |p_3|\lesssim \rho^*\}$.}\label{sect-27-3-3-1-1} Observe that Proposition~\ref{prop-27-3-6}\ref{prop-27-3-6-ii} implies that for
$\varphi_j\in \sC_0^\infty ([-2,-\frac{1}{2}]\cup[\frac{1}{2},2])$ estimate
\begin{gather}
|F_{t\to h^{-1}\tau} \bar{\chi}_T(t) p_x^\alpha p_y^{\alpha'}
\varphi_1 (\rho^{-1}p_{3x})\varphi_2 (\rho^{-1}p_{3y}) U|_{x=y=z}| \le C S(\rho, T)
\label{27-3-27}\\
\shortintertext{holds with}
S(\rho, T)= \bigl(\beta h^{-1}+\rho^2h^{-2}\bigr)
\bigl(\rho^{-1} +\rho  h^{-2}\nu \gamma^2 T^2\bigr)
\label{27-3-28}
\end{gather}
where $\gamma=\gamma(\rho,T)$ is defined by (\ref{27-3-24}).

Indeed, one can prove easily that
\begin{multline}
|F_{t\to h^{-1}\tau} \bar{\chi}_T(t)
p_x^\alpha p_y^{\alpha'}\varphi_1 (\rho^{-1}p_{3x})\varphi_2 (\rho^{-1}p_{3y}) U_z|_{x=y=z}|\\[3pt]
\le C(\beta h^{-1}\rho^{-1} +\rho^2 h^{-2}).
\label{27-3-29}
\end{multline}
Let us take $\alpha=\alpha'$, $\varphi_1=\varphi_2$ and $\rho_1=\rho_2$. Since in this case expression
\begin{equation}\\[2pt]
p_x^\alpha p_y^{\alpha'}\varphi_1 (\rho_1^{-1}p_{3x})\varphi_2 (\rho_2^{-1}p_{3y}) e(.,.,\tau)|_{x=y}
\label{27-3-30}
\end{equation}
is a monotone function with respect to $\tau$ then the standard Tauberian arguments (part I; we leave easy details to the reader) imply that
\begin{multline}
|p_x^\alpha p_y^{\alpha}\varphi_1 (\rho^{-1}p_{3x})\varphi_1 (\rho^{-1}p_{3y}) \bigl[e(.,.,\tau)-e(.,.,\tau')\bigr]|_{x=y}|\\[3pt]
\le C S(\rho,T) \bigl(T^{-1}+|\tau-\tau'|h^{-1}\bigr)
\label{27-3-31}
\end{multline}
for all $\tau'\le \tau \le c$ and therefore
\begin{multline}
|p_x^\alpha p_y^{\alpha'}
\varphi_1 (\rho_1^{-1}p_{3x}) \varphi_2 (\rho_2^{-1}p_{3y}) \bigl[e(.,.,\tau)-e(.,.,\tau')\bigr]|_{x=y}| \\[3pt]
\le C \bigl(S(\rho_1,T_1) S(\rho_2,T_2)\bigr)^{\frac{1}{2}}
\Bigl(T_1^{-\frac{1}{2}}T_2^{-\frac{1}{2}}+
|\tau-\tau'|^{\frac{1}{2}}\bigl(T_1^{-\frac{1}{2}}+T_2^{-\frac{1}{2}}\bigr)
h^{-\frac{1}{2}} +|\tau-\tau'|h^{-1}\Bigr).
\label{27-3-32}
\end{multline}

Then the standard Tauberian arguments (part II, with the minor modifications; again we leave easy details to the reader) imply that expression (\ref{27-3-30})
is given by the standard Tauberian formula with an error not exceeding the right-hand expression of (\ref{27-3-32}) with $|\tau-\tau'|$ replaced by $hT^{-1}$ which is
\begin{equation}
C \bigl(S(\rho_1,T_1) S(\rho_2,T_2)\bigr)^{\frac{1}{2}}
\Bigl(T_1^{-\frac{1}{2}}T_2^{-\frac{1}{2}}+
\bigl(T_1^{-\frac{1}{2}}+T_2^{-\frac{1}{2}}\bigr)
T^{-\frac{1}{2}} +T^{-1}\Bigr).
\label{27-3-33}
\end{equation}
Note that for $T\ge \max(T_1,T_2)$ the second factor in (\ref{27-3-33}) is $\asymp T_1^{-\frac{1}{2}}T_2^{-\frac{1}{2}}$.

In other words, contribution of the pair $(\rho_1,\rho_2)$ to the Tauberian error does not exceed a square root of
$S(\rho_1,T_1) T_1^{-1} \times S(\rho_2,T_2)T_2^{-1}$ with
\begin{multline}
S(\rho,T) T^{-1}=\beta h^{-1}
\Bigl(\rho ^{-1}T^{-1} +
h^{-2}\rho \nu T \bigl(T^4 + \rho ^2 T^2 + h^2\rho ^{-2}\bigr)\Bigr)\\
\shoveright{\asymp
\beta h^{-1}
\Bigl(\rho^{-1}T^{-1} + h^{-2} \nu \rho^3 T^3 +
h^{-2}\nu \rho T^5 +\nu \rho^{-1} T \Bigr)}\\
\text{as\ \ } \bar{\rho}\le \rho \le \rho^*.
\label{27-3-34}
\end{multline}
Minimizing this expression by $h \le T\lesssim 1$ we get
\begin{multline*}
\beta h^{-1}
\Bigl(\rho^{-\frac{3}{4}}(h^{-2} \nu \rho^3)^{\frac{1}{4}}+
\rho^{-\frac{5}{6}}(h^{-2} \nu \rho )^{\frac{1}{6}}+
\rho^{-1} \nu^{\frac{1}{2}}\Bigr)\\
\asymp
\beta h^{-1} \Bigl(h^{-\frac{1}{2}}\nu^{\frac{1}{4}}+
h^{-\frac{1}{3}} \nu ^{\frac{1}{6}}\rho^{-\frac{2}{3}} +
\nu ^{\frac{1}{2}}\rho^{-1}\Bigr);
\end{multline*}
summation by $\rho\in [\rho',\rho^*]$ returns
\begin{equation}
\beta h^{-1} \Bigl(h^{-\frac{1}{2}}\nu^{\frac{1}{4}}|\log h|+
h^{-\frac{1}{3}} \nu ^{\frac{1}{6}}\rho^{-\frac{2}{3}} +
\nu ^{\frac{1}{2}}\rho^{-1}\Bigr)
\label{27-3-35}
\end{equation}
with $\rho=\rho'$ to be selected later.

\subsubsection{Zone $\{\rho^*\lesssim |p_3|\lesssim 1\}$.}\label{sect-27-3-3-1-2}
Further, we claim that

\begin{claim}\label{27-3-36}
For $h^{\frac{1}{3}}\le \rho \le C_0$, $h\le T\le \epsilon_0\rho$ we can take $\gamma =h\rho^{-1}$.
\end{claim}
Indeed, observe that if we use $\varepsilon$-approximation with
$\varepsilon = (\rho^{-1}h)^{1-\delta}$ then the contribution of time intervals $\{t:\,T_*\le |t|\le T^*\}$ with $T_*=\rho^{-1}(\rho^{-1}h)^{1-\delta}$, $T^*=\epsilon_0\rho$ is negligible and the transition from
$\varepsilon = (\rho^{-1}h)^{1-\delta}$ to $\varepsilon = (\rho^{-1}h)$ is done like in the previous Chapter~\ref{book_new-sect-26}. Again we leave easy details to the reader.

Then
\begin{gather}
S(\rho,T) T^{-1} \asymp \rho h^{-2}\bigl(T^{-1} + \nu T \bigr);
\label{27-3-37}\\
\intertext{minimizing by $T:\, h\le T\le \epsilon \rho $ we get}
\rho h^{-2}\bigl(  \nu ^{\frac{1}{2}} + \rho^{-1} + \nu h\bigr);
\notag
\end{gather}
then summation by $\rho\in [\rho^*, C_0]$ returns
\begin{equation}
h^{-2}\bigl(  \nu ^{\frac{1}{2}} + (1 + \nu h)|\log h|\bigr).
\label{27-3-38}
\end{equation}

Observe that $\rho^*= (\beta h)^{\frac{1}{2}}\ge h^{\frac{1}{3}}$ as
$\beta\ge h^{-\frac{1}{3}}$.

\subsubsection{Zone $\{ |p_3|\lesssim \rho'\}$.}\label{sect-27-3-3-1-3}
Finally, the remaining zone $\{|p_3|\lesssim \rho'\}$ is covered by a single element $\varphi(\rho^{-1}p_3)$ with $\varphi \in \sC_0^\infty ([-2,2])$, $\rho=\rho'$.

Then instead of minimized $S(\rho,T) T^{-1}$ we can take $\rho \beta h^{-2}$ which should be added to the sum of expressions (\ref{27-3-38}) and (\ref{27-3-35}) which estimate contributions of two other zones resulting in
\begin{multline}
h^{-2}\bigl(  \nu ^{\frac{1}{2}} + (1 + \nu h)|\log h|\bigr) \\
+ \beta h^{-1} \Bigl[h^{-\frac{1}{2}}\nu^{\frac{1}{4}}|\log h|+
\underbracket{h^{-\frac{1}{3}} \nu ^{\frac{1}{6}}\rho^{-\frac{2}{3}} +
\nu ^{\frac{1}{2}}\rho^{-1}+ \rho h^{-1}}\Bigr].
\label{27-3-39}
\end{multline}

Obviously the second term in (\ref{27-3-39}) should be minimized by $\rho=\rho'\in [\bar{\rho}, \rho^*]$ resulting in
\begin{multline*}
\beta h^{-1}\Bigl[h^{-\frac{1}{2}}\nu ^{\frac{1}{4}}|\log h|+
h^{-\frac{2}{5}}(h^{-\frac{1}{3}}\nu ^{\frac{1}{6}})^{\frac{3}{5}} +
h^{-\frac{1}{2}}(\nu ^{\frac{1}{2}})^{\frac{1}{2}}\Bigr]\\
\asymp \beta h^{-1}\Bigl[h^{-\frac{1}{2}}\nu ^{\frac{1}{4}}|\log h|+
h^{-\frac{3}{5}} \nu ^{\frac{1}{10}} +
h^{-\frac{1}{2}} \nu ^{\frac{1}{4}})\Bigr];
\end{multline*}
two terms arising when we set $\rho= \rho^*$ in the terms with negative power of $\rho$ and one term arising when we set $\rho = \bar{\rho}$ in the term with positive power of $\rho$ are absorbed by other terms in (\ref{27-3-39}) which becomes
\begin{multline}
h^{-2}\bigl(  \nu ^{\frac{1}{2}} + (1 + \nu h)|\log h|\bigr) \\
+ \beta h^{-1}\Bigl[h^{-\frac{1}{2}}\nu ^{\frac{1}{4}}|\log h|+
h^{-\frac{3}{5}} \nu ^{\frac{1}{10}} +
h^{-\frac{1}{2}} \nu ^{\frac{1}{4}}\Bigr].
\label{27-3-40}
\end{multline}

This is an estimate for the whole Tauberian error (with variable $T=T(\rho)$).

\subsection{Calculating Tauberian expression}
\label{sect-27-3-3-2}

Now we need to consider the Tauberian expression for
$p^\alpha_{x}p^{\alpha'}_{y}e(.,.,0)|_{x=y=z}$ and estimate an error made when we replace it by the Tauberian expression for $p^\alpha_{x}p^{\alpha'}_{y}e_z(.,.,0)|_{x=y=z}$; we will call it the \emph{second error\/} in contrast to the first (Tauberian) error. Note that we are interested only in the case $|\alpha|+|\alpha'|=1$.

\medskip
Let us again consider contribution of pair $(\rho_1,\rho_2)$. First, observe that for $\rho_1\asymp \rho_2$ this error does not exceed
$CS(\rho_2,T)T^{-1}$ due to our standard arguments and therefore we get for such pairs the same contribution to the total error as we already got for the Tauberian error.

\medskip
Second, consider pairs with $\rho_1 \gg \rho_2$ and in this case redoing previous arguments we observe that the contribution to the first error does not exceed $C S(\rho_1,T)^{\frac{1}{2}} S(\rho_2,T)^{\frac{1}{2}} T^{-1}$ and the contribution to the second error does not exceed
\begin{equation}
C\beta h^{-1} \times \rho_1^{\frac{1}{2}} \rho_2^{\frac{1}{2}} h^{-2}\nu T^3 (T+\rho_2)^2
\label{27-3-41}
\end{equation}
where the first term which was $C\beta h^{-1}\rho^{-1}T^{-1}$ in the former case $\rho_1\asymp \rho_2$ simply disappear. Indeed, it appears only due to the contribution of the time interval $\{|t|\le \rho\}$ where we should take $\rho=\max(\rho_1,\rho_2)=\rho_1$ and estimate an error due to propagation of singularities.

Similarly, the second term leading to expression (\ref{27-3-41}) would also disappear unless $\rho_1 \lesssim T$ again due to the propagation of singularities. Therefore the combined contribution of any pair to both errors does not exceed
\begin{multline}
\ \Bigl(\rho_1^{-1}T^{-1} + \rho_1^3 h^{-2} \nu  T^3 +
\rho_1h^{-2}\nu  T^5 + \rho_1^{-1} \nu T \Bigr)^{\frac{1}{2}}\\
\times
\Bigl(\rho_2^{-1}T^{-1} + \rho_2^3 h^{-2} \nu  T^3 +
\rho_2 h^{-2}\nu  T^5 + \rho_2^{-1} \nu T \Bigr)^{\frac{1}{2}} +
\rho_1^{\frac{1}{2}} \rho_2^{\frac{1}{2}} h^{-2}\nu T^5
\label{27-3-42}
\end{multline}
multiplied by $C\beta h^{-1}$ as we consider at this moment
$\rho'\le \rho_2 \ll \rho_1 \le \rho^*$ and other cases
($\rho_2\le \rho' \le \rho_1\le \rho^*$; $\rho_2\le \rho' \le \rho^*\le \rho_1$; $\rho'\le \rho_2\le \rho^* \le \rho_1$; and
$\rho^*\le \rho_2\ll \rho_1$) are easier and left to the reader.

Opening parenthesis in (\ref{27-3-42}) and eliminating all smaller terms we arrive to
\begin{multline*}
\rho_1^{-\frac{1}{2}} \rho_2^{-\frac{1}{2}} T^{-1}
+ \rho_1^{\frac{1}{2}} \rho_2^{\frac{1}{2}} h^{-2} \nu T^5
+\rho_1^{\frac{3}{2}}\rho_2^{\frac{1}{2}} h^{-2} \nu T^4\\[3pt]
+ \bigl(\rho_1^{\frac{3}{2}}\rho_2^{\frac{3}{2}} h^{-2} \nu
+ \rho_1^{\frac{1}{2}} \rho_2^{-\frac{1}{2}} h^{-1} \nu \bigr) T^3
+\bigl(\rho_1^{\frac{1}{2}}\rho_2^{-\frac{1}{2}} (h^{-2} \nu )^{\frac{1}{2}} +\rho_1^{\frac{3}{2}}\rho_2^{-\frac{1}{2}} h^{-1} \nu \bigr) T^2\\[3pt]
+\bigl(\rho_1^{-\frac{1}{2}} \rho_2^{-\frac{1}{2}}\nu
+\rho_1^{\frac{3}{2}}\rho_2^{-\frac{1}{2}} (h^{-2} \nu )^{\frac{1}{2}} \bigr) T+
\rho_1^{-\frac{1}{2}} \rho_2^{-\frac{1}{2}} \nu^{\frac{1}{2}};
\end{multline*}
minimizing by $T$ we get
\begin{multline*}
\rho_1^{-\frac{1}{3}}\rho_2^{-\frac{1}{3}}(h^{-2}\nu)^{\frac{1}{6}}+
\rho_1^{-\frac{1}{10}}\rho_2^{-\frac{3}{10}}(h^{-2}\nu)^{\frac{1}{5}}+
(h^{-2}\nu)^{\frac{1}{4}}+
\rho_1^{-\frac{1}{4}}\rho_2^{-\frac{1}{2}}(h^{-1}\nu)^{\frac{1}{4}}\\[3pt]
+
\rho_1^{-\frac{1}{6}}\rho_2^{-\frac{1}{2}}(h^{-1}\nu)^{\frac{1}{6}}+
\rho_1^{-\frac{1}{2}}\rho_2^{-\frac{1}{2}}\nu^{\frac{1}{2}}+
\rho_1^{\frac{1}{2}}\rho_2^{-\frac{1}{2}}(h^{-2}\nu)^{\frac{1}{4}}+
\underbracket{\rho_1^{-\frac{1}{2}} \rho_2^{-\frac{1}{2}} \nu^{\frac{1}{2}}}.
\end{multline*}

Observe, that only the last term has $\rho_1$ in the positive degree. Also observe, that the optimal $T=T(\rho)$ in the Tauberian error is a decreasing function of $\rho$, so $T_1\le T_2$ where $T_j\Def T(\rho_j)$; therefore we consider the Tauberian expression for $T \le T_2$ and thus for
$\rho_1\le T \lesssim T_2$.

Therefore $T_2$ must be an upper bound for $\rho_1$ and therefore summation by $\rho_1: \rho_2\le \rho_1 \le T_2$ results in all the terms with negative power of $\rho_1$ in the value as $\rho_1=\rho_2$ and in the exceptional (last) term with $\rho_1=T_2$:
\begin{multline}
\rho_2^{-\frac{2}{3}}(h^{-2}\nu)^{\frac{1}{6}}+
\rho_2^{-\frac{2}{5}}(h^{-2}\nu)^{\frac{1}{5}}+
(h^{-2}\nu)^{\frac{1}{4}}|\log \rho_2(\rho_2^{-2}h^2\nu^{-1})^{-\frac{1}{6}}|\\[3pt]
+
\rho_2^{-\frac{3}{4}}(h^{-1}\nu)^{\frac{1}{4}}
+\rho_2^{-1}\nu^{\frac{1}{2}}+
(\rho_2^{-2}h^2\nu^{-1})^{\frac{1}{12}}
\rho_2^{-\frac{1}{2}}(h^{-2}\nu)^{\frac{1}{4}}
\label{27-3-43}
\end{multline}
(where we used inequality $T_2 \le (\rho_2^{-2}h^2\nu^{-1})^{\frac{1}{6}}$) with the last term equal to the first one.

Then summation by $\rho_2\ge \rho '$ results in the same expression (\ref{27-3-43}) calculated for $\rho_2=\rho'$; adding as usual $\rho' h^{-1}$ (as $\rho'\beta h^{-2}$ estimates the contribution of zone
$\{\rho_2\le \rho'\}$) and minimizing by $\rho'\ge \bar{\rho}$ we get after multiplying by $\beta h^{-1}$ and adding contributions of all other zones and also Tauberian estimate (\ref{27-3-40})
\begin{multline}
h^{-2}\bigl(  \nu ^{\frac{1}{2}} + (\mu + \nu h)|\log h|\bigr) \\
+ \beta h^{-1}\Bigl[h^{-\frac{1}{2}}\nu ^{\frac{1}{4}}|\log h|+
h^{-\frac{3}{5}} \nu ^{\frac{1}{10}} + h^{-\frac{4}{7}}\nu^{\frac{1}{7}}+
h^{-\frac{1}{2}} \nu ^{\frac{1}{4}}\Bigr].
\label{27-3-44}
\end{multline}
Recall that we derived estimate for the difference between
$p_x^\alpha p_y^{\alpha'}e(.,.,0)|_{x=y=z}$ and
$p_x^\alpha p_y^{\alpha'}e_z(.,.,0)|_{x=y=z}$ and thus as $\mu=1$ we arrive to Statement~\ref{prop-27-3-8-i} of Proposition~\ref{prop-27-3-8} below as $\mu=1$.

Observe however that in virtue of Subsection~\ref{sect-27-3-1} the same estimate holds as $\beta \le h^{-\frac{1}{3}}$. Then as $1\le \mu \le \beta$ one can scale $x\mapsto \mu x$, $h\mapsto \mu h$, $\nu\mapsto \mu^2 \nu$,
$\kappa\mapsto \mu\kappa$ and we arrive to the same statement without assumption $\mu=1$.

Furthermore, in virtue of Propositions~\ref{prop-27-A-4} and~\ref{prop-27-A-5} expression $|p_x^\alpha p_y^{\alpha'}e_z(.,.,0)|_{x=y=z}|$  does not exceed $C\beta^{\frac{1}{2}}h^{-2}\|\partial V\| _{\sL^\infty}$ as $|\alpha|+|\alpha'|=1$\,\footnote{\label{foot-27-6} Actually Proposition~\ref{prop-27-A-4} provides better estimate for
$\|\partial V\| _{\sL^\infty}\le \beta^2 h$.}. Therefore we arrive to Statement~\ref{prop-27-3-8-ii} below:

\begin{proposition}\label{prop-27-3-8}
Let $\beta \le h^{-1}$ and \textup{(\ref{27-3-14})} and \textup{(\ref{27-3-15})} be fulfilled. Then as $|\alpha|\le 2$, $|\alpha'|\le 2$

\begin{enumerate}[label=(\roman*), fullwidth]

\item\label{prop-27-3-8-i}
$|p_x^\alpha p_y^{\alpha'}\bigl[e (.,.,0)-e_z(.,.,0)\bigr]|_{x=y=z}|$ does not exceed expression \textup{(\ref{27-3-44})};

\item\label{prop-27-3-8-ii}
Consider $|\alpha|+|\alpha'|=1$; then $|p_x^\alpha p_y^{\alpha'}e (.,.,0)|_{x=y=z}|$ does not exceed expression \textup{(\ref{27-3-44})} plus
$C \omega h^{-2}$ with
\begin{equation}
\omega= \left\{\begin{aligned}
&1\qquad &&\text{as\ \ } \beta \le h^{-\frac{1}{3}},\\
&\beta^{\frac{3}{2}}h^{\frac{1}{2}} \qquad
&&\text{as\ \ } h^{-\frac{1}{3}} \le \beta \le h^{-\frac{1}{2}},\\
&\beta^{\frac{1}{2}} &&\text{as\ \ } h^{-\frac{1}{2}} \le \beta \le h^{-1}.
\end{aligned}\right.
\label{27-3-45}
\end{equation}
\end{enumerate}
\end{proposition}

\begin{remark}\label{rem-27-3-9}
\begin{enumerate}[fullwidth, label=(\roman*)]
\item\label{rem-27-3-9-i}
Observe that as $\beta \le h^{-\frac{1}{2}}$ we got any improvement over results of Subsection~\ref{sect-27-3-1};

\item\label{rem-27-3-9-ii}
One can replace $\mu$ in the definition of $\bar{\rho}$ by $\nu^{\frac{1}{2}}$. Indeed, we can assume that $\partial A'(z)=0$. Then $\gamma$-vicinity of $z$ we have $\mu=O(\nu \gamma)$ and scaling we should be concerned only abut this vicinity. We select $\gamma=\nu^{-\frac{1}{2}}$.
\end{enumerate}
\end{remark}

\subsection{Estimating \texorpdfstring{$|\partial^2 A'|$}{|\textpartial\texttwosuperior A'|}}
\label{sect-27-3-3-3}

Recall that if $A'$ is a minimizer then it must satisfy (\ref{26-2-14x}) and then as $h^{-\frac{1}{2}}\le \beta \le h^{-1}$ due to Proposition~\ref{prop-27-3-8}\ref{prop-27-3-8-ii} $\|\Delta A'\|_{\sL^\infty}$ does not exceed
$C\kappa h^2\bigl(\textup{(\ref{27-3-44})}+ \beta^{\frac{1}{2}}h^{-2}\bigr)$ and then $\|\partial^2 A'\|_{\sL^\infty}$ must not exceed this expression multiplied by $C|\log h|$ plus
$\|\partial A'\|_{\sL^\infty}$\,\footnote{\label{foot-27-7} Which can be replaced by a different norm, say, $\|\partial A'\|$.}:
\begin{multline}
\|\partial^2 A'\|_{\sL^\infty}
\le C\kappa |\log h|
\bigl(  \nu ^{\frac{1}{2}} + (\mu + \nu h)|\log h|\bigr) \\[3pt]
+ C\kappa \beta h|\log h|\bigl(h^{-\frac{3}{5}}\nu^{\frac{1}{10}}+
h^{-\frac{4}{7}}\nu^{\frac{1}{7}}+
h^{-\frac{1}{2}}\nu^{\frac{1}{4}}|\log h|^2\bigr)+
C \kappa |\log h| \beta^{\frac{1}{2}}  \|\partial V\| _{\sL^\infty}\\[3pt]
+ C\|\partial A'\|_{\sL^\infty}.
\label{27-3-46}
\end{multline}
Also recall that we can define
$\nu\Def \max\bigl(\|\partial^2 A'\|_{\sL^\infty},\,1\bigr)$; then we arrive to

\begin{proposition}\label{prop-27-3-10}
Let $1\le \beta \lesssim h^{-1}$ and \textup{(\ref{27-3-14})} be fulfilled. Let $A'$ be a minimizer satisfying \textup{(\ref{27-3-15})}.

Then one of the following two cases holds: \underline{either}
\begin{multline}
\|\partial^2 A'\|_{\sL^\infty}
\le C\mu (\kappa |\log h|^2 +1) \\[3pt]
+ C(\kappa \beta |\log h| )^{\frac{10}{9}} h^{\frac{4}{9}}+
C(\kappa \beta |\log h|)^{\frac{7}{6}} h^{\frac{1}{2}}+
C(\kappa \beta |\log h|^3 )^{\frac{4}{3}} h^{\frac{2}{3}}\\[3pt]
+ C\kappa |\log h| \omega
+ C\|\partial A'\|_{\sL^\infty}
\label{27-3-47}
\end{multline}
with the right-hand expressions $\ge C$ \underline{or}
\begin{multline}
\|\partial^2 A'\|_{\sL^\infty}
\le C \mu ( \kappa|\log h|^2 +1)\\[3pt] +
C \kappa \beta |\log h| h^{\frac{2}{5}}+
C\kappa \beta |\log h| h^{\frac{3}{7}}+
C \kappa \beta |\log h|^3 h^{\frac{1}{2}}\\[3pt]
+ C\kappa |\log h| \omega
+ C\|\partial A'\|_{\sL^\infty}
\label{27-3-48}
\end{multline}
with the right-hand expression $\le C$. Recall that $\omega$ is defined by \textup{(\ref{27-3-45})}.
\end{proposition}

\begin{proof}
Indeed, if $\nu \ge C$ we have
\begin{multline*}
\nu \le C\kappa \mu |\log h| +
C(\kappa \beta |\log h| h^{\frac{2}{5}})^{\frac{10}{9}}+
C(\kappa \beta |\log h| h^{\frac{3}{7}})^{\frac{7}{6}}+
C(\kappa \beta |\log h|^3 h^{\frac{1}{2}})^{\frac{4}{3}}\\[3pt]
+
C\kappa \beta^{\frac{1}{2}}  \|\partial V\| _{\sL^\infty}+
C\|\partial A'\|_{\sL^\infty}
\end{multline*}
which leads to (\ref{27-3-47}); if $\nu \asymp 1$ we have (\ref{27-3-48}).
\end{proof}

\begin{remark}\label{rem-27-3-11}
\begin{enumerate}[label=(\roman*), fullwidth]

\item\label{rem-27-3-11-i}
Observe that the right-hand expressions of (\ref{27-3-47}) and (\ref{27-3-48}) are either $\lesssim 1$ or $\gtrsim 1$ simultaneously;

\item\label{rem-27-3-11-ii}
The second term in the right-hand expression of (\ref{27-3-47}) (i.e. with the power $\frac{10}{9}$) is always greater than the third and the fourth terms unless $\kappa \beta h \ge |\log h|^{-K}$). Because of this we just take power $K$ of $|\log h|$ in this term and skip two other terms. One can find easily that $K=4$ is sufficient;

\item\label{rem-27-3-11-iii}
The second term in the right-hand expression of (\ref{27-3-47}) is less than the last one as $\beta \le h^{-\frac{8}{11}}(\kappa |\log h|)^{-\frac{20}{11}}$;

\item\label{rem-27-3-11-iv}
Obviously, in (\ref{27-3-48}) we can take $\mu=1$; however we are missing estimate of $\mu $ in (\ref{27-3-47}). Sure we know that $\mu \le C\nu$ but we will be able to do a better work after we estimate $\|\partial A'\|^2$ in Subsubsection~\ref{sect-27-3-5-3}.3.
\end{enumerate}
\end{remark}

\section{Trace term asymptotics}
\label{sect-27-3-4}

\subsection{General microlocal arguments}
\label{sect-27-3-4-1}

Now let us consider the trace term. We are not assuming anymore that $A'$ is a minimizer but that it satisfies
\begin{phantomequation}\label{27-3-49}\end{phantomequation}
\begin{equation}
\|\partial A'\|_{\sL^\infty}\le \mu,\qquad
\|\partial^2 A'\|_{\sL^\infty}\le \nu \qquad
\text{with\ \ } 1\le \mu \le \nu\le \epsilon.
\tag*{$\textup{(\ref*{27-3-49})}_{1,2}$}\label{27-3-49-*}
\end{equation}
We assume that $V\in \sC^2$ uniformly. Later we will impose on $V$ different non-degeneracy assumptions; from now on small constant $\epsilon>0$ in conditions (\ref{27-3-15}), \ref{27-3-49-*} depends also on the constants in the non-degeneracy assumption.

Let us introduce scaling function
\begin{gather}
\ell (x) \Def
\epsilon_0 \bigl( \min_j |V- 2j \beta h| +|\partial V|^2 \bigr)^{\frac{1}{2}}+
\bar{\ell}
\label{27-3-50}\\
\shortintertext{with}
\bar{\ell} \Def
C_0\max \bigl(\nu \beta^{-1},\, \mu \beta^{-1},\, h^{\frac{1}{2}}\bigr).
\label{27-3-51}
\end{gather}

We need the following

\begin{proposition}\label{prop-27-3-12}
Let $\beta h\lesssim 1$ and \ref{27-3-49-*} be fulfilled. Consider $(\gamma,\rho)$-element with respect to $(x,p_3)$ with $\gamma\rho \ge h$, $\gamma\le \max(\ell,\rho)$ and
\begin{equation}
\rho \ge \bar{\rho} \Def
C_0\max \bigl(\mu \beta^{-1},\, h^{\frac{1}{2}}\bigr).
\label{27-3-52}
\end{equation}

Then as
\begin{equation}
T_*\Def h\rho^{-2}\le T\le T^*\Def \epsilon_0\min (1, \rho \ell^{-1})
\label{27-3-53}
\end{equation}
for $(\gamma,\rho)$-element
$\{(x,\xi_3):\, x\in B(z,\gamma), |\xi_3 - A_3(z)|\asymp \rho\}$ estimate
\begin{multline}
|F_{t\to h^{-1}\tau} \bar{\chi}_T(t)
\Gamma \bigl( p_x^\alpha p_y^{\alpha'}\varphi_1 (\rho^{-1}p_{3x})\varphi_2 (\rho^{-1}p_{3y}) \psi_1 (\gamma^{-1}x) \psi_2 (\gamma^{-1}y) U \bigr)| \\
\le CS (\rho, T)\gamma^3
\label{27-3-54}
\end{multline}
holds with
\begin{equation}
S(\rho, T)= \bigl(\beta h^{-1}+\rho^2 h^{-2}\bigr)
\bigl(\rho^{-1} +\rho ^{-3}  h\nu^2 \varepsilon^4 T^3\bigr)
\label{27-3-55}
\end{equation}
where $\varepsilon=h\rho^{-1}$.
\end{proposition}

Observe that we redefined $\bar{\rho}$ possibly increasing it.

\begin{proof}
The proof is similar to one of Theorem~\ref{book_new-thm-26-2-17} and is based on $h\rho^{-1}$-approximation. Note first that the propagation speed with respect to $x_3$ is $\asymp \rho $, propagation speed with respect to $p_3$ is $O(\ell)$ and all other propagation speeds are bounded by $\bar{\rho}$. Therefore the shift with respect to $x_3$ is $\asymp \rho T\lesssim \ell$ as $T\le T^*$ and it is observable as $T\ge T_*=h|\log h|\rho^{-2}$\,\footnote{\label{foot-27-8} But in estimates we can skip the logarithmic factor using our standard scaling arguments.}.

Let us apply three-term approximation. Then as the first term does not includes any error we can estimate it by
\begin{equation*}
 C\bigl(\beta h^{-1}+\rho^2 h^{-2}\bigr)\rho\gamma^3 h^{-1}T_*\asymp
C\bigl(\beta h^{-1}+\rho^2 h^{-2}\bigr)\rho ^{-1}\gamma^3
\end{equation*}
which delivers the first term in $S(\rho,T)\ell^3$.

The second term is linear with respect to perturbation $(A-A_\varepsilon)$ and as we consider a shift by $x_3$ in the estimate of this term we also can take $T=T_*$. Indeed, contribution of intervals $|t|\asymp T'$ with $T_*\le T'\le T^*$ to this term are negligible if we include a logarithmic factor in $T_*$\,\footref{foot-27-8}. Then this term does not exceed
$ C\bigl(\beta h^{-1}+\rho^2 h^{-2}\bigr)\rho \ell^3 \nu\gamma^2 T_*^2 h^{-2}$ and it is less than the first term in $S(\rho,T)\gamma^3$.

Finally, the third term does not exceed the second term in $S(\rho,T)\gamma^3$.
\end{proof}

After estimate (\ref{27-3-54}) has been proven we can estimate the contribution of the given element to the Tauberian error by
$CS(\rho, T) \gamma^3\rho ^2 h^2T^{-2}$\ \footnote{\label{foot-27-9} Factors $\rho^2$ and $h^2T^{-2}$ (rather than $hT^{-1}$) appear because we consider the trace term.} which is
\begin{equation}
C(\beta+\rho ^2 h^{-1} )
\bigl(\rho T^{-2} + \rho^{-1} \nu^2 h T\bigr) \gamma^3.
\label{27-3-56}
\end{equation}
Consider an error appearing when we replace in the Tauberian expression $T$ by $T_*$. The first two terms are negligible on intervals $|t|\asymp T'$ with $T_*\le T'\le T^*$ and the third term contributes here
\begin{equation*}
C(\beta+\rho ^2 h^{-1} ) \rho^{-1} \nu^2 h T' \gamma^3
\end{equation*}
which sums to its value as $T'=T$. Therefore this error does not exceed (\ref{27-3-56}) as well.

Minimizing expression (\ref{27-3-56}) by $T\le T^*$ we get
\begin{multline}
C(\beta+\rho ^2 h^{-1} )
\bigl(\rho T^{*\,-2} + \rho^{-\frac{1}{3}} \nu^{\frac{4}{3}}h^{\frac{2}{3}} \bigr) \gamma^3 \\
\asymp C(\beta+\rho ^2 h^{-1} )
\bigl(\rho+ \rho ^{-1}\ell^2 + \rho^{-\frac{1}{3}} \nu^{\frac{4}{3}}h^{\frac{2}{3}} \bigr) \gamma^3
\label{27-3-57}
\end{multline}
where we do not include the last term with $T=T_*$ as then the first term would be larger than $C\beta h^{-2}\rho^3\gamma^3$ which is the trivial estimate.

Now let us sum over the partition. Observe first that

\begin{claim}\label{27-3-58}
Contribution of zone $\{\rho:\, \rho \ge (\beta h)^{\frac{1}{2}}\}$ does not exceed
\begin{equation}
 Q_0\Def Ch^{-1} + Ch^{-\frac{1}{3}}\nu^{\frac{4}{3}}
\label{27-3-59}\\[-15pt]
\end{equation}
\end{claim}
as $\rho$ here would be in the positive degree. Consider now contribution of zone $\{\rho:\, \rho \le (\beta h)^{\frac{1}{2}}\}$.

\subsection{Strong non-degenerate case}
\label{sect-27-3-4-2}

Here $\rho$ is in the negative degree but we can help it under \emph{strong non-degeneracy assumption\/}
\begin{equation}
\min_j |V- 2j \beta h|+|\partial V| \ge \epsilon_0
\qquad\text{in\ \ } B(0, 1).
\label{27-3-60}
\end{equation}
which later will be relaxed. Indeed, the relative measure of $x$-balls with $\gamma=\rho^2$ where operator is non-elliptic is $\rho^2 (\beta h)^{-1}$ as $\rho\ge h^{\frac{1}{3}}$. Then the total contribution of such elements does not exceed $\rho^2 h^{-1} \bigl(\rho ^{-1}\ell^2 + \rho^{-\frac{1}{3}} \nu^{\frac{4}{3}}h^{\frac{2}{3}} \bigr)$
and summation over $\rho$ results in (\ref{27-3-59}). Meanwhile the total contribution of balls with $|\xi_3-A'_3| \le \bar{\rho}=h^{\frac{1}{3}}$ does not exceed $C\bar{\rho}^5 h^{-3}$ which is smaller\footnote{\label{foot-27-10} These arguments work even if $\beta \le h^{-\frac{1}{3}}$ (and therefore
$(\beta h)^{\frac{1}{2}}\le h^{\frac{1}{3}}$): we just set $\gamma= h\rho^{-1}$ as $(\beta h)^{\frac{1}{2}}\le \beta \le h^{\frac{1}{3}}$.}.

However we have another restriction, namely,
$\bar{\rho}\ge C_0 \mu \beta^{-1}$\ \footnote{\label{foot-27-11} Here we can take $\mu =\|\partial A'\|_{\sL^\infty}$ without resetting it to $1$ if the former is smaller.}. Because of this we need to increase the remainder estimate
by $C \bar{\rho} \beta ^2h^{-1} \times \bar{\rho} ^2(\beta h)^{-1}$ i.e.
\begin{equation}
Q'= \mu ^3\beta^{-2}  h^{-2}.
\label{27-3-61}
\end{equation}
We should not be concerned about zone
$\{\rho:\, \bar{\rho}\ge\rho\ge (\beta h)^{\frac{1}{2}}\}$ as here we can always use $T\asymp \beta ^{-1}$ and its contribution to remainder will be the same
$\mu ^3\beta^{-2} h^{-2}$.

Now we need to pass from the Tauberian expression with $T=T_*$ to the magnetic Weyl expression and we need to consider only two first terms in the successive approximations. We can involve our standard methods of Section~\ref{book_new-sect-18-9}:
note that $|x-y|\le c\rho T_*=C\varepsilon$ in the propagation and then we consider another unperturbed operator with $V=V(y)$ and
$A'_j= A'(y) + \langle \nabla A_j (y), x-y\rangle$ frozen at point $y$ (when we later set $x=y$). Then one can see that there terms modulo error not exceeding $Q_0$ are respectively
\begin{phantomequation}\label{27-3-62}\end{phantomequation}
\begin{gather}
-h^{-3}\int P_{B_\varepsilon h} (V)\psi\,dx
\tag*{$\textup{(\ref*{27-3-62})}_1$}\label{27-3-62-1}\\
\shortintertext{and}
-h^{-3}\int \bigl(P_B(V)-P_{B_\varepsilon h} (V)\bigr)\psi\,dx
\tag*{$\textup{(\ref*{27-3-62})}_2$}\label{27-3-62-2}
\end{gather}
with $B_\varepsilon = |\nabla \times (A^0+A'_\varepsilon)|$ and then we arrive to estimate (\ref{27-3-64}) below.

Observe that non-degeneracy condition (\ref{27-3-60}) was used only to estimate by $\rho^2(\beta h)^{-1}$ a relative measure of some set. However the same estimate would be achieved under slightly weaker non-degeneracy condition
\begin{equation}
\min_j |V- 2j \beta h|+|\partial V| + |\det(\Hess V)| \ge \epsilon_0 \qquad\text{in\ \ } B(0, 1);
\label{27-3-63}
\end{equation}
all arguments including transition to the magnetic Weyl expression work. Therefore under this assumption the same estimate holds and we arrive to

\begin{proposition}\label{prop-27-3-13}
Let $\beta h\lesssim 1$ and conditions \ref{27-3-49-*} be fulfilled. Then under non-degeneracy assumptions \textup{(\ref{27-3-60})} or \textup{(\ref{27-3-63})} estimate
\begin{equation}
|\Tr (H^-_{A,V} \psi)+ h^{-3}\int P_{Bh} (V)\psi\, dx|\le CQ
\label{27-3-64}
\end{equation}
holds with $Q=Q_0+Q'$ with $Q_0$ and $Q'$ defined by \textup{(\ref{27-3-59})} and \textup{(\ref{27-3-61})}.
\end{proposition}

\begin{remark}\label{rem-27-3-14}
We will show that for a minimizer $Q\asymp Q_0$ in both cases (and even under even weaker non-degeneracy assumption (\ref{27-3-65})).
\end{remark}

\subsection{Non-degenerate case}
\label{sect-27-3-4-3}

Assume now that  even weaker non-degeneracy condition is fulfilled:
\begin{equation}
\min_j |V- 2j \beta h|+|\partial V| + |\partial^2 V| \ge \epsilon_0 \qquad\text{in\ \ } B(0, 1).
\label{27-3-65}
\end{equation}
Then the measure of degenerate set is $\rho (\beta h)^{-\frac{1}{2}}$ but even this is sufficient to obtain the same sum from the second term. In the first term we get however extra $C\beta |\log h|$ (which is $O(h^{-1})$ provided $\beta \le (h|\log h|)^{-1}$) but we can help with this too: indeed, as we fix $\ell \ge 2\bar{\ell}$ the relative measure does not exceed
$\rho^2 (\beta h \ell)$ and summation results in $O(h^{-1})$. We still need to consider set $\{x:\, \ell (x)\le 2\bar{\ell}\}$, but its contribution is obviously less than $C\beta \bar{\ell}|\log h|$ which in turn is
$O(h^{-1}+\nu |\log h|)$ (and this is $O(h^{-1})$ for a minimizer).

However contribution of the degenerate set becomes
$C\beta h^{-2}\bar{\rho}^3\bar{\ell}$ which boils down to the same expression $Q'$.
Then we arrive to

\begin{proposition}\label{prop-27-3-15}
Let $\beta h\lesssim 1$ and conditions \ref{27-3-49-*} be fulfilled. Then under non-degeneracy assumption \textup{(\ref{27-3-65})} estimate \textup{(\ref{27-3-63})} holds with $Q=Q_0+Q''$, $Q''=Q'+ \nu |\log h|$ with $Q_0$ and $Q'$ defined by \textup{(\ref{27-3-59})} and \textup{(\ref{27-3-61})} respectively.
\end{proposition}

We leave easy details to the reader.

\subsection{Degenerate case}
\label{sect-27-3-4-4}

Let us derive a remainder estimate without any non-degeneracy assumptions. In comparison with the non-degenerate case we need to sum
$\beta \rho^{-\frac{1}{3}} \nu^{\frac{4}{3}}h^{\frac{2}{3}}$ and we sum it over $\rho\ge \rho_*$ resulting in the same expression with $\rho$ replaced by $\rho_*$; adding contribution of degenerate zone $\beta \rho_*^3$ we get
$\beta \rho_*^{-\frac{1}{3}} \nu^{\frac{4}{3}}h^{\frac{2}{3}}+\beta h^{-2}\rho_*^3$ which should be minimized by $\rho_*\ge \bar{\rho}$ resulting in
$\beta \nu^{\frac{6}{5}}h^{\frac{2}{5}} +\beta \bar{\rho}^3$ i.e.
\begin{equation}
Q'''=\beta \nu^{\frac{6}{5}}h^{\frac{2}{5}} + \beta h^{-\frac{1}{2}}+
 \mu^3\beta^{-2}h^{-2}+ \nu |\log h|.
\label{27-3-66}
\end{equation}
Thus we arrive to

\begin{proposition}\label{prop-27-3-16}
Let $\beta h\lesssim 1$ and conditions \ref{27-3-49-*} be fulfilled. Then estimate \textup{(\ref{27-3-63})} holds with $Q$ replaced by $Q=Q_0+Q'''$ with $Q_0$ and $Q'''$ a defined by \textup{(\ref{27-3-59})} and \textup{(\ref{27-3-66})} respectively.
\end{proposition}

\begin{remark}\label{rem-27-3-17}
We are going to apply our results to $V=W^\TF_B+\lambda$ with chemical potential $\lambda$. We know that

\begin{enumerate}[label=(\roman*), fullwidth]
\item\label{rem-27-3-17-i}
As $M=1$ (single nucleus case) after rescalings non-degeneracy condition (\ref{27-3-60}) is fulfilled everywhere including the \emph{boundary zone\/}
$\{x:\, \epsilon_0 \bar{r} \le r(x) \le C_0 \bar{r}\}$;

\item\label{rem-27-3-17-ii}
As $M\ge 2$ (multiple nuclei case) after rescalings non-degeneracy condition (\ref{27-3-60}) is fulfilled as $r(x)\le \epsilon d$ where $d$ is the minimal distance between nuclei;

\item\label{rem-27-3-17-iii}
Further, as $M\ge 2$ and $B\le Z^{\frac{4}{3}}$ after rescalings non-degeneracy condition (\ref{27-3-65}) is fulfilled in the zone
$\{x:\,Z^{-\frac{1}{3}}\le r(x) \le \epsilon_0 \bar{r}\}$ where $r(x)$ is the distance to the closest nuclei and
$\bar{r}=\min \bigl(B^{-\frac{1}{4}},\, (Z-N)_+^{-\frac{1}{3}}\bigr)$;

\item\label{rem-27-3-17-iv}
On the other hand, non-degeneracy condition (\ref{27-3-63}) is uncalled:
while we believe that that this condition is often fulfilled while \textup{(\ref{27-3-60})} fails we have no proof of this;

\item\label{rem-27-3-17-v}
As $M\ge 2$ in the boundary zone a more delicate scaling needs to be applied to improve remainder estimate which is possible not only because $W^\TF_B$ is more regular than just $\sC^2$ but also has some special properties.
\end{enumerate}
\end{remark}

\section{Endgame}
\label{sect-27-3-5}

Until now in this Section we assumed that $A'$ satisfies equation to minimizer locally (and assumptions (\ref{27-3-14})--(\ref{27-3-15})) but now we assume that $A'$ is a minimizer.

\subsection{Upper estimate}
\label{sect-27-3-5-1}
Let us first derive an upper estimate for $\E^*_\kappa$ and for this we need to pick-up some $A'$. First of all we try $A'=0$ resulting in
\begin{gather}
\E^*_\kappa \le \cE^*_0 + Ch^{-1}\le \cE^*_\kappa + Ch^{-1}+ C\kappa \beta^2
\notag\\
\intertext{which is a good estimate as $\kappa \beta^2 \lesssim h^{-1}$:}
\E^*_\kappa \le \cE^*_\kappa + Ch^{-1}.
\label{27-3-67}
\end{gather}
However as $\kappa \beta^2 \gtrsim h^{-1}$ we need to be more subtle. Namely we pick up a mollified minimizer for modified functional $\bar{\cE}_\kappa (A')$ (\ref{27-2-7}). More precisely, let $A'$ be the minimizer for $\bar{\cE}_\kappa (A')$; then $\cE_\kappa (A') = \cE^*_\kappa + O(h^{-1})$.

Still it is not a good choice as our approach relies upon $\sC^2$-smoothness but $A'$ is only $\sC^{\frac{3}{2}}$-smooth.

\begin{proposition}\label{prop-27-3-18}
Let $\beta h\lesssim 1$ and $\kappa \beta^2 h\gtrsim 1$; let $A'$ be a minimizer for the modified functional $\bar{\cE}_\kappa (A')$ and let $A'_\varepsilon$ be its $\varepsilon$-mollification. Then there exists $\varepsilon>0$
such that\begin{phantomequation}\label{27-3-68}\end{phantomequation}
\begin{gather}
|\partial A'_\varepsilon|\le C\mu = \kappa \beta h,\quad
|\partial^2 A'_\varepsilon|\le C\nu =
C \bigl(1+ (\kappa \beta)^{\frac{4}{3}}h^{\frac{2}{3}}\bigr)|\log h|
\tag*{$\textup{(\ref*{27-3-68})}_{1,2}$}\label{27-3-68-*}\\
\shortintertext{and}
\cE_\kappa (A'_\varepsilon) = \cE^*_\kappa + O(h^{-1}).
\label{27-3-69}
\end{gather}
\end{proposition}

\begin{proof}
From equation to $A'$ we observe that
\begin{gather}
|\partial(A'-A'_\varepsilon)|\le
C\kappa( \varepsilon + \beta h\varepsilon^{\frac{1}{2}}), \quad
|\partial ^2 A'_\varepsilon|\le
C\kappa(1 +\beta h\varepsilon^{-\frac{1}{2}})|\log h|
\label{27-3-70}\\
\intertext{and}
|\cE_\kappa (A')- \cE_\kappa (A'_\varepsilon)|\le
C(\varepsilon^2 + \beta h\varepsilon^{\frac{3}{2}})\kappa h^{-3} +
C\kappa (\varepsilon +\beta h\varepsilon^{\frac{1}{2}})^2 h^{-2}
\label{27-3-71}
\end{gather}
because linear with respect to $\partial(A'-A'_\varepsilon)$ terms disappear due to equation to a minimizer. Then as $\varepsilon \ge h$ the right-hand expression of (\ref{27-3-71}) is
$O(h^{-1}+ \kappa \beta h^{-2}\varepsilon^{\frac{3}{2}})$
and we take
$\varepsilon = \min \bigl(1, (\kappa \beta)^{-\frac{2}{3}}h^{\frac{2}{3}}\bigr)$\,\footnote{\label{foot-27-12} Then $\varepsilon \gtrsim h$ due to $\kappa \beta^2 h\gtrsim 1$.} resulting in \ref{27-3-68-*} and (\ref{27-3-69}) since $|\partial A'|\le C\kappa\beta h$.
\end{proof}

Then applying Propositions~\ref{prop-27-3-12},~\ref{prop-27-3-15} and \ref{prop-27-3-16} to operator $H_{A_\varepsilon,V}$ with $A_\varepsilon=A^0+A'_\varepsilon$ we arrive to

\begin{proposition}\label{prop-27-3-19}
\begin{enumerate}[label=(\roman*), fullwidth]
\item\label{prop-27-3-19-i}

As $\beta h\lesssim 1$ and $\kappa \beta ^2 h \lesssim 1$ estimate \textup{(\ref{27-3-67})} holds;

\item\label{prop-27-3-19-ii}
As $\beta h\lesssim 1$ and $\kappa \beta ^2 h \gtrsim 1$ estimate
$\E^*_\kappa \le \cE^*_\kappa + CQ$ holds with ${Q=Q_0}$ under non-degeneracy assumption \textup{(\ref{27-3-65})} and $Q=Q_0+Q'''$ in the general case calculated with
$\nu =\bigl(1+ (\kappa \beta)^{\frac{4}{3}}h^{\frac{2}{3}}\bigr)|\log h|$,
$\bar{\rho}=h^{\frac{1}{2}}$, and $\bar{\ell}= \max (h^{\frac{1}{2}}, \beta^{-1}\nu)$.
\end{enumerate}
\end{proposition}

Indeed, for $A'$ selected above $\mu\le\kappa \beta h $ and one can check easily that $Q'\le Q_0$. Since $\nu$ here is lesser than one derived for a minimizer of $\E^*_\kappa(A')$, we are happy and skip calculation of $Q'''$.

\subsection{Lower estimate}
\label{sect-27-3-5-2}

Estimate (\ref{27-3-63}) implies that
\begin{gather}
\underbracket{
\Tr (H^-_{A,V} \psi)+\kappa^{-1}h^{2}\|\partial A'\|^2}_{=\E_\kappa (A')}
\ge \underbracket{
-h^{-3} \int P_{Bh} (V)\psi\, dx +\kappa^{-1}h^{-2}\|\partial A'\|^2}_{
=\cE_\kappa (A')} - CQ
\notag\\
\intertext{and therefore}
\E^* _\kappa \ge \cE^* _\kappa -CQ\qquad\text{with\ \ }
\E^*_\kappa=\inf_{A'} \E_\kappa(A'),\quad
\cE^*_\kappa=\inf_{A'} \cE_\kappa(A')
\label{27-3-72}
\end{gather}
where $A'=A'_\kappa$ is a minimizer of $\E_\kappa (A')$ and $Q$ is defined in Propositions~\ref{prop-27-3-12}, \ref{prop-27-3-15} and \ref{prop-27-3-16} and $\nu$ is a right-hand expression of (\ref{27-3-47}) or $1$ whatever is larger.
For a sake of simplicity we replace it by a marginally larger expression
\begin{equation}
\nu= (\mu+1) (\kappa |\log h|^2+1)
+ (\kappa \beta )^{\frac{10}{9}} h^{\frac{4}{9}} |\log h|^K+
 \kappa \beta^{\frac{1}{2}}|\log h|;
\tag*{$\textup{(\ref*{27-3-47})}^*$}\label{27-3-47-*}
\end{equation}
recall that in ``old'' (\ref{27-3-47}) $\mu=\|\partial A'\|_{\cL^\infty}$ or $\mu=1$ whatever is larger so we modified the first term here accordingly.

Our problem is that so far we know neither
$\|\partial A'\|_{\cL^\infty}$ nor $\|\partial A'\|$. Observe however that
\begin{multline*}
\E^*_\kappa =\E_\kappa (A') \ge
\E_{2\kappa }(A') +(2\kappa)^{-1}\|\partial A'\|^2\ge
\cE^*_{2\kappa} + (2\kappa)^{-1}\|\partial A'\|^2 - CQ\\[3pt]
\ge
\cE^*_{\kappa}(A') -C\kappa \beta^2 (2\kappa)^{-1}\|\partial A'\|^2 - CQ
\end{multline*}
and therefore combining this with an upper estimate we arrive to estimate
\begin{gather}
\|\partial A'\| \le C (\kappa h^2 Q)^{\frac{1}{2}} + C\kappa \beta h
\label{27-3-73}\\
\shortintertext{and we have also}
\|\partial A'\|_{\sL^\infty} \le C \|\partial A'\| ^{\frac{2}{5}}\cdot
\|\partial^2 A'\|_{\sL^\infty}^{\frac{3}{5}}.
\label{27-3-74}
\end{gather}
We are going to explore what happens if
\begin{equation}
\nu \asymp \mu (\kappa |\log h|^2+1)
\label{27-3-75}
\end{equation}
where the right-hand expression is the sum of all terms in \ref{27-3-47-*} containing $\mu$.

\begin{remark}\label{rem-27-3-20}
\begin{enumerate}[label=(\roman*), fullwidth]
\item\label{rem-27-3-20-i}
Observe first that remainder estimate $Q=O(\beta ^2 h^{-1})$ is guaranteed and therefore $\|\partial A'\|\le C\beta h^{\frac{1}{2}}$. Then $\mu \le C\beta^{\frac{2}{5}}h^{\frac{1}{5}}\nu^{\frac{3}{5}}$ due to (\ref{27-3-74});

\item\label{rem-27-3-20-ii}
Further, if (\ref{27-3-75}) is fulfilled then $\mu \le \beta h^{\frac{1}{2}}|\log h|^K$ and the same estimate holds for $\nu$ and then $Q_0\asymp h^{-1}$, $Q'''\le C\beta h^{-\frac{1}{2}}|\log h|^K$; then the rough estimate to $Q$ is improved and then $\mu \ll \beta h^{\frac{1}{2}}$,
$\nu \ll \beta h^{\frac{1}{2}}$ and finally $Q\asymp h^{-1}$ under assumption (\ref{27-3-65}) and $Q\asymp h^{-1} +\beta h^{-\frac{1}{2}}$ in the general case and we also have nice estimates to $\mu,\nu$.

Therefore we can assume that (\ref{27-3-75}) is not fulfilled, but then $\nu$ is defined by remaining terms and then
\begin{equation}
\nu = \kappa |\log h|+
\kappa \min (\beta ^{\frac{3}{2}}h^{-\frac{1}{2}},\,\beta^{\frac{1}{2}})|\log h| + (\kappa \beta)^{\frac{10}{9}}h^{\frac{4}{9}}|\log h|^K
\label{27-3-76}
\end{equation}
(and from now we do not reset to $1$ if this expression is smaller).
\end{enumerate}
\end{remark}

We almost proved the following estimates for $Q$:
\begin{claim}\label{27-3-77}
$Q=Q_0$ under non-degeneracy assumption (\ref{27-3-65}) and $Q=Q_0+\beta h^{-\frac{1}{2}}\asymp h^{-1}+\beta h^{-\frac{1}{2}}$ in the general case.
\end{claim}
However we need still explore what happens if
$Q\asymp \mu^3 \beta^{-2}h^{-2}$. In this case
$\mu \le (\mu^3 \beta^{-2})^{\frac{1}{5}}\nu^{\frac{3}{5}}$ and then
$\mu^{\frac{2}{5}}\le \beta^{-\frac{2}{5}}\nu^{\frac{3}{5}}$ and using (\ref{27-3-76}) one can prove easily (\ref{27-3-77}) unless
$\beta \le h^{-\frac{1}{5}}|\log h|^K$ in which case
 $\nu = (\kappa |\log h|+1)$.

\subsection{Weak magnetic field approach.}
\label{sect-27-3-5-3}

Consider now case $\beta \lesssim h^{-\frac{1}{3}}$. Recall that then
$\nu \le C\kappa |\log h|$ and $\mu \ll 1$. Then as we study propagation with respect to $p_3$ we do not need to correct it to $p'_3$ and then we do not need $\bar{\rho}$.

Then we can apply \emph{weak magnetic field approach\/} (see Section~\ref{book_new-sect-13-3}). Now contribution of a partition element with
$\rho \ge C_0\beta^{-1}$ does not exceed $\rho^3 h^{-1}\times \rho^{-2}$ as we use $T\le \epsilon_0\rho$ and the total contribution of such elements does not exceed $Ch^{-1}$; meanwhile the total contribution of elements with
$\rho =C_0\beta^{-1}$ does not exceed
$C\beta^2 h^{-1}\times \beta^{-3}\le Ch^{-1}$ as we use $T=\epsilon_0\beta^{-1}$.

\subsection{Main theorem}
\label{sect-27-3-5-4}

Therefore after we plug $\nu$ into $Q_0$ we have proven our estimate from below and also the main theorem of this Section:

\begin{theorem}\label{thm-27-3-21}
Let $\beta h\lesssim 1$ and $\kappa\le \kappa^*$. Then
\begin{enumerate}[label=(\roman*), fullwidth]
\item\label{thm-27-3-21-i}
Estimate
\begin{equation}
|\E^* _\kappa - \cE^*_\kappa |\le C Q
\label{27-3-78}
\end{equation}
holds where under non-degeneracy assumption \textup{(\ref{27-3-65})}
\begin{gather}
Q\Def h^{-1}+
\kappa^{\frac{40}{27}} \beta^{\frac{40}{27}} h^{\frac{7}{27}}|\log h|^K
\label{27-3-79}\\
\intertext{and in the general case}
Q\Def h^{-1} + \beta h^{-\frac{1}{2}};
\label{27-3-80}
\end{gather}
\item\label{thm-27-3-21-ii}
For a minimizer the following estimate holds:
$\|\partial^2 A'\|_{\sL^\infty}\le C\nu$ with $\nu$ defined by \textup{(\ref{27-3-76})}.
\end{enumerate}
\end{theorem}

We leave as an easy exercise to the reader

\begin{problem}\label{problem-27-3-22}
\begin{enumerate}[label=(\roman*), fullwidth]

\item\label{problem-27-3-22-i}
Starting from estimate $\|\partial^2 A'\|_{\sL^\infty}\le C\nu $ derive from (\ref{27-3-73})--(\ref{27-3-74}) estimate for $\|\partial A'\|_{\sL^\infty}$; consider separately cases $1\le \beta \le h^{-\frac{1}{3}}$,
$h^{-\frac{1}{3}}\le \beta \le h^{-\frac{1}{2}}$ and
$h^{-\frac{1}{2}}\le \beta \le 1$;

\item\label{problem-27-3-22-ii}
Prove that as $\kappa \beta^2 h\gtrsim 1$
\begin{equation}
\|\partial(A'-A'')\| \le C(\kappa Qh^2)^{\frac{1}{2}}
\label{27-3-81}
\end{equation}
where $A'$ and $A''$ are minimizers for $\E_\kappa $ and $\bar{\cE}_\kappa$ respectively; in particular, observe that
$\|\partial(A'-A'')\|\ll \|\partial A'\|$ for $\kappa \beta^2h\gg 1$;

\item\label{problem-27-3-22-iii}
Since (\ref{27-3-81}) holds for $A''$ replaced by $A''_\varepsilon$ as well and since we have estimate $\|\partial^2 A''_\varepsilon \|_{\sL^\infty}\le C\nu $
derive from (\ref{27-3-81}) estimate for
$\|\partial (A'-A''_\varepsilon)\|_{\sL^\infty}$ and then the same estimate for $\|\partial A'\|_{\sL^\infty}$.
\end{enumerate}
\end{problem}

\begin{remark}\label{rem-27-3-23}
In the following observations we use simpler but less sharp upper estimates to critical $\beta$:
\begin{enumerate}[label=(\roman*), fullwidth]
\item
\label{rem-27-3-23-i}
Under non-degeneracy assumption (\ref{27-3-65}) $Q\le C h^{-1}$ and
$\nu \le C\beta^{\frac{1}{2}}|\log h|$ as $\beta \le h^{-\frac{2}{3}}$ and therefore $\mu \le 1$ as $\beta \le h^{-\frac{2}{3}}|\log h|^{-K}$;

\item
\label{rem-27-3-23-ii}
In the general case  $Q\le C h^{-1}+ C\beta h^{-\frac{1}{2}}$ and
$\nu \le C\beta^{\frac{1}{2}}|\log h|$ as $\beta \le h^{-\frac{2}{3}}$ and therefore $\mu \le 1$ as $\beta \le h^{-\frac{3}{5}}|\log h|^{-K}$; we used here estimate $\mu \le C(\kappa h^2Q)^{\frac{1}{5}}\nu^{\frac{3}{5}}$.
\end{enumerate}
\end{remark}

\section{$\N$-term asymptotics}
\label{sect-27-3-6}

\subsection{Introduction}
\label{sect-27-3-6-1}

In the application to the ground state energy one needs to consider also $\N$-term asymptotics and $\D$-term estimates. Let us start from the former: we consider $\N$-term
\begin{equation}
\int e(x,x,0)\psi(x)\,dx.
\label{27-3-82}
\end{equation}
Again we consider this asymptotics in the more broad content of assumptions $V\in \sC^2$ and \ref{27-3-49-*} and we will follow arguments employed for a trace term using the same notations. Then the Tauberian error does not exceed
$CS(\rho, T) \gamma^3 hT^{-1}$\,\footnote{\label{foot-27-13} With $S(\rho, T)$ defined by (\ref{27-3-28}).} which is
\begin{equation}
C(\beta+\rho ^2 h^{-1} )h^{-1}
\bigl(\rho ^{-1} T^{-1} + \rho^{-3} \nu^2 h T^{-2}\bigr) \gamma^3.
\tag*{$\textup{(\ref*{27-3-56})}'$}\label{27-3-56-'}
\end{equation}
Minimizing by $T\le T^*$ we get
\begin{equation}
C(\beta+\rho ^2 h^{-1} )h^{-1}
\bigl(\rho ^{-1}+ \rho ^{-2}\ell  +
\rho^{-\frac{5}{3}} \nu^{\frac{2}{3}} h^{\frac{1}{3}}\bigr) \gamma^3.
\tag*{$\textup{(\ref*{27-3-57})}'$}\label{27-3-57-'}
\end{equation}
Then summation over zone $\{\rho:\,\rho^2 \ge \beta h\}$ results in
\begin{gather}
R_0=C(h^{-2} + h^{-\frac{5}{3}}\nu^{\frac{2}{3}})
\tag*{$\textup{(\ref*{27-3-59})}'$}\label{27-3-59-'}\\
\intertext{and we need to consider only contributions of elements belonging to zone $\{\rho:\,\rho^2\le \beta h\}$}
\beta h^{-1}\bigl(\rho ^{-1}+ \rho ^{-2}\ell  +
\rho^{-\frac{5}{3}} \nu^{\frac{2}{3}} h^{\frac{1}{3}}\bigr) \gamma^3.
\label{27-3-83}
\end{gather}

\subsection{Strong non-degenerate case}
\label{sect-27-3-6-2}
Under non-degeneracy conditions (\ref{27-3-60}) or (\ref{27-3-63}) expression (\ref{27-3-83}) should be multiplied by $\rho^2 (\beta h)^{-1}\gamma^{-3}$ and after summation by $\rho$ we get an extra term $Ch^{-2}|\log h|$.

However we can get rid off the logarithmic factor by the standard trick: in one direction time could be improved to $\rho ^{1-\delta}\ell^{\delta}$. We leave easy details to the reader.

Further, adding $R'=C\beta h^{-2}\bar{\rho}\times \bar{\rho}^2(\beta h)^{-1}$
with $\bar{\rho}=C_0\max(\mu \beta^{-1},\, h^{\frac{1}{2}})$\, \footref{foot-27-11} i.e.
\begin{equation}
R'= \mu^3\beta^{-3}h^{-3}
\tag*{$\textup{(\ref*{27-3-61})}'$} \label{27-3-61-'}
\end{equation}
which is a contribution of zone $\{\rho:\,\rho\le \bar{\rho}\}$ we arrive to

\begin{proposition}\label{prop-27-3-24}
Let $\beta h\lesssim 1$ and conditions \ref{27-3-49-*} be fulfilled. Then under non-degeneracy assumption \textup{(\ref{27-3-60})} or \textup{(\ref{27-3-63})} estimate
\begin{equation}
|\int \bigl(\tr e(x,x,0)-h^{-3}\int P'_{B h} (V)\bigr)\psi(x)\,dx|\le CR
\label{27-3-84}
\end{equation}
holds with $R=R_0+R'$, $R_0$, $R'$ defined by \ref{27-3-59-'} and \ref{27-3-61-'}.
\end{proposition}

\begin{remark}\label{rem-27-3-25}
The above estimate is sufficiently good as a weak magnetic field approach brings remainder estimate $C\mu h^{-2}$ even without any non-degeneracy assumption. We leave easy details to the reader.
\end{remark}

\subsection{Non-degenerate case}
\label{sect-27-3-6-3}

Under non-degeneracy assumption (\ref{27-3-65}) we need to apply more subtle arguments than before. Consider first subelements\footnote{\label{foot-27-14} We call them ``subelements'' but they live in the phase spaces in contrast to elements which live in the coordinate spaces.}. with $\rho\ge \ell$; for them we need to multiply expression (\ref{27-3-83}) by $\rho\gamma^{-3}$ and sum by $\rho\ge \rho_*$ resulting in
$C\beta h^{-1}|\log h|+
 C\beta h^{-\frac{2}{3}} \nu^{\frac{2}{3}}\rho_*^{-\frac{2}{3}}$.

On the other hand, for subelements with $\rho\le \ell$ we need to multiply (\ref{27-3-83}) by $\rho^2(\beta h )^{-1}\gamma^{-3}$ and sum by
$\rho_*\le \rho\le \ell$ and then by $\ell\ge \rho_*$ resulting in the same expression albeit with a factor $|\log h|^2$ instead of $|\log h|$. Here we get also $Ch^{-2}|\log h|$ term but we deal with it exactly as in the previous Subsubsection.

We need also add contributions of subelements with
$\rho_*\le \rho \le \ell \le \bar{\ell}$ and with $\rho\le \rho_*$. For the former subelements we need to consider only term
$\beta h^{-1}\rho^{-2}\bar{\ell}\gamma^3$ in (\ref{27-3-83}) and only as $\bar{\ell}=\nu \beta^{-1}$ resulting in
$\beta h^{-1}\rho_*^{-1}\bar{\ell}\gamma^3$ and contribution of the latter subelements we estimate by $\beta h^{-2}\rho_*^2$. So we get
\begin{equation}
C\beta \bigl( h^{-1}|\log h|+ h^{-1}\rho_*^{-1}\bar{\ell}+
h^{-\frac{2}{3}} \nu^{\frac{2}{3}} \rho_*^{-\frac{2}{3}}
+ h^{-2}\rho_*^2\bigr)
\label{27-3-85}
\end{equation}
which should be minimized by $\rho_*\ge \bar{\rho}$ resulting in
\begin{equation*}
C\beta \bigl(h^{-1}|\log h|^2+ h^{-\frac{4}{3}}\bar{\ell}^{\frac{2}{3}}+
h^{-1} \nu^{\frac{1}{2}} + h^{-2}\bar{\rho}^2\bigr)
\end{equation*}
and plugging $\bar{\rho}$ and $\bar{\ell}=\nu \beta^{-1}$ we arrive to
\begin{equation}
R''=C\beta \bigl(h^{-1}|\log h|^2+ h^{-\frac{4}{3}}\nu^{\frac{2}{3}}\beta^{-\frac{2}{3}}+
h^{-1} \nu^{\frac{1}{2}}\bigr) + C\mu^2 \beta^{-1}h^{-2}
\label{27-3-86}
\end{equation}
thus proving Proposition~\ref{prop-27-3-26}\ref{prop-27-3-26-i} below.

\subsection{Degenerate case}
\label{sect-27-3-6-4}

In the general case we arrive to (\ref{27-3-85}) albeit with factor $\rho_*^{-1}$
\begin{equation}
C\beta \bigl( h^{-1}\rho_*^{-1}+h^{-1}\rho_*^{-2}\bar{\ell}+
h^{-\frac{2}{3}} \nu^{\frac{2}{3}} \rho_*^{-\frac{5}{3}}
+ h^{-2}\rho_*\bigr)
\label{27-3-87}
\end{equation}
which should be minimized by
$\rho_*\ge \bar{\rho}$ resulting in
\begin{equation*}
C\beta \bigl(h^{-\frac{3}{2}} + h^{-\frac{5}{3}}\bar{\ell}^{\frac{1}{3}}
+ h^{-\frac{4}{3}} \nu^{\frac{1}{4}} +h^{-2}\bar{\rho}\bigr)
\end{equation*}
and plugging $\bar{\rho}$ and $\bar{\ell}$ we arrive to
\begin{equation}
R'''=C\beta \bigl(h^{-\frac{3}{2}} + h^{-\frac{5}{3}}\nu^{\frac{1}{3}}\beta^{-\frac{1}{3}}+
 h^{-\frac{4}{3}} \nu^{\frac{1}{4}} \bigr)+
C\mu h^{-2}
\label{27-3-88}
\end{equation}
thus proving Proposition~\ref{prop-27-3-26}\ref{prop-27-3-26-ii}:

\begin{proposition}\label{prop-27-3-26}
Let $\beta h\lesssim 1$ and conditions \ref{27-3-49-*} be fulfilled. Then
\begin{enumerate}[fullwidth, label=(\roman*)]
\item\label{prop-27-3-26-i}
Under non-degeneracy assumption \textup{(\ref{27-3-65})} estimate \textup{(\ref{27-3-63})}
holds with $R=R_0+R''$, $R_0$, $R''$ defined by \ref{27-3-59-'} and \textup{(\ref{27-3-86})};

\item\label{prop-27-3-26-ii}
In the general case estimate \textup{(\ref{27-3-63})}
holds with $R=R_0+R'''$, $R_0$, $R'''$ defined by \ref{27-3-59-'} and \textup{(\ref{27-3-88})}.
\end{enumerate}
\end{proposition}

\section{$\D$-term estimate}
\label{sect-27-3-7}

Consider now $\D$-term
\begin{equation}
\D \bigl( [e(x,x, 0) - h^{-3} P'_{B h} (V)]\psi ,\,
[e(x,x, 0) - h^{-3} P'_{B h} (V)]\psi\bigr)
\label{27-3-89}
\end{equation}
with $\psi \in \sC^\infty_0((B(0,1))$.

\begin{proposition}\label{prop-27-3-27}
Let $\beta h \lesssim 1$ and $A'$ satisfy \ref{27-3-49-*}. Then under non-degeneracy assumption \textup{(\ref{27-3-60})} $\D$-term \textup{(\ref{27-3-89})} does not exceed $CR^2$ with $R=R_0+R'$, $R_0$ and $R'$ defined by \ref{27-3-59-'} and \ref{27-3-61-'}.
\end{proposition}

\begin{proof}\emph{Step 1\/}.
Let us apply Fefferman-de Llave decomposition (\ref{book_new-16-3-1}); then we need to consider pairs of elements $B(\bar{x},r)$ and $B(\bar{y},r)$ with
$3r\le |\bar{x}-\bar{y}|\le 4r$. If $r\ge \bar{\rho}$ on each of these elements we should consider $(\gamma,\rho)$ subelements (we call them ``subelements'' but they live in the phase spaces in contrast to elements which live in the coordinate spaces). Then we have three parameters--$(r,\rho_x,\rho_y)$.

Since for each $\bar{y}$ the number of matching $x$-elements is $\asymp 1$ and then summation with respect to $\rho_x\ge \rho^*=(\beta h)^{\frac{1}{2}}$ results in $R_0r^2$ multiplied by a contribution of $(\gamma,\rho_y)$-subelement; after summation by $r$ and then by these $(\gamma,\rho_y)$-subelements we get $CR_0R$. Similarly we are dealing with $\rho_y\ge \rho^*$.

Therefore we need to consider only case when both $\rho_x$ and $\rho_y$ do not exceed $\rho^*$. If $r \ge c\rho^*$ then ``the relative measure trick'' allows us to add factors $\rho_x^2 (\beta h)^{-1}$ and
$\rho_y^2 (\beta h)^{-1}$ even if $\rho_x^2 \ge r$ or $\rho_y^2 \ge r$ and then the total contribution of such subelements also does not exceed $CR^2$.

\medskip\emph{Step 2.\/}
Consider next $h\le r \le \rho^*$\,\footnote{\label{foot-27-15} Observe that we do not need to keep $t\ge \bar{\rho}$ but we need to keep $\rho r\ge h$.} and we should look only at $\rho_x\le \rho^*$,
$\rho_y\le \rho^*$. Further, if $\rho_x^2\ge \epsilon r$ or
$\rho_y^2\ge \epsilon r$ we can always inject factor $c\rho_x^2r^{-1}$ or $c\rho_y^2r^{-1}$ ending up again with $CR^2$.

On the other hand, if both $\rho_x^2\le \epsilon r$ and $\rho_y^2\le \epsilon r$ but $r\ge c\bar{\rho}$ we can apply ``the relative measure trick'' but comparing the measure of $\rho_x^2$- or $\rho_y^2$-elements with violated ellipticity assumption to the total measure of $B(z,r)$; then we can inject factors $(\rho_x^2/r)^{\theta}$ and $(\rho_yx^2/r)^{\theta}$ with
$0\le \theta \le 1$ and we select any $\theta: \frac{5}{6}< \theta <1$ to have positive powers of $\rho_x$ and $\rho_y$ and power of $r$ (counting $r^{-1}$) greater than $-3$. We end up again with $CR^2$.

Observe that these arguments cover also cases $\rho_x\le \bar{\rho}$ or
$\rho_y\le \bar{\rho}$.

\medskip\emph{Step 3.\/}
To estimate contribution of zone $\{ (x,y):\, r\le h\}$ we just estimate
$|e(x,x,\tau)|\le Ch^{-3}$.
\end{proof}

As $M\ge 2$ we will need to estimate $\D$-term under non-degeneracy assumptions (\ref{27-3-63}), or (\ref{27-3-65}), or without any non-degeneracy assumption.

\begin{proposition}\label{prop-27-3-28}
Let $\beta h \lesssim 1$ and $A'$ satisfy \ref{27-3-49-*}. Then $\D$-term \textup{(\ref{27-3-89})} does not exceed $CR^2$ where
\begin{enumerate}[fullwidth, label=(\roman*)]

\item\label{prop-27-3-28-i}
Under non-degeneracy assumption \textup{(\ref{27-3-63})} $R=R_0+R'$, $R_0$ and $R'$ defined by \ref{27-3-59-'} and \ref{27-3-61-'};

\item\label{prop-27-3-28-ii}
Under non-degeneracy assumption \textup{(\ref{27-3-65})} $R=R_0+R''$, $R_0$ and $R''$ defined by \ref{27-3-59-'} and \textup{(\ref{27-3-86})};

\item\label{prop-27-3-28-iii}
In the general case $R=R_0+R'''$, $R_0$ and $R'''$ defined by \ref{27-3-59-'} and \textup{(\ref{27-3-88})}
\end{enumerate}
\end{proposition}

\begin{proof}
Let us use ideas used in the proofs of Proposition~\ref{prop-27-3-26} and Propositions~\ref{prop-27-3-24} and~\ref{prop-27-3-26}. Let us apply Fefferman-de Llave decomposition (\ref{book_new-16-3-1}); then we need to consider pairs of elements $B(\bar{x},r)$ and $B(\bar{y},r)$ with
$3r\le |\bar{x}-\bar{y}|\le 4r$. If $r\ge \bar{\rho}$ on each of these elements we should consider $(\gamma,\rho)$ subelements\footref{foot-27-14}. Then we have three parameters--$(r,\rho_x,\rho_y)$. On the other hand, there is a scaling function $\ell(x)$ and covering of $B(0,1)$ by $\ell$-elements.

\medskip\noindent
\emph{Part 1.\/} Consider case of $\ell_x \lesssim r$ (and therefore
$\ell_y \lesssim r$). Then we must assume that $\rho_x \ell_x \gtrsim h$, $\rho_y \ell_y \gtrsim h$. Observe as in Step 1 that if
$\rho_x \gtrsim \rho^*=(\beta h)^{\frac{1}{2}}$ then the relative density of such subelements is $\rho_x^2 /(\beta h)$ and therefore summation over such subelements of the given $x$-element results in $CR_0 \ell_x^3$. Therefore double summation over corresponding subelements of $x$- and $y$-elements
results in $CR_0^2 \ell_x ^3\ell_y^3 r^{-1}$. Finally, after double summation over $x$- and $y$- elements we get $CR_0^2\int |x-y|^{-1}\,dxdy$ which does not exceed $CR_0^2$.

\smallskip
 Therefore in what follows we need to consider only subelements with
$\rho_x \lesssim \rho^*$, $\rho_y \lesssim \rho^*$\,\footnote{\label{foot-27-16} Due to positivity quadratic form $\D(.,.)$ we need to consider only ``pure'' pairs. We will use this observation many times.}.
Further, observe that the same arguments are applicable as
$\ell_x \gtrsim \rho^*$, $\ell_y \gtrsim \rho^*$ and we are left with pairs of elements with $\ell_x \lesssim \rho^*$, $\ell_y\lesssim \rho^*$ and their subelements with $\rho_x \lesssim \rho^*$, $\rho_y \lesssim \rho^*$ as we will always keep $\ell_x \ge \bar{\rho}$, $\ell_y \ge \bar{\rho}$.

Observe that summation of (\ref{27-3-83}) over subelements with
$\rho\ge \ell$ of the given element results in
\begin{equation}
C\beta h^{-1} \bigl(\ell^{-1}+\ell^{-\frac{5}{3}}\nu^{\frac{2}{3}}h^{\frac{1}{3}}\bigr)\ell^3.
\label{27-3-90}
\end{equation}
On the other hand, as $\rho\le \ell$ the relative density of $\rho$-subelements of the given $\ell$-element does not exceed
$C\rho^2\ell^{-2}$ and therefore summation over such subelements results in (\ref{27-3-90}) again.

However in (\ref{27-3-83}) if $\ell\le \bar{\ell}$ we need to take in the middle term $\ell=\bar{\ell}$ and here we can ignore other options but $\bar{\ell}=\nu \beta^{-1}$.

Then summation of this term over subelements with $ \rho \ge \max(\ell,\rho_*)$ results in $C\beta h^{-1}\min (\rho_*^{-2},\ell^{-2})\bar{\ell}\ell^3$ and summation over subelements with $\rho_*\le \rho \le \ell$
results in $C\beta h^{-1}\rho_*^{-2}\bar{\ell}(1+|\log \rho_*\ell^{-1}|)\ell^3$ (and should be counted as $\ell\ge \rho_*$ only.

Finally, contribution of subelements with $\rho\le \rho_*$ does not exceed
$C\beta h^{-2}\rho_*^3\ell$ and $C\beta h^{-2}\rho_*\ell^3$ as $\rho_*\le \ell$ and $\rho_*\ge \ell$ respectively. So, in the former case the total contribution of all subelements does not exceed
\begin{gather}
C\beta h^{-1} \Bigl[\ell^{-1}+\ell^{-\frac{5}{3}}\nu^{\frac{2}{3}}h^{\frac{1}{3}}+
\rho_*^{-2}\bar{\ell}(1+|\log \rho_*\ell^{-1}|)+ h^{-1}\rho_*^3\ell^{-2}\Bigr]\ell^3.
\label{27-3-91}\\
\intertext{Minimizing this expression by $\rho_*$ we get}
C\beta h^{-1} \Bigl[\ell^{-1}+\nu^{\frac{2}{3}}h^{\frac{1}{3}}\ell^{-\frac{5}{3}}+
\bar{\ell}^{\frac{3}{5}} h^{-\frac{2}{5}}
(1+|\log \ell_*\ell^{-1}|)^{\frac{3}{5}}\ell^{-\frac{4}{5}} +
h^{-1}\bar{\rho}^3\ell^{-2}\Bigr]\ell^3
\label{27-3-92}
\end{gather}
achieved as
\begin{multline}
\rho_* =\rho_*(\ell)
\asymp \max( (\bar{\ell}h \ell^2|\log h|)^{\frac{1}{5}},\bar{\rho})\\
\text{as\ \ }
\ell\ge \ell_*= \max((\bar{\ell}h|\log h|)^{\frac{1}{3}},\,\bar{\rho}).
\label{27-3-93}
\end{multline}

Then $\N$-term does not exceed
\begin{multline}
C\beta h^{-1} \int_{\ell_x \ge \ell_*} \Bigl[\ell_x^{-1}+\nu^{\frac{2}{3}}h^{\frac{1}{3}}\ell_x^{-\frac{5}{3}}+
\bar{\ell}^{\frac{3}{5}} h^{-\frac{2}{5}}
(1+|\log \ell_*\ell_x^{-1}|)^{\frac{3}{5}}\ell_x^{-\frac{4}{5}} +
h^{-1}\bar{\rho}^3\ell_x^{-2}\Bigr]\,dx+\\[3pt]
C\beta h^{-1} \Bigl[\ell_*^{-1}+\nu^{\frac{2}{3}}h^{\frac{1}{3}}\ell_*^{-\frac{5}{3}}+
\bar{\ell}^{\frac{3}{5}} h^{-\frac{2}{5}}\ell_*^{-\frac{4}{5}} +
h^{-1}\bar{\rho}^3\ell_*^{-2}\Bigr]
\mes(\{\ell_x\le \ell_*\})
\label{27-3-94}
\end{multline}
where the first and second terms estimate contributions of elements with $\ell_x\ge \ell_*$ and $\ell_x\le \ell_*$ respectively.

\begin{remark}\label{rem-27-3-29}
Observe that
\begin{enumerate}[label=(\roman*), fullwidth]
\item\label{rem-27-3-29-i}
Under non-degeneracy assumption (\ref{27-3-60}) we get $C(R_0+R')$ as expected and under non-degeneracy assumption (\ref{27-3-63}) we get $C(R_0+R'|\log h|)$ but this is only because we counted here contribution of subelements with
$\{x:\,\rho_x\le \ell_x\}$ in the less efficient way.

\item\label{rem-27-3-29-ii}
Under non-degeneracy assumption (\ref{27-3-65}) we get $C(R_0+R'')$ and
and in the general case we get $C(R_0+R''')$ where $R_0,R',R'',R'''$ are defined in Propositions~\ref{prop-27-3-24} and~\ref{prop-27-3-26}.
\end{enumerate}
\end{remark}

Similarly, total contribution of the zone considered here (in Part I) to $\D$-term does not exceed
\begin{multline}
C\beta^2 h^{-2} \iint_{\ell_x \ge \ell_*, \ell_y\ge \ell*, |x-y|\ge \max(\ell_x,\ell_y)} \\
\times \Bigl[\ell_x^{-1}+\nu^{\frac{2}{3}}h^{\frac{1}{3}}\ell_x^{-\frac{5}{3}}+
\bar{\ell}^{\frac{3}{5}} h^{-\frac{2}{5}}
(1+|\log \ell_*\ell_x^{-1}|)^{\frac{3}{5}}\ell_x^{-\frac{4}{5}} +
h^{-1}\bar{\rho}^3\ell_x^{-2}\Bigr]\\[3pt]
\times
\Bigl[\ell_y^{-1}+\nu^{\frac{2}{3}}h^{\frac{1}{3}}\ell_y^{-\frac{5}{3}}+
\bar{\ell}^{\frac{3}{5}} h^{-\frac{2}{5}}
(1+|\log \ell_*\ell_y^{-1}|)^{\frac{3}{5}}\ell_y^{-\frac{4}{5}} +
h^{-1}\bar{\rho}^3\ell_y^{-2}\Bigr]\\[3pt]
\shoveright{\times \,|x-y|^{-1}\,dxdy}\\[3pt]
+ C\beta ^2h^{-2} \Bigl[\ell_*^{-1}+\nu^{\frac{2}{3}}h^{\frac{1}{3}}\ell_*^{-\frac{5}{3}}+
\bar{\ell}^{\frac{3}{5}} h^{-\frac{2}{5}}\ell_*^{-\frac{4}{5}} +
h^{-1}\bar{\rho}^3\ell_*^{-2}\Bigr]^2
\iint _{\ell_x \le \ell_*, \ell_y\le \ell*} |x-y|^{-1}\,dxdy.
\label{27-3-95}
\end{multline}
Then under non-degeneracy assumption (\ref{27-3-60}) we get $C(R_0+R')^2$ as expected and under non-degeneracy assumption (\ref{27-3-63}) we get $C(R_0+R'|\log h|)^2$ but this is only because we counted here contribution of subelements with $\{x:\,\rho_x\le \ell_x\}$ in the less efficient way. Using method employed in the proof of Proposition~\ref{prop-27-3-24} we can recover estimate $C(R_0+R')^2$ as well.

Further, under non-degeneracy assumption (\ref{27-3-63}) we get $C(R_0+R'')^2$
and under non-degeneracy assumption (\ref{27-3-65}) we get $C(R_0+R''')^2$.

\medskip\noindent
\emph{Part 2.\/} Consider case of $\ell_x \ge C r$ (and therefore
$\ell_y \asymp \ell_x$). Then we apply the same arguments as before albeit with $\ell^2$ replaced by $\ell r$. First, consider pairs of subelements with $\rho_x \ge \rho^*$, $\rho_y \ge \rho^*$. Their contributions to $\D$-term does not exceed expression (\ref{27-3-83}) with $\rho=\rho_x$, $\gamma=\gamma_x$ multiplied by $\rho_x^2 (\beta h)^{-1}$ multiplied by expression (\ref{27-3-83}) with $\rho=\rho_y$, $\gamma=\gamma_y$ multiplied by $\rho_x^2 (\beta h)^{-1}$ multiplied by $|x-y|^{-1}$
and double summation by $\rho_x$, $\rho_y$ results in
$Ch^{-4} (1+\nu^{\frac{4}{3}}h^{\frac{2}{3}}) \gamma_x^3\gamma_y^3 |x-y|^{-1}$;
then the double summation over $x$, $y$ returns $Ch^{-4} (1+\nu^{\frac{4}{3}}h^{\frac{2}{3}})\iint |x-y|^{-1}\,dxdy\lesssim CR^2$.

Then we need to consider pairs with $\rho_x \le \rho^*$,
$\rho_y \le \rho^*$ and also pairs with
$|x-y|\le h \max (\rho_x^{-1}, \rho_y^{-1})$.

\medskip
Next consider pairs of subelements with
$\rho^*\ge \rho_x \ge (\ell r)^{\frac{1}{2}}$,
$\rho^*\ge \rho_y \ge (\ell r)^{\frac{1}{2}}$. Their contributions to $\D$-term does not exceed expression
\begin{equation*}
C\beta^2h^{-2}\bigl( \rho_x^{-1} \rho_y^{-1} +
\bar{\ell}^2 \rho_x^{-2} \rho_y^{-2} +\rho_x^{-\frac{5}{3}}\rho_y^{-\frac{5}{3}}\nu^{\frac{4}{3}}h^{\frac{2}{3}}
\bigr)\underbracket{\gamma_x^3\gamma_y^3 |x-y|^{-1}}
\end{equation*}
and the double summation over $x,y$ in $B(z,\ell)$ with $\ell_z=\ell$ results in the same expression with the selected factor replaced by $\ell^5$ and then the double summation over $\rho_x$, $\rho_y$ results in
\begin{equation}
C\beta^2h^{-2}\Bigl[ \ell^{-2} + \bar{\ell}^2 \ell^{-4}+ \ell^{-\frac{10}{3}}\nu^{\frac{4}{3}}h^{\frac{2}{3}}\Bigr]\ell^5.
\label{27-3-96}
\end{equation}
Meanwhile, considering pairs of subelements with
$(\ell r)^{\frac{1}{2}}\ge \rho_x \ge \rho_*$,
$(\ell r)^{\frac{1}{2}}\ge \rho_y \ge \rho_*$ (we use $\rho_*=\rho_*(\ell)$ and $\ell_*$ introduced in (\ref{27-3-93})) we gain factor
$\rho_x^2\rho_y^2 /(\rho \ell)$ in the summation by subelements and we arrive to the same expression (\ref{27-3-96}) but with a logarithmic factor at $\bar{\ell}^2$:
\begin{equation*}
C\beta^2h^{-2}\Bigl[ \ell^{-2} +
\bar{\ell}^2 \ell^{-4}(1+|\log \rho_*\ell^{-1}|)^3+ \ell^{-\frac{10}{3}}\nu^{\frac{4}{3}}h^{\frac{2}{3}}\Bigr]\ell^5.
\end{equation*}
However we can get rid off logarithmic factors exactly as in the proof of Proposition~\ref{prop-27-3-26} thus getting (\ref{27-3-96})
Then we need to sum by balls $B(z,\ell_z)$ resulting in the same expression multiplied by $\ell^{-3}$ and integrated
\begin{gather}
C\beta^2h^{-2}\int_{\{\ell_x\ge \ell_*\}} \Bigl[ \ell_x^{-2} +
\bar{\ell}^2 \ell_x^{-4}+ \ell_x^{-\frac{10}{3}}\nu^{\frac{4}{3}}h^{\frac{2}{3}}\Bigr]\ell_x^2\,dx
\notag\\
\intertext{which we estimate by}
C\beta^2h^{-2}\int_{\ell\ge \ell_*} \Bigl[ \ell^{-2} +
\bar{\ell}^2 \ell^{-4}+ \ell^{-\frac{10}{3}}\nu^{\frac{4}{3}}h^{\frac{2}{3}}\Bigr]\ell^{2+m}\,d\ell
\label{27-3-97}
\end{gather}
with $m=2,1,0$ under non-degeneracy assumptions (\ref{27-3-63}),
(\ref{27-3-65}) and in the general case respectively. Then we arrive to the terms of integrand multiplied by $\ell$ and calculated either for $\ell=1$ or $\ell=\ell_*$. One can see easily that this does not exceed $CR^2$ with $R$ defined in Proposition~\ref{prop-27-3-24}.

One also can derive easily the same estimate for contributions of the pairs of subelements with $\rho_x\le \rho_*(\ell)$, $\rho_y\le \rho_*(\ell)$,
$\ell\ge \ell_*$, and for contributions of the pairs of subelements with $\rho_x\le \ell_*$, $\rho_y\le \ell_*$, $\ell\le \ell_*$, assuming in both cases that $\rho_x r\ge h$, $\rho_y r\ge h$.

Finally, like in the proof of Proposition~\ref{prop-27-3-26} we estimate contribution of zone $\{x,y:\,\rho_x r\le h,\ \rho_y r\le h\}$. We leave easy details to the reader.
\end{proof}

\begin{remark}\label{rem-27-3-30}
These arguments also work to estimate
\begin{equation}
\D\bigl( \Gamma_x (hD-A)_x\cdot \boldupsigma e(.,.,0),\,
\Gamma_x \bigr(hD-A)_x\cdot \boldupsigma e(.,.,0)).
\label{27-3-98}
\end{equation}
Indeed, Weyl expression for $\Gamma_x (hD-A)_x\cdot \boldupsigma e(.,.,0)$ is just $0$. Therefore we arrive under either of non-degeneracy assumptions (\ref{27-3-60}), (\ref{27-3-63}), (\ref{27-3-65}) and in the general case to estimate
\begin{equation}
\|\nabla A'\| \le C\kappa h^2 R
 \label{27-3-99}
\end{equation}
which could be better or worse than estimate
$\|\nabla A'\| \le C\kappa^{\frac{1}{2}}$ which we have already. It is not clear if estimate $\|\nabla A'\| \le C\kappa$ holds.
\end{remark}

\chapter{Microlocal Analysis: \texorpdfstring{$\beta h\gtrsim 1$}{\textbeta h \textgtrsim 1}}
\label{sect-27-4}

Now let us investigate the case of $\beta h\gtrsim 1$. We assume not only that $\kappa \lesssim 1$ but also \ref{27-2-25-*}:
$\kappa \beta h^2 |\log h|^K\le 1$. We can apply the same arguments as before and in the end of the day we will get the series of the statements; we leave most of the easy details to the reader.

\section{Estimate to minimizer}
\label{sect-27-4-1}

Observe first that
\begin{gather}
\|\partial A'\|^2 \le C\beta h^{-2}\times \kappa h^2=
C\kappa \beta \le Ch^{-2}|\log h|^{-K}
\label{27-4-1}\\
\shortintertext{and}
\qquad \qquad |\Delta A'|\le C\beta h^{-2} \times \kappa h^2 =
C\kappa \beta \le Ch^{-2}|\log h|^{-K}\notag\\
\shortintertext{and therefore}
|\partial^2 A'|\le
C\kappa \beta \le Ch^{-2}|\log h|^{-K}.
\label{27-4-2}
\end{gather}

First of all repeating arguments leading to Proposition~\ref{prop-27-4-8} we arrive to estimate (\ref{27-3-46}) modified
\begin{multline}
\|\partial^2 A'\|_{\sL^\infty}
\le C\kappa\beta h \nu ^{\frac{1}{2}} |\log h|^2
\\[3pt]
+ C\kappa \beta h|\log h|\bigl(
\underbracket{h^{-\frac{3}{5}} \nu^{\frac{1}{10}} }+
\underbracket{ h^{-\frac{4}{7}}\nu^{\frac{1}{7}}}+
h^{-\frac{1}{2}}\nu^{\frac{1}{4}}|\log h|^2\bigr)+
C \kappa |\log h| \beta^{\frac{1}{2}}  \|\partial V\| _{\sL^\infty}\\[3pt]
+ C\|\partial A'\|_{\sL^\infty}
\label{27-4-3}
\end{multline}
(note a new factor $\beta h$ in the first line). We prove it first as $\mu=1$ and then rescale and in virtue of Remark~\ref{rem-27-3-9}\ref{rem-27-3-9-ii} we can always replace $\mu$ by $\nu^{\frac{1}{2}}\ll h^{-1}$.

Further, under additional super-strong non-degeneracy assumption
\begin{equation}
\min _{x, j\ge 0} |V-2j\beta h |\asymp 1
\label{27-4-4}
\end{equation}
we can skip two selected terms in the second line of (\ref{27-4-3}):
\begin{multline}
\|\partial^2 A'\|_{\sL^\infty}
\le C\kappa\beta h \nu ^{\frac{1}{2}} |\log h|^2 \\[3pt]
+ C\kappa \beta h^{-\frac{1}{2}}
\nu^{\frac{1}{4}}|\log h|^3+
C \kappa |\log h| \beta^{\frac{1}{2}}  \|\partial V\| _{\sL^\infty}
+ C\|\partial A'\|_{\sL^\infty}.
\label{27-4-5}
\end{multline}

Then we arrive to the following:

\begin{proposition}\footnote{\label{foot-27-17} Cf.  Proposition~\ref{prop-27-3-10}.}\label{prop-27-4-1}
Let $\beta h \gtrsim 1$, $\kappa \le \kappa^*$ and \ref{27-2-25-*} be fulfilled; let $V\in \sC^2$; then
\begin{enumerate}[label=(\roman*), fullwidth]
\item\label{prop-27-4-1-i}
The following estimates hold:
\begin{equation}
\|\partial ^2A'\|_{\sL^\infty} \le \nu
\label{27-4-6}
\end{equation}
with
\begin{align}
&\nu \Def C\kappa \beta^{\frac{1}{2}}|\log h|+
C(\kappa \beta)^{\frac{10}{9}}h^{\frac{4}{9}}|\log h|^K
 &&\text{as\ \ } \kappa \beta h\le 1 \label{27-4-7}\\
\shortintertext{and}
&\nu \Def C\kappa \beta^{\frac{1}{2}}|\log h|+
C(\kappa \beta)^{\frac{4}{3}}h^{\frac{2}{3}}|\log h|^K
\quad &&\text{as\ \ } \kappa \beta h\ge 1; \label{27-4-8}
\end{align}
\item\label{prop-27-4-1-ii}
Moreover, under assumption \textup{(\ref{27-4-4})}
estimate $\nu$ is given by \textup{(\ref{27-4-8})} even as
$\kappa \beta h \lesssim 1$.
\end{enumerate}
\end{proposition}

\begin{remark}\label{rem-27-4-2}
\begin{enumerate}[label=(\roman*), fullwidth]
\item\label{rem-27-4-2-i}
While case $\beta h \asymp 1$ has been already explored, we missed an important case when non-degeneracy assumption (\ref{27-4-4}) is fulfilled; so we reexamine this case;

\item\label{rem-27-4-2-ii}
While technically (\ref{27-4-3}) and (\ref{27-4-5}) hold even if assumption \ref{27-2-25-*} fails provided $\nu \le \epsilon \beta$ we cannot guarantee in this case that this inequality holds.
\end{enumerate}
\end{remark}

\section{Trace term asymptotics}
\label{sect-27-4-2}

Further, continuing our analysis we arrive to the following

\begin{proposition}\footnote{\label{foot-27-18} Cf.  Propositions~\ref{prop-27-3-13}, \ref{prop-27-3-15} and~\ref{prop-27-3-16}; only factor $\beta h$ appears in the definition of $Q_0$.}\label{prop-27-4-3}
Let $\beta h \gtrsim 1$, $\kappa \le \kappa^*$ and $\nu h^2\le 1$\,\footnote{\label{foot-27-19} In the framework of Proposition~\ref{prop-27-4-1} for a minimizer this assumption is due to \ref{27-2-25-*}.}. Then

\begin{enumerate}[label=(\roman*), fullwidth]
\item\label{prop-27-4-3-i}
Under non-degeneracy assumption \textup{(\ref{27-3-60})}, or \textup{(\ref{27-3-63})} remainder estimate
\begin{gather}
|\Tr (H^- _{A,V}\psi)+h^{-3}\int P_{B h}(V)\psi\,dx|\le Q\label{27-4-9}\\
\intertext{holds with $Q=Q_0+Q'$,}
Q_0 \Def C\beta + C\beta \nu^{\frac{4}{3}}h^{\frac{2}{3}},
\label{27-4-10}
\end{gather}
and $Q'$ defined by \textup{(\ref{27-3-61})};

\item\label{prop-27-4-3-ii}
Under non-degeneracy assumption \textup{(\ref{27-3-65})} remainder estimate
\textup{(\ref{27-4-9})} holds with $Q=Q_0+Q''$ with $Q''=Q'+ \nu |\log h$ and $Q_0$ and $Q'$ defined by \textup{(\ref{27-4-9})} and \textup{(\ref{27-3-61})} respectively;

\item\label{prop-27-4-3-iii}
In the general case remainder estimate
\textup{(\ref{27-4-9})} holds with $Q=Q_0+Q'''$ with $Q_0$ and $Q'''$ defined by \textup{(\ref{27-4-9})} and \textup{(\ref{27-3-66})} respectively.
\end{enumerate}
\end{proposition}

Applying Proposition~\ref{prop-27-4-1} we arrive to

\begin{corollary}\label{cor-27-4-4}
In the framework of Proposition~\ref{prop-27-4-3} let $A'$ be a minimizer. Then
\begin{enumerate}[label=(\roman*), fullwidth]
\item\label{cor-27-4-4-i}
Under non-degeneracy assumption \textup{(\ref{27-3-60})}, or \textup{(\ref{27-3-63})}, or even \textup{(\ref{27-3-65})} estimate \textup{(\ref{27-4-9})} holds with
\begin{align}
&Q = Q_0= C\beta +
C\kappa^{\frac{40}{27}} \beta^{\frac{67}{27}}h^{\frac{34}{27}}|\log h|^K
\qquad &&\text{as\ \ } \kappa \beta h\le 1 \label{27-4-11}
\shortintertext{and}
&Q = Q_0= C\beta +
C\kappa^{\frac{16}{9}} \beta^{\frac{25}{9}}h^{\frac{14}{9}}|\log h|^K
\qquad &&\text{as\ \ } \kappa \beta h\ge 1. \label{27-4-12}
\end{align}
\item\label{cor-27-4-4-ii}
Furthermore, under assumption \textup{(\ref{27-4-4})} $Q_0$ is defined by \textup{(\ref{27-4-12})} even as $\kappa \beta h \le 1$;
\item\label{cor-27-4-4-iii}
In the general case estimate \textup{(\ref{27-4-9})} holds with
\begin{equation}
Q=Q_0+\beta h^{-\frac{1}{2}}+
\kappa^{\frac{8}{5}} \beta^{\frac{13}{5}}h^{\frac{6}{5}}|\log h|^K.
\label{27-4-13}
\end{equation}
\end{enumerate}
\end{corollary}

\begin{remark}\label{rem-27-4-5}
Observe that
\begin{enumerate}[label=(\roman*), fullwidth]
\item\label{rem-27-4-5-i}
If assumption \ref{27-2-25-*} holds then
$Q_0 \lesssim \kappa^{\frac{1}{3}}\beta^{\frac{4}{3}}h^{-\frac{4}{3}}\lesssim \beta h^{-2}$ where the middle expression appears in (\ref{27-2-23}); on the other hand, if \ref{27-2-25-*} fails then the reverse inequalities hold; in the general case we assume that $\beta h^2 \ll 1$ to get a remainder estimate smaller than the main term;

\item\label{rem-27-4-5-ii}
Also $\nu \lesssim \beta$ provided \ref{27-2-25-*} holds; if $\kappa=1$ then $\nu \lesssim \beta$ if and only if \ref{27-2-25-*} holds.
\end{enumerate}
\end{remark}

\section{Endgame}
\label{sect-27-4-3}

Similarly to Theorem~\ref{thm-27-3-21} we arrive to

\begin{theorem}\label{thm-27-4-6}
Let $\beta h \gtrsim 1$, $\kappa \le \kappa^*$ and \ref{27-2-25-*} be fulfilled. Then estimate \textup{(\ref{27-3-78})}
\begin{equation*}
|\E^* _\kappa - \cE^*_\kappa |\le C Q
\end{equation*}
holds where
\begin{enumerate}[label=(\roman*), fullwidth]
\item\label{thm-27-4-6-i}
Under non-degeneracy assumption \textup{(\ref{27-3-65})} $Q$ is defined by \textup{(\ref{27-4-11})} and \textup{(\ref{27-4-12})};
\item\label{thm-27-4-6-ii}
In the general case $Q$ is defined by \textup{(\ref{27-4-11})}; in particular, $Q=\beta h^{-\frac{1}{2}}$ as $\kappa \beta h\lesssim 1$.
\end{enumerate}
\end{theorem}

\begin{problem}\footnote{\label{foot-27-20} Cf. problem~\ref{problem-27-3-22}.}\label{problem-27-4-7}
In this new settings recover estimates for $\|\partial (A'-A'')\|$,
$\|\partial (A'-A'')\|_{\sL^\infty}$ and $\|\partial A'\|_{\sL^\infty}$ where
$A''$ is a minimizer for $\bar{\cE}(A'')$.
\end{problem}

\section{$\N$-term asymptotics and $\D$-term estimates}
\label{sect-27-4-4}

Repeating arguments of the proofs of Propositions~\ref{prop-27-3-24}, \ref{prop-27-3-26} we arrive to

\begin{proposition}\label{prop-27-4-8}
Let $\beta h\gtrsim 1$ and conditions \ref{27-3-49-*} be fulfilled. Then

\begin{enumerate}[label=(\roman*), fullwidth]
\item\label{prop-27-4-8-i}
Under non-degeneracy assumption \textup{(\ref{27-3-60})} or \textup{(\ref{27-3-63})} estimate \textup{(\ref{27-3-84})} holds with $R=R_0+R'$,
\begin{equation}
R_0=\beta h^{-1}+\beta h^{-\frac{2}{3}}\nu^{\frac{2}{3}}
\label{27-4-15}
\end{equation}
and $R'$ defined by \ref{27-3-61-'};
\item\label{prop-27-4-8-ii}
Under non-degeneracy assumption \textup{(\ref{27-3-63})} estimate \textup{(\ref{27-3-84})} holds with $R=R_0+R''$, $R_0$ and $R''$ defined by \textup{(\ref{27-4-15})} and \textup{(\ref{27-3-86})};
\item\label{prop-27-4-8-iii}
In the general case estimate \textup{(\ref{27-3-84})} holds with $R=R_0+R'''$, $R_0$ and $R'''$ defined by \textup{(\ref{27-4-15})} and \textup{(\ref{27-3-88})}.
\end{enumerate}
\end{proposition}

Repeating arguments of the proof of Propositions~\ref{prop-27-3-27} and~\ref{prop-27-3-27} we arrive to

\begin{proposition}\label{prop-27-4-9}
Let $\beta h\gtrsim 1$ and conditions \ref{27-3-49-*} be fulfilled. Then
\begin{enumerate}[label=(\roman*), fullwidth]
\item\label{prop-27-4-9-i}
Under non-degeneracy assumptions \textup{(\ref{27-3-60})} or \textup{(\ref{27-3-63})} $\D$-term \textup{(\ref{27-3-89})} does not exceed $CR^2$ with $R=R_0+R'$, $R_0$ and $R'$ defined by \textup{(\ref{27-4-15})} and \ref{27-3-61-'} respectively;

\item\label{prop-27-4-9-ii}
Under non-degeneracy assumption \textup{(\ref{27-3-65})} $\D$-term \textup{(\ref{27-3-89})} does not exceed $CR^2$ with $R=R_0+R''$, $R_0$ and $R'$ defined by \textup{(\ref{27-4-15})} and \textup{(\ref{27-3-86})};

\item\label{prop-27-4-9-iii}
In the general case  $\D$-term \textup{(\ref{27-3-89})} does not exceed $CR^2$ with $R=R_0+R'''$, $R_0$ and $R'$ defined by \textup{(\ref{27-4-15})} and \textup{(\ref{27-3-88})}.
\end{enumerate}
\end{proposition}

\begin{Problem}\label{Problem-27-4-10}
In the general case (without any non-degeneracy assumptions) for
$\beta h\lesssim 1$ and for $\beta h\gtrsim 1$ improve remainder estimates for trace term and $\N$-term and estimates for $\D$-term (so, make $R'''$ and $Q'''$ smaller) under assumption $V\in \sC^{s}$ with $s>2$.

To do this use more advanced partition of unity as in Chapter~\ref{book_new-sect-25}. Most likely, however, it will affect only terms $C\beta h^{-\frac{1}{2}}$ and
$C\beta h^{-\frac{3}{2}}$ in $Q'''$ and $R'''$ replacing them by
$C\beta h^{(s-4)/(s+2)}$ and $C\beta h^{-1-2/(s+2)}$ respectively.
\end{Problem}

\chapter{Global trace asymptotics in the case of Thomas-Fermi potential: \texorpdfstring{$B\le Z^{\frac{4}{3}}$}{B \textle Z\textfoursuperior\textthreesuperior}}
\label{sect-27-5}

\section{Introduction}
\label{sect-27-5-1}

In this Section we consider global trace asymptotics for Thomas-Fermi potential. First we consider the singularity zones where our results would follow from Section~\ref{book_new-sect-26-3},then we consider their interaction with the regular zone which would lead to the deterioration of the remainder estimates as $\beta \gg h^{-\frac{1}{2}}$ and finally the boundary zone where non-degeneration properties could be violated (especially for $M\ge 2$) which requires rather subtle analysis and usage of specific properties of Thomas-Fermi potential.

\begin{remark}\label{rem-27-5-1}
Recall that according to Chapter~\ref{book_new-sect-25} there are two cases:
\begin{enumerate}[label={(\alph*)}, fullwidth]
\item\label{rem-27-5-1-a}
$B\le Z^{\frac{4}{3}}$ when the most contributing to both the number of particles and the energy zone is $\{x:\,\ell(x) \asymp r^* = Z^{-\frac{1}{3}}\}$ (where $\ell(x)$ is the distance to the closest nucleus) and then rescaling
$x\mapsto xr^{*\,-1}$, $\tau \mapsto \tau Z^{-\frac{4}{3}}$ we arrive in this zone to $\beta= BZ^{-1}$, $h =Z^{-\frac{1}{3}}$ with $\beta h\le 1$;

\item\label{rem-27-5-1-b}
$Z^{\frac{4}{3}}\le B\le Z^3$ when the most contributing to both the number of particles and the energy zone is
$\{x:\,\ell(x) \asymp r^* = B^{-\frac{2}{5}}Z^{\frac{1}{5}}\}$
and then rescaling $x\mapsto xr^{*\,-1}$,
$\tau \mapsto \tau B^{-\frac{2}{5}}Z^{-\frac{4}{5}}$ we arrive in this zone to $\beta= B^{\frac{2}{5}}Z^{-\frac{1}{5}}$,
$h =B^{\frac{1}{5}}Z^{-\frac{3}{5}}$ with $\beta h\ge 1$.
\end{enumerate}
\end{remark}

We also recall that in the movable nuclei model distance between nuclei was greater than $\epsilon r^*$ (which would be the case in the current settings as well as we show later), so we will assume that it is the case deducting our main results.

\section{Estimates to the minimizer}
\label{sect-27-5-2}

\subsection{Preliminary analysis}
\label{sect-27-5-2-1}

Consider potential $V$ with Coulomb-like singularities, exactly as in Section~\ref{book_new-sect-26-3} i.e. satisfying (\ref{book_new-26-3-1})--(\ref{book_new-26-3-3}).

\begin{proposition}\footnote{\label{foot-27-21} Cf. Proposition~\ref{book_new-prop-26-3-1}.}\label{prop-27-5-2}
Let $V$ satisfy \textup{(\ref{book_new-26-3-1})}--\textup{(\ref{book_new-26-3-2})} and
\begin{multline}
|D^\alpha W|\le C_\alpha \sum_{1\le m\le M} z_m\bigl(|x-\bar{\y}_m|+1\bigr)^{-4}|x-\bar{\y}_m|^{-|\alpha|}\\
\forall \alpha:|\alpha|\le 2.
\label{27-5-1}
\end{multline}
Let $\kappa \le \kappa^*$ and $\beta h\le 1$. Then the near-minimizer $A$ satisfies
\begin{gather}
|\Tr (H_{A,V}^-) +
\int h^{-3}P_{\beta h} (V(x))\Bigr)\,dx|
\le Ch^{-2}
\label{27-5-2}\\
\shortintertext{and}
\|\partial A'\| \le C\kappa^{\frac{1}{2}}.
\label{27-5-3}
\end{gather}
\end{proposition}

\begin{proof}
We follow the proof of Proposition~\ref{book_new-prop-26-3-1}. Observe that scaling $x\mapsto (x-\bar{\y}_m )\ell^{-1}$, $\tau \mapsto \tau\zeta ^{-2}$ leads us to
\begin{gather}
h\mapsto h_1=h\ell^{-1}\zeta^{-1},\qquad
\beta \mapsto \beta_1=\beta \ell \zeta^{-1},\qquad
\kappa \mapsto \kappa_1=\kappa\zeta^2\ell.
\label{27-5-4}
\intertext{Also observe that for $\zeta=\ell^{-\frac{1}{2}}$}
\kappa_1\beta_1 h_1^2 \gtrsim 1 \implies
\ell \gtrsim \ell^* =(\beta+1)^{-\frac{1}{2}}h^{-1}
\label{27-5-5}
\end{gather}
and for $\kappa \asymp 1$ those are equivalent.

\medskip\noindent
(i) First, we pick up $A'=0$. Then
\begin{equation}
|\Tr \bigl( H_{A^0,V} ^-(0)\bigr)+h^{-3}\int P_{\beta h}(V(x))\,dx|\le
Ch^{-2};
\label{27-5-6}
\end{equation}
this estimate follows from the standard partition with $\ell$-admissible partition elements, supported in $\{x:\,\ell (x)\lesssim \ell \}$ as $\ell=\ell_*$ and and in $\{x:\,\ell (x)\asymp \ell \}$ as $\ell\ge 2\ell_*$.

\medskip\noindent
(ii) On the other hand, consider $A'\ne 0$. Let us prove first that
\begin{equation}
\Tr^- (\psi_\ell H \psi_\ell) \ge -C_\varepsilon h^{-2} -
\varepsilon \kappa^{-1} h^{-2} \|\partial A'\|^2
\label{27-5-7}
\end{equation}
as $\ell=\ell_*=h^2$ where one can select constant $\varepsilon$ arbitrarily small.

Rescaling $x\mapsto (x-\bar{\y}_m) /\ell$ and $\tau \mapsto \tau/\ell$ and therefore $h\mapsto h \ell^{-\frac{1}{2}}\asymp 1$ and $A\mapsto A\ell^{\frac{1}{2}}$ (because singularity is Coulomb-like), we arrive to the same problem with the same $\kappa$ and with $\ell=h=1$ and with $\beta$ replaced by $\beta h^3$. Then as $\beta h^3\le 1$ we refer to Appendix~\ref{book_new-sect-26-A-1} as $H_{A,V}\ge H_{A',V'}$ with $V'=V-\beta^2|x|^2$.

\medskip\noindent
(iii) Consider now $\psi_\ell$ as in (i) with $\ell\ge \ell_*$. Then according to Theorems~\ref{thm-27-3-21} and \ref{thm-27-4-6} as $\beta_1h_1\lesssim 1$,
$\kappa \zeta^2\ell\le \kappa^*$
\begin{multline}
\Tr^- \bigl(\psi_\ell H_{A,V} \psi_\ell\bigr) +
h^{-3}\int P_{\beta h} (V(x))\psi_\ell^2(x) \,dx\\
\ge - C_\varepsilon \zeta^2 (h_1^{-1}+\beta_1h_1^{-1})
-\varepsilon \kappa ^{-1} h^ {-2} \|\partial A'\|^2.
\label{27-5-8}
\end{multline}

\begin{remark}\label{rem-27-5-3}
Observe that if $\psi_\ell $ is supported in
$\{x:\,\frac{1}{2}r \le \ell(x)\le 2r\}$ then we can take a norm of $\partial A'$ over $\{x:\,\frac{1}{4}r \le \ell(x)\le 4r\}$. Indeed, we can just replace $A'$ by $A''=\phi_\ell (A'-\eta)$ with arbitrary constant $\eta$ and with $\phi_\ell$ supported in $\{x:\,\frac{1}{4}r \le \ell(x) \le 4r\}$ and equal $1$ in $\{x:\,\frac{1}{3}r \le \ell\le 3r\}$ (and $\varepsilon $ by
$c \varepsilon$).

Then summation of these norms returns
$-C_0\varepsilon \kappa^{-1}h^{-2}\|\partial A'\|^2$.
\end{remark}

Furthermore, the first term in the right-hand expression of (\ref{27-5-8}) is
$-C_\varepsilon (\zeta^3\ell^{-1} h^{-1}+ \zeta^2\ell^2 \beta h^{-1})$ and summation over $\ell\ge h^2$ returns $-C_\varepsilon h^{-2}$ since
$\zeta =\min (\ell^{\frac{1}{2}},\, \ell^{-2})$.

\bigskip\noindent
(iv) Consider next zone where $\beta_1h_1\ge 1$ (and $\ell\ge 1$) but still $h_1\le 1$. According to previous Section~\ref{sect-27-5} (\ref{27-5-8}) should be replaced by
\begin{multline}
\Tr^- \bigl(\psi_\ell H_{A,V} \psi_\ell\bigr) +
h^{-3}\int P_{\beta h} (V(x))\psi_\ell^2(x) \,dx \\
\ge - C_\varepsilon \zeta^2 \beta_1 h_1^{-1}
\bigl(1+ \nu_1^{\frac{4}{3}}h_1^{\frac{5}{3}}\bigr)
-\varepsilon \kappa ^{-1} h^ {-2} \|\partial A'\|^2
\label{27-5-9}
\end{multline}
with $\nu_1 = (\kappa_1\beta_1)^{\frac{10}{9}}h_1^{\frac{4}{9}} |\log h|^K $\,\footnote{\label{foot-27-22} Because $\kappa_1\beta_1h_1=
\kappa \beta h\ell \le \kappa (\beta h)^{\frac{3}{4}}\le 1$.} and
$\kappa_1\beta_1 = \kappa\beta \zeta\ell^2=\kappa\beta $ and thus with $\beta_1h_1^{-1}\zeta^2 = \beta h^{-1}\zeta^2\ell^2$. Then summation of the first term in the right-hand expression results in its value as $\ell$ is the smallest i.e. $\beta_1h_1=1$ and one can check easily\footnote{\label{foot-27-23} Sufficient to check as $\beta=h^{-1}$, $\ell=1$ and $\kappa=1$.} that this is less than $Ch^{-2}$.

Further, Remark~\ref{rem-27-5-3} remains valid. Then adding this zone does not change inequality in question.

\bigskip\noindent
(v) The rest of the proof is obvious. Zone $\{x:\,\ell(x)\ge \ell^*\}$ is considered as a single element and just rough variational estimate is used there to prove that its contribution does not exceed $C h^{-2}$.
\end{proof}

\begin{remark}\label{rem-27-5-4}
Later we will improve both upper and lower estimates using different tricks:
imposing non-degeneracy assumptions, picking for an upper estimate semiclassical self-generated magnetic field, using Scott approximation terms. These improvements will lead not only to our final goal, but also to our intermediate one--getting better estimates for a minimizer.
\end{remark}

\begin{proposition}\footnote{\label{foot-27-24} Cf. Proposition~\ref{book_new-prop-26-3-2}.}\label{prop-27-5-5}
In the framework of Proposition~\ref{prop-27-5-2} there exists a minimizer $A$.
\end{proposition}

\begin{proof}
After Proposition~\ref{prop-27-5-2} has been proven we just repeat arguments of the proof of Proposition~\ref{book_new-prop-26-2-2}.
\end{proof}

\subsection{Estimates to a minimizer: interior zone}
\label{sect-27-5-2-2}

Recall equation (\ref{book_new-26-2-14}) for a minimizer $A$:
\begin{multline}
\frac{2}{\kappa h^2} \Delta A_j (x)  = \Phi_j\Def\\
-\Re\tr \upsigma_j\Bigl( (hD -A)_x \cdot \boldupsigma e (x,y,\tau)+
e (x,y,\tau)\,^t (hD-A)_y \cdot \boldupsigma \Bigr) \Bigr|_{y=x}
\tag{\ref*{book_new-26-2-14}}\label{26-2-14xx}
\end{multline}
After rescaling $x\mapsto x/\ell$, $\tau\mapsto \tau/\zeta^2$,
$h\mapsto \hbar=h/(\zeta\ell)$, $A\mapsto A\zeta^{-1}\ell$,
$\beta \mapsto \beta \zeta^{-1}\ell$ this equation becomes (\ref{book_new-26-3-12})
\begin{multline}
\Delta A_j=\\
-2\kappa \zeta^2\ell \hbar^2 \Re \tr \upsigma_j \Bigl(
( \hbar D -\zeta^{-1} A)_x \cdot\boldupsigma e (x,y,\tau) +
e (x,y,\tau) \,^t(\hbar D-\zeta^{-1}A)_y\cdot\boldupsigma
 \Bigr)\Bigr|_{y=x}
\tag{\ref*{book_new-26-3-12}}\label{26-3-12x}
\end{multline}
and as we can take $\zeta^2\ell=1$ we arrive to (\ref{book_new-26-3-13})
\begin{multline}
\Delta A_j=\\
-2\kappa \hbar^2
\Re \tr \upsigma_j\Bigl(
 ( \hbar D-\zeta^{-1} A)_x \cdot \boldupsigma  e (x,y,\tau) +
e (x,y,\tau)\,^t( \hbar D-\zeta^{-1} A)_y \cdot \boldupsigma  \Bigr)\Bigr|_{y=x}.
\tag{\ref*{book_new-26-3-13}}\label{26-3-13x}
\end{multline}
Let us modify arguments of Subsection~\ref{book_new-26-3-1}. First observe that
\begin{equation}
|\partial A'|\le C\kappa ^{\frac{1}{2}}h^{-3}, \quad
|\partial^2 A'|\le C\kappa^{\frac{1}{2}} h^{-5}
\qquad \text{as\ \ }\ell \le 2\ell_*
\label{27-5-10}
\end{equation}
with $\ell_*=h^2$; this follows from above equations rescaled and from
$\beta h^3\le \epsilon_0$. Let
\begin{equation}
\mu (r)= \sup _{\ell(x)\ge r} |\partial A'|\ell\zeta^{-1},\qquad
\nu (r)= \sup _{\ell(x)\ge r} |\partial^2 A'|\ell^2\zeta^{-1};
\label{27-5-11}
\end{equation}
then $\nu(r)$ should not exceed\footnote{\label{foot-27-25} As long as $\beta_1h_1\le 1$.}
\begin{multline}
F(\nu)= C\kappa_1\Bigl(1+\mu +
\min
\bigl(\beta_1^{\frac{3}{2}}h_1^{\frac{1}{2}}, \, \beta_1^{\frac{1}{2}}\bigr)\\
+\beta_1 h_1 \bigl(\nu^{\frac{1}{10}}h_1^{-\frac{3}{5}}+
\nu^{\frac{1}{7}}h_1^{-\frac{4}{7}}+
\nu^{\frac{1}{4}}h_1^{-\frac{1}{2}}|\log h_1|^2\bigr)\Bigr)|\log h_1|
+ C\kappa^{\frac{1}{2}}(\ell\zeta^2)^{-\frac{1}{2}}
\label{27-5-12}
\end{multline}
where here $\nu=\nu(\frac{1}{2}r)$, $\mu=\mu(r)$, the last term is just an estimate for $\|\partial A'\|$ rescaled and $\ell\asymp r$ in that term. Indeed, (\ref{27-5-12}) is derived exactly as (\ref{27-3-46}) but here we cut a hole
$\{x:\,\ell(x)\le \frac{1}{2}r\}$ in our domain.

We also know that
$\mu \le C\nu^{\frac{3}{5}}\kappa^{\frac{1}{5}}(\ell\zeta^2)^{-\frac{1}{5}}$.
Using (\ref{27-5-12}) and (\ref{27-5-10}) one can prove easily that $\nu(r)$ does not exceed solution of the equation $\nu=F(\nu)$ multiplied by $C$\,\footnote{\label{foot-27-26} As long as a resulting expression rescaled, see (\ref{27-5-14}) is a decaying function of $\ell$.}, i.e.
\begin{multline}
\nu \le C\kappa_1 \Bigl(1+ \min \bigl(\beta_1^{\frac{3}{2}}h_1^{\frac{1}{2}}, \, \beta_1^{\frac{1}{2}}\bigr)\Bigr) |\log h_1| \\[3pt]
+ C\Bigl((\kappa_1 \beta _1)^{\frac{10}{9}}h_1^{\frac{4}{9}}+
(\kappa_1 \beta _1)^{\frac{4}{3}}h_1^{\frac{2}{3}}\Bigr)|\log h_1|^K+
C\kappa^{\frac{1}{2}}(\ell\zeta^2)^{-\frac{1}{2}}.
\label{27-5-13}
\end{multline}

In particular, scaling back and setting $\zeta=\ell^{-\frac{1}{2}}$ we arrive to
\begin{multline}
|\partial^2 A'|\le
C\kappa \Bigl(\ell^{-\frac{5}{2}}+
\min\bigl(\beta^{\frac{3}{2}}h^{\frac{1}{2}}\ell^{-\frac{1}{2}},\,
\beta^{\frac{1}{2}}\ell^{-\frac{7}{4}}\bigr)\Bigr)|\log \ell/\ell_*| \\[3pt]
+
C\Bigl((\kappa \beta )^{\frac{10}{9}}h ^{\frac{4}{9}}\ell^{-\frac{19}{18}}+
(\kappa \beta )^{\frac{4}{3}}h ^{\frac{2}{3}}\ell^{-\frac{5}{6}}\Bigr)
|\log \ell/\ell_*|^K+ C\kappa^{\frac{1}{2}}\ell^{-\frac{5}{2}}.
\label{27-5-14}
\end{multline}

The same arguments work for $\beta_1h_1\ge 1$ but now we need to replace
$|\log h_1|$ by $|\log \beta_1|$ which however is also $\asymp \ell/\ell_*$ as $\beta h\le 1$.

After this estimate is proven we can remove the last term in the right-hand expression and we arrive to

\begin{proposition}\label{prop-27-5-6}
In the framework of Proposition~\ref{prop-27-5-2}
\begin{multline}
|\partial^2 A'|\le
C\kappa \Bigl(\ell^{-\frac{5}{2}} +
\min\bigl(\beta^{\frac{3}{2}}h^{\frac{1}{2}}\ell^{-\frac{1}{2}},\,
\beta^{\frac{1}{2}}\ell^{-\frac{7}{4}}\bigr)\Bigr)|\log \ell/\ell_*|\\[3pt]
+
C\Bigl((\kappa \beta )^{\frac{10}{9}}h ^{\frac{4}{9}}\ell^{-\frac{19}{18}}+
(\kappa \beta )^{\frac{4}{3}}h ^{\frac{2}{3}}\ell^{-\frac{5}{6}}\Bigr)
|\log \ell/\ell_*|^K.
\label{27-5-15}
\end{multline}
\end{proposition}

\subsection{Estimates to a minimizer: exterior zone}
\label{sect-27-5-2-3}

Let us estimate $|\partial^2 A'|$ as $\ell\ge \ell^*$. Observe that
\begin{equation}
A'_j(x) = -\frac{\kappa h^2}{4\pi} \int |x-y|^{-1} \Phi_j (y)\,dy
\label{27-5-16}
\end{equation}
where $\Phi_j$ is given by (\ref{26-2-14xx}). Then $\partial^2 A'$ is expressed via $\Phi_j$ as an integral with a kernel $K(x,y)$ singular as $x=y$ and such that $|K(x,y)|\le c(|x|+|y|)^{-3}$ as $|x-y|\asymp |x|+|y|$. Further, applying representation like in Proposition~\ref{book_new-prop-26-3-9} we can get an extra factor $|y|(|x|+|y|)^{-1}$ upgrading it to $|K(x,y)|\le c|y|(|x|+|y|)^{-4}$.

Then, starting from (\ref{27-5-15}) and iterating (\ref{27-3-46}) we arrive to estimate
\begin{equation*}
|\partial^2 A'|\le
C\kappa \ell^{-4} |\log h|+
C(\kappa \beta)^{\frac{10}{9}}h^{\frac{4}{9}}\ell^{-\frac{32}{9}+\delta}
|\log h|^K
\qquad\text{as\ \ } \ell\ge 1
\end{equation*}
with arbitrarily small $\delta>0$. Furthermore using arguments of the proof of Proposition~\ref{prop-27-5-6} we can make $\delta=0$ thus arriving to

\begin{proposition}\label{prop-27-5-7}
In the framework of Proposition~\ref{prop-27-5-2}
\begin{equation}
|\partial^2 A'|\le
C\kappa \ell^{-4} |\log h|+
C(\kappa \beta)^{\frac{10}{9}}h^{\frac{4}{9}}\ell^{-\frac{32}{9}}|\log h|^K
\qquad\text{as\ \ } \ell\ge 1.
\label{27-5-17}
\end{equation}
\end{proposition}

\begin{remark}\label{rem-27-5-8}
Sure, as $\ell\ge 1$, $\beta _1h_1^{\frac{1}{2}}|\log h_1|^K\le 1$ this estimate could be improved but these improvements would not affect our crucial estimates.
\end{remark}

\section{Trace asymptotics}
\label{sect-27-5-3}

Before proving trace estimates observe

\begin{remark}\label{rem-27-5-9}
\begin{enumerate}[label=(\roman*), fullwidth]
\item\label{rem-27-5-9-i}
All local asymptotics and estimates with with mollification with respect to spatial variables\,\footnote{\label{foot-27-27} Thus trace and $\N$-term asymptotics and $\D$-term estimates.} proven in Sections~\ref{sect-27-3} and~\ref{sect-27-4} with unspecified $\nu \le \epsilon \beta$ remain valid
in the more general framework of smooth non-degenerate external field $A^0(x)$: namely\begin{phantomequation}\label{27-5-18}\end{phantomequation}
\begin{equation}
\|\partial^4 A^0 \|\le C_0 \beta, \qquad
B^0=|\nabla \times A^0|\ge \epsilon_0 \beta.
\tag*{$\textup{(\ref*{27-5-18})}_{1,2}$}\label{27-5-18-*}
\end{equation}
Indeed, we use only $\varepsilon$-approximations with $\varepsilon=h$ or $\varepsilon=h\rho^{-1}$ and we can always change coordinate system so magnetic lines are $(x_1,x_2)=\const$. We leave easy arguments to the reader;

\item\label{rem-27-5-9-ii}
However since we do not have estimates (\ref{27-3-47}) or (\ref{27-4-6})--(\ref{27-4-8}) in this more general framework\footnote{\label{foot-27-28} Even if we believe that these estimates are true. So far we have no need in such generalization.}, we do not have (\ref{27-4-11}), (\ref{27-4-12}) in this framework.
\end{enumerate}
\end{remark}

Now consider the trace term assuming that
\begin{equation}
d \Def \min_{1\le m<m'\le M}|\bar{\y}_m-\bar{\y}_{m'}| \gtrsim 1 .
\label{27-5-19}
\end{equation}
(a) Due to the proofs of Theorem~\ref{book_new-thm-26-3-22} and Proposition~\ref{prop-27-5-2} we can evaluate contribution of zone $\{x:\, |x-\bar{\y}_m| \le \epsilon\}$ provided $\beta \le 1$:
\begin{multline}
|\Tr (H^-_{A,V}\psi_m) - \Tr (H^-_{A,V_m}\psi_m)\\
+h^{-3}\int P_{Bh} (V)\psi_m \,dx -
h^{-3}\int P_{0} (V_m)\psi_m \,dx|
\le C\bigl(h^{-1}+ \kappa |\log \kappa|^{\frac{1}{3}}h^{-\frac{4}{3}}\bigr)
\label{27-5-20}
\end{multline}
where $V_m= z_m|x-\bar{\y}_m|^{-1}$ and $\psi_m$ is supported in
$\{x: |x-\bar{\y}_m|\le \epsilon\}$ and equal $1$ in
$\{x: |x-\bar{\y}_m|\le \frac{1}{2}\epsilon\}$.

Further, we can replace in this estimate
$\Tr (H^-_{A',V_m}\psi_m) +h^{-3}\int P_{0} (V_m)\psi_m $ by
\begin{equation}
\int \Bigl( \int^0_{-\infty} e_{V_m, A'} (x,x,\tau)\,d\tau +
h^{-3} P_{0} (V_m)\Bigr)\psi_m\,dx
\label{27-5-21}
\end{equation}
and we can also replace in the latter expression $\psi_m$ by $1$.

\medskip\noindent
(b) If $\beta \ge 1$ we can apply estimate (\ref{27-5-20}) to zone
$\{x:\, |x-\bar{\y}_m| \le \epsilon b\}$ with $b=\beta^{-\frac{2}{3}}$ scaling
$x\mapsto (x-\bar{\y}_m)b^{-1}$ and $\tau \to \tau $ and
$h\mapsto h_1=hb^{-\frac{1}{2}}$, $\beta\mapsto \beta b^{\frac{3}{2}}=1$,
$\kappa \mapsto \kappa $; now $\psi_m$ is supported in $\{x: |x-\bar{\y}_m|\le \epsilon b\}$ and equals $1$ in
$\{x: |x-\bar{\y}_m|\le \frac{1}{2}\epsilon b\}$ and the right-hand expression of (\ref{27-5-20}) becomes
\begin{equation}
C b^{-1} \bigl(h_1^{-1}+
\kappa |\log \kappa|^{\frac{1}{3}}h_1^{-\frac{4}{3}}\bigr) =
C \bigl(\beta^{\frac{1}{3}} h ^{-1}+
\kappa |\log \kappa|^{\frac{1}{3}}\beta^{\frac{2}{9}} h ^{-\frac{4}{3}}\bigr).
\label{27-5-22}
\end{equation}
As $\beta \ge 1$ consider contribution of zone
$\{x:\, \epsilon_0 b\le |x-\bar{\y}_m|\le \epsilon_0\}$ where due to assumption (\ref{27-5-19})  non-degeneracy condition (\ref{27-3-60}) is automatically satisfied after rescaling; namely before rescaling it is
\begin{equation}
\min_j |V-2 j\beta h|+|\nabla V|\ell \asymp \zeta^2.
\label{27-5-23}
\end{equation}
Contribution of $\ell$-element in this zone does not exceed
\begin{gather}
C\zeta^2 \bigl(h_1^{-1}+ h_1^{-\frac{1}{3}}\nu ^{\frac{4}{3}}\bigr)
\label{27-5-24}\\
\shortintertext{with}
\nu = \sup _{|x-\bar{\y}_m|\asymp \ell} |\partial^2 A'| \ell^2\zeta^{-1}
\label{27-5-25}\\
\shortintertext{and plugging (\ref{27-5-4}) into (\ref{27-5-24}) we get}
C \bigl(\ell^{-\frac{1}{2}} h ^{-1}+
(\kappa \beta )^{\frac{40}{27}}h^{\frac{7}{27}}\ell^{\frac{59}{54}}
|\log (h \ell^{-\frac{1}{2}}) |^K\bigr)\notag\\
\shortintertext{which sums to}
C \bigl(\beta^{\frac{1}{3}} h ^{-1}+
(\kappa \beta )^{\frac{40}{27}}h^{\frac{7}{27}} |\log h |^K\bigr)
\label{27-5-26}
\end{gather}
which obviously does not exceed (\ref{27-5-22}). Thus after scaling \footnote{\label{foot-27-29} $x\mapsto Z^{\frac{1}{3}} x$,
$\tau\mapsto Z^{\frac{4}{3}} \tau $, $1\mapsto h= Z^{-\frac{1}{3}}$,
$B\mapsto \beta =Z^{-1}$, $\alpha \mapsto \kappa =\alpha Z$; recall hat $\beta h\lesssim 1\iff B\lesssim Z^{-\frac{4}{3}}$.} we arrive to

\begin{proposition}\label{prop-27-5-10}
Let $V=W^\TF_B+\lambda$ be Thomas-Fermi potential with $N\le Z$,
$N\asymp Z_1\asymp Z_2 \asymp \ldots \asymp Z_M$, $B\lesssim Z^{\frac{4}{3}}$ and
\begin{equation}
|\y_m-\y_{m'}|\ge d \gtrsim Z^{-\frac{1}{3}} \qquad \forall 1\le m<m'\le M.
\label{27-5-27}
\end{equation}
Then as $\psi_m$ is supported in $\epsilon r^*$-vicinity of $\y_m$
\begin{multline}
|\Tr (H^-_{A,V}\psi_m) - \Tr (H^-_{A,V_m}\psi_m)\\
+\int P_{B} (V)\psi_m \,dx - \int P_{0} (V_m)\psi_m \,dx|
\label{27-5-28}
\end{multline}
does not exceed $CQ_0$ with
\begin{equation}
Q_0\Def \left\{\begin{aligned}
&\bigl(Z^{\frac{5}{3}}+
\alpha |\log (\alpha Z)|^{\frac{1}{3}}Z^{\frac{25}{9}}\bigr)
&&\text{as\ \ } B\le Z,\\[3pt]
&\bigl(B^{\frac{1}{3}} Z^{\frac{4}{3}} +
\alpha |\log (\alpha Z)|^{\frac{1}{3}}B^{\frac{2}{9}}Z^{\frac{23}{9}}\bigr)
&&\text{as\ \ } Z \le B\le Z^{\frac{4}{3}}.
\end{aligned}\right.
\label{27-5-29}
\end{equation}
Furthermore, as $B\le Z$ expression \textup{(\ref{27-5-28})} does not exceed
\begin{equation}
C\Bigl(Z^{\frac{5}{3}}[Z^{-\delta}+(BZ^{-1})^{\delta}+ (dZ^{\frac{1}{3}})^{-\delta}]+
\alpha |\log (\alpha Z)|^{\frac{1}{3}}Z^{\frac{25}{9}}\Bigr).
\label{27-5-30}
\end{equation}
\end{proposition}

Here improved estimate (\ref{27-5-30}) can be proven by our standard propagation arguments.

Furthermore let us consider the \emph{regular exterior zone\/}
 $\{x:\, \epsilon_0 r^*\le |x-\y_m|\le \epsilon \bar{r}\}$ with
$\bar{r}\Def \min \bigl(B^{-\frac{1}{4}},\, (Z-N)_+^{-\frac{1}{3}}\bigr)$. Then due to Thomas-Fermi equation $W^\TF_B+\lambda$ satisfies here non-degeneracy condition (\ref{27-3-65}) after rescaling and $\zeta= \ell^{-2}$, $h_1=\ell(x)$, $\beta_1=B\ell^3$.

Then the contribution of $\ell$-element in this zone does not exceed (\ref{27-5-25}) as long as $\beta_1h_1\le 1$, $h_1\le 1$ i.e.
$\ell (x)\le \min (\bar{r},1)$ and due to (\ref{27-5-24}) this contribution does not exceed
$C\zeta^2 \bigl( h_1^{-1} +h_1^{-\frac{1}{3}}\nu^{\frac{4}{3}}\bigr)$ where $\nu_1$ is estimate for $|\partial ^2 A'|$ multiplied by $\zeta^{-1}\ell^2$:
\begin{equation}
\nu_1 = \kappa |\log h| + (\kappa\beta)^{\frac{10}{9}} h^{\frac{4}{9}}\ell^{\frac{4}{9}}|\log h|^K.
\label{27-5-31}
\end{equation}
Then calculating $\nu_1$ and plugging it and $h_1=h\ell$ into
$C\zeta^2 \bigl( h_1^{-1} +h_1^{-\frac{1}{3}}\nu^{\frac{4}{3}}\bigr)$ one can see easily that here all terms contain $\ell$ in the negative powers.

Then summation by $\ell$ results in the same expression calculated as $\ell=r^*$ and one can observe easily that it does not exceed (\ref{27-5-29}), (\ref{27-5-30}) as $B\le Z$, $Z\le B\le Z^{\frac{4}{3}}$ respectively. One can see easily that dealing with terms
$C\zeta^2 \times \mu^3 \beta^{-2}h^{-2}$ due to (\ref{27-3-61}) and
$C\zeta^2 \nu|\log h|$ (see Propositions~\ref{prop-27-3-13} and~\ref{prop-27-3-15}) leads to smaller expressions.

Furthermore, using standard propagation arguments one can upgrade (\ref{27-5-29}) to (\ref{27-5-30}). Therefore we conclude that

\begin{claim}\label{27-5-32}
Proposition~\ref{prop-27-5-10} remains true for $\psi_m$ supported in the zone $\{x:\, |x-\y_m|\le \epsilon \min(\bar{r},1)\}$.
\end{claim}

Consider now contribution of the \emph{boundary zone\/}
$\{x:\,\ell(x)\ge \min(\bar{r},1)\}$.

\medskip\noindent
(a) Let us start from more difficult and interesting case $B\ge 1$ assuming first that $Z=N$. Rescale this zone first $x\mapsto xB^{\frac{1}{4}}$, $\tau\mapsto \tau B^{-1}$, then we have
$h_1= B^{-\frac{1}{4}}$, $\beta_1=B^{\frac{1}{4}}$.
 Observe that
\begin{claim}\label{27-5-33}
After this rescaling a rescaled magnetic field satisfies
$|\partial^2 A'|\le \nu_1$ where $\nu_1$ is given by (\ref{27-5-31}) as $\ell=\bar{r}$.
\end{claim}

As we know from Subsection~\ref{book_new-sect-25-5-1} after first scaling there exists scaling function $\gamma$ such that
\begin{gather}
|\partial^\alpha V|\le C\gamma^{4-|\alpha|}\qquad |\alpha |\le 4,
\label{27-5-34}\\[2pt]
V\asymp \gamma^4,\qquad |\partial^2 V|\asymp \gamma^2
\label{27-5-35}
\end{gather}
and therefore we can use a $\gamma$-admissible partition. Then scaling again $x\mapsto x\gamma^{-1}$, $\tau \mapsto \tau \gamma^{-4}$,
$h_1\mapsto h_2=h\gamma^{-3}$, $\beta_1\mapsto \beta_2=\beta_1\gamma^{-1}$ and $\nu_1\mapsto \nu_2=\nu_1$ we see that non-degeneracy assumption (\ref{27-3-65}) is fulfilled and therefore according to Proposition~\ref{prop-27-4-3}\ref{prop-27-4-3-ii} the contribution of $\gamma$-element to the remainder does not exceed
$CB\gamma^4\beta_2 \bigl(1+ h_2^{\frac{2}{3}}\nu_2^{\frac{4}{3}}\bigr)$ because now $\beta_2h_2\ge 1$ and the total contribution of such elements does not exceed
\begin{equation}
CB \int
\beta_2\bigl(1 + h_2^{\frac{2}{3}}\nu_2^{\frac{4}{3}}\bigr)\gamma^{-3} \,dx
\label{27-5-36}
\end{equation}
with integral taken over zone
$\{x:\,\gamma(x)\ge \bar{\gamma}\Def h_1^{\frac{1}{3}}\}$ where $h_2\le 1$. Plugging $\beta_2$, $h_2$ and $\nu_2=\nu_1$ we get in the second term $\gamma^{-2}$ which is not good. Let us apply Remark~\ref{rem-27-5-9}. Recall that $A'(x)$ is a solution of the Laplace equation and therefore
$A'(x)= \int |x-y|^{-1} F(y)\, dy$; then $A'(x)$ with $x\in B(z,\gamma (z))$ can be decomposed into the sum of two terms; the first one is given by integral over $\{y: |y-z|\ge \frac{3}{2}\gamma(z)\}$ and therefore is smooth and could be included in $A^0(x)$ while the second is given by integral over
$\{y: |y-z|\le 2\gamma(z)\}$ and could be estimated by (\ref{27-3-46}) with $\beta$, $h$, $\kappa$ and $\nu$ replaced by $\beta_2$, $h_2$,
$\kappa_2 =\kappa_1\gamma^5$ and $\nu_1$ and then one can see easily that $\nu_2\le \nu_1\gamma^{\frac{3}{2}+\delta}$. Therefore

\begin{remark}\label{rem-27-5-11}
In the boundary zone calculating trace term, $\N$- and $\D$-terms one can take $\nu_2= \nu_1\gamma^{\frac{3}{2}+\delta}$ with $\delta>0$.
\end{remark}

Plugging this improved $\nu_2$ into (\ref{27-5-36}) we get everywhere $\gamma$ in the positive power and therefore (\ref{27-5-36}) does not exceed integrand as $\gamma=1$, i.e. $CB\beta_1 \bigl(1+h_1^{\frac{2}{3}}\nu_1^{\frac{4}{3}}\bigr)$
which we already got when estimating the contribution of the regular zone.

\medskip\noindent
(b) In the zone $\{x:\, \gamma(x)\le \bar{\gamma}\}$ we just reset $\gamma=\bar{\gamma}$ and as $h_2\asymp 1$ we do not need any non-degeneracy condition here so its contribution does not exceed $CB\beta_1$. Therefore we arrive to Proposition~\ref{prop-27-5-12}\ref{prop-27-5-12-i} as $Z=N$ and $B\ge 1$.

\medskip\noindent
(c) In the case $B\le 1$ we need no non-degeneracy assumption in zone $\{x:\,\ell(x)\gtrsim 1\}$ as $h_1=1$; Proposition~\ref{prop-27-5-12}\ref{prop-27-5-12-i} has been proven in this case as well.

\medskip\noindent
(d) Explore now case $N<Z$. Then eventually we will need to take $V=W^\TF_B+\lambda$ ' where $\lambda$ is a chemical potential. In this case the same arguments hold provided
$B\gamma ^4 \le |\lambda|\implies \gamma\le h_1^{\frac{1}{3}}$ which is equivalent to $(Z-N)_+\le B^{\frac{5}{12}}$ since in this case
$-\lambda\asymp (Z-N)_+\bar{r}^{-1}=(Z-N)_+ B^{\frac{1}{4}}$.

\medskip\noindent
(e) Further, as $M=1$ we do not need assumption $N=Z$ as we can always refer to non-degeneracy condition (\ref{27-5-23}) which is fulfilled and we arrive to Proposition~\ref{prop-27-5-12}\ref{prop-27-5-12-ii}. Indeed, as in Section~\ref{book_new-sect-25-5} we do not partition further elements where this condition is fulfilled. We arrive to Proposition~\ref{prop-27-5-12}\ref{prop-27-5-12-i}, \ref{prop-27-5-12-ii} below.

\medskip\noindent
(f) Furthermore, consider case $M\ge 2$ and
$B^{\frac{5}{12}}\le (Z-N)_+ \le B^{\frac{3}{4}}$. Then
$B\bar{\gamma}^4=(Z-N)_+ B^{\frac{1}{4}}$ and
$\bar{\gamma}=(Z-N)_+^{\frac{1}{4}} B^{-\frac{3}{16}}$.

In this case contribution of $\bar{\gamma}$-element (in the excess what was prescribed before) does not exceed term (\ref{27-3-66}) multiplied by $\zeta^2$. Note that two last terms in (\ref{27-3-66}) are not new. Meanwhile the first term there (i.e. $C\beta \nu^{\frac{6}{5}}h^{\frac{2}{5}}$) becomes $CB\bar{\gamma}^4 \times B^{\frac{1}{4}}\bar{\gamma}^{-1}\times \bar{\nu}^{\frac{6}{5}} \times (B^{-\frac{1}{4}}\bar{\gamma}^{-3})^{\frac{2}{5}}$ and after multiplication by $\bar{\gamma}^{-2}|\log \bar{\gamma}$ it has $\bar{\gamma}$ in the negative degree, so it does not exceed the same value as $\bar{\gamma}=B^{-\frac{1}{12}}$. After easy but tedious calculations one can see that it is less than (\ref{27-5-22}).

This leaves us with the second term in (\ref{27-3-66})
(i.e. $C\beta h^{-\frac{1}{2}}$) which becomes
$CB \bar{\gamma}^4 \times B^{\frac{1}{4}} \bar{\gamma}^{-1} \times
B^{\frac{1}{8}}\bar{\gamma}^{\frac{3}{2}}$ and after multiplication by $\bar{\gamma}^{-2}|\log \bar{\gamma}$ we get
\begin{equation}
CB^{\frac{11}{8}} \bar{\gamma}^{\frac{5}{2}}|\log \bar{\gamma}| \asymp
CQ' \Def CB^{\frac{29}{32}} (Z-N)_+^{\frac{5}{8}}
\bigl(1+ |\log (Z-N)_+ B^{-\frac{3}{4}}|\bigr).
\label{27-5-37}
\end{equation}

\smallskip\noindent
(g) Finally, as $(Z-N)_+\ge B^{\frac{3}{4}}$ we again end up with the term
$C\zeta^2 \beta h^{-\frac{1}{2}}$ this time with
$\zeta^2= (Z-N)_+^{\frac{4}{3}}$, $h=(Z-N)_+^{-\frac{1}{3}}$,
$\beta = B(Z-N)_+^{-1}$ (now $\bar{\gamma}\asymp 1$); so we arrive to
\begin{equation}
CQ''\Def CB(Z-N)_+^{\frac{1}{2}}.
\label{27-5-38}
\end{equation}
We arrive to Proposition~\ref{prop-27-5-12}\ref{prop-27-5-12-iii}, \ref{prop-27-5-12-iv} below.

\begin{proposition}\label{prop-27-5-12}
Let $V=W^\TF_B+\lambda$ be Thomas-Fermi potential with $N\le Z$,
$N\asymp Z_1\asymp Z_2 \asymp \ldots \asymp Z_M$, $B\lesssim Z^{\frac{4}{3}}$ and chemical potential $\lambda$. Let assumption \textup{(\ref{27-5-27})} be fulfilled. Then
\begin{enumerate}[fullwidth, label=(\roman*)]
\item\label{prop-27-5-12-i}
As $Z=N$ ($\lambda=0$) the trace remainder
\begin{multline}
|\Tr (H^-_{A,V}) + \int P_{Bh} (V) \,dx -\\
\sum_{1\le m\le M} \Bigl(\Tr (H^-_{A,V_m}\psi_m)-\int P_{0} (V_m)\psi_m \,dx\Bigr)|
\label{27-5-39}
\end{multline}
does not exceed $CQ_0$ with $Q_0$ defined by \textup{(\ref{27-5-29})}; the same is true as $(Z-N)_+\le B^{\frac{5}{12}}$;

\item\label{prop-27-5-12-ii}
As $M=1$ the trace remainder \textup{(\ref{27-5-39})} also does not exceed $CQ_0$;

\item\label{prop-27-5-12-iii}
As $M\ge 1$ and $B^{\frac{5}{12}}\le (Z-N)_+\le B^{\frac{3}{4}}$ the trace remainder \textup{(\ref{27-5-39})} does not exceed $C(Q_0 +Q')$ with $Q'$ defined by \textup{(\ref{27-5-37})};

\item\label{prop-27-5-12-iv}
As $M\ge 1$ and $(Z-N)_+^{\frac{4}{3}}\ge B$ the trace remainder \textup{(\ref{27-5-39})} does not exceed $C(Q_0 +Q'')$ with $Q''$ defined by \textup{(\ref{27-5-38})}.
\end{enumerate}
\end{proposition}

\begin{remark}\label{rem-27-6S-12}
\begin{enumerate}[label=(\roman*), fullwidth]
\item\label{rem-27-5-13-i}
As $B\le Z$ expression (\ref{27-5-29}) could be upgraded to (\ref{27-5-30});

\item\label{rem-27-5-13-ii}
Terms (\ref{27-5-38}) and (\ref{27-5-39}) are rather superficial: they do not depend on $\alpha$ and they were not present in Chapter~\ref{book_new-sect-25}. Indeed, using more precise arguments of Chapter~\ref{book_new-sect-25} one can get rid off them, at least for sufficiently small $\alpha$.

However in the upper estimate we will need to deal with $\D$-term as well and this would give us a far larger error.
\end{enumerate}
\end{remark}

\section{Endgame}
\label{sect-27-6-4}

\subsection{Main theorem: $M=1$}
\label{sect-27-6-4-1}

As $M=1$ we almost immediately arrive to the following statement which we formulate in ``rescaled'' terms:

\begin{theorem}\label{thm-27-5-14}
Let $V=W^\TF_B+\lambda$ be a Thomas-Fermi potential as $B\le Z^{\frac{4}{3}}$, $N\asymp Z$ and $M=1$. Then
\begin{multline}
\cE^*_0+ 2 S(\alpha Z)Z^2 - CZ^{\frac{5}{3}}\bigl(1+\alpha B \bigr)\le
\E^*_\alpha \\[3pt]
\le
\cE^*_\alpha+ 2 S(\alpha Z)Z^2 +
C\bigl(Z^{\frac{5}{3}}+\alpha B^2 Z^{\frac{1}{3}} \bigr)
\label{27-5-40}
\end{multline}
where $S(\alpha Z)Z^2$ is a Scott correction term derived in Section~\ref{book_new-sect-26-3}.
\end{theorem}

\begin{proof}[Proof: Estimate from above]\label{proof-27-5-14-i}
We already know from (\ref{27-5-20}) that after standard rescaling for any magnetic field $A'$ satisfying the same estimates as a minimizer of
$\E_\kappa (A')$
\begin{multline}
 \Tr (H^-_{A,V}) +\kappa ^{-1}h^{-2}\|\partial A'\|^2\\[3pt]
\le - h^{-3}\int P_{Bh}(V)\, dx +
\underbracket{\textup{(\ref{27-5-21})} + \kappa ^{-1}h^{-2}\|\partial A'\|^2}+
\textup{(\ref{27-5-22})}.
\label{27-5-41}
\end{multline}
Here the left-hand expression is $\E_\kappa (A')\ge \E^*_\kappa$ and in expression (\ref{27-5-21}) we take $\psi_1=1$.

Let us pick up $A'$ which is a minimizer for the Coulomb potential\linebreak $V_1(x)=z_1|x-\y_1|^{-1}$ without any external magnetic field\footnote{\label{foot-27-30} Actually since for Coulomb potential trace is infinite we take potential $V_1(x)+\tau$ with $\tau <0$, establish estimates and then tend $\tau\to -0$.}. Then we can replace selected sum of two terms by $2h^{-2}S(\kappa z)z^2$ (in virtue of Subsubsection~\ref{book_new-sect-26-3-5}.1 an error does not exceed (\ref{27-5-22})).

Unfortunately $A'$ is not a minimizer for a local problem; however we can replace $P_{Bh} (V)$ by
\begin{equation}
P_{\beta h}(V) + \partial _{\beta} P_{\beta h}(V) \cdot \Phi , \qquad
\Phi =\partial_1 A'_2-\partial_2 A_1'
\label{27-5-42}
\end{equation}
and if we apply partition then on each partition element an error in (\ref{27-5-41}) does not exceed $C\kappa_1 \zeta^2\beta_1 h_1^{-\frac{1}{2}} =
C\kappa \beta h^{-\frac{1}{2}}\zeta^{\frac{5}{2}} \ell ^{\frac{5}{2}}$. Indeed, we know that for a minimizer
$\|\partial A'\|\le C(\kappa + \kappa^{\frac{1}{2}} h^{\frac{1}{2}})$ (also see in the proof of the estimate from below).

Summation over partition results in $C\kappa \beta^2$ and we arrive to
\begin{equation}
- h^{-3}\int P_{\beta h}(V)\, dx -
 h^{-3}\int \partial_\beta P_{\beta h}(V)\cdot \Phi \, dx +
 C(h^{-1}+\kappa \beta^2 )
\label{27-5-43}
\end{equation}
in the right-hand expression where the first term is $\cE^*_0$\,\footnote{\label{foot-27-31} I.e. Thomas-Fermi energy calculated as $A'=0$.}.

The second term is rather unpleasant as we cannot estimate it by anything better than $C\beta h^{-1}$ (see in the estimate from below) but \emph{here} we have a trick\footnote{\label{foot-27-32} Which unfortunately we cannot repeat in estimate from below.}: we replace $A'$ by $-A'$ which is also a minimizer for the same Coulomb potential $V_1$ without external magnetic field. Then $\Phi$ and the second term change signs and since nothing else happens we can skip the second term which concludes the proof of the upper estimate.

Scaling back we arrive to the upper estimate in (\ref{27-5-40}). \end{proof}\vspace{-10pt}

\begin{proof}[Proof: Estimate from below]\label{proof-27-5-14-ii}
 Again from (\ref{27-5-20}) we already know that for a minimizer $A'$ of $\E_\kappa (A')$ estimate (\ref{27-5-41}) could be reversed
\begin{multline}
 \Tr (H^-_{A,V}) +\kappa ^{-1}h^{-2}\|\partial A'\|^2\\[3pt]
\ge - h^{-3}\int P_{Bh}(V)\, dx +
\underbracket{\textup{(\ref{27-5-21})} + \kappa ^{-1}h^{-2}\|\partial A'\|^2}-
\textup{(\ref{27-5-22})}.
\label{27-5-44}
\end{multline}
Here the left-hand expression is $\E_\kappa (A')= \E^*_\kappa$ and the selected sum of two terms can be estimated from below by $2h^{-2}S(\kappa z)z^2$.

Again $A'$ is not a minimizer for a local problem; however we can replace
$P_{Bh} (V)$ by $P_{\beta h}(V)$ with an error not exceeding
\begin{multline*}
Ch^{-3} \int |\partial _{\beta} P_{\beta h}(V)|\cdot |\partial A'|\,dx +
Ch^{-3} \int |\partial^2 _{\beta} P_{\beta h}(V)|\cdot |\partial A'|^2\,dx\\
\le Ch^{-3} \|\partial _{\beta} P_{\beta h}(V)\| \cdot \|\partial A'\|+
Ch^{-3}\int |\partial^2 _{\beta} P_{\beta h}(V)|\,dx \times \|\partial A'\|^2
\end{multline*}
with the right-hand expression not exceeding
\begin{equation}
C\beta h^{-1} |\partial A'\| + Ch^{-1} \|\partial A'\|^2.
\label{27-5-45}
\end{equation}
However we already know that
$\E^*_{2\kappa} \ge \cE^*_0 - C(\kappa h^{-2}+h^{-1})$ and therefore since
$\E_\kappa (A')= \E_{2\kappa} (A')+ (2\kappa h^2)^{-1} \|\partial A'\|^2$ we conclude that
\begin{equation}
\|\partial A'\| \le C(\kappa + \kappa^{\frac{1}{2}} h^{\frac{1}{2}});
\label{27-5-46}
\end{equation}
then expression (\ref{27-5-45}) does not exceed $C(\kappa \beta +1)h^{-1}$ which concludes the proof of the lower estimate.

Since $A'$ is a minimizer in the presence of the external field, we cannot replace $A'$ by $-A'$ and thus cannot repeat the trick used in the proof of the upper estimate. Thus in the estimate from below we are left with
$C\kappa \beta h^{-1}$ rather than with $C\kappa \beta^2$.

Scaling back we arrive to the lower estimate in (\ref{27-5-40}).
\end{proof}

\begin{remark}\label{rem-27-5-15}
\begin{enumerate}[label=(\roman*), fullwidth]
\item\label{rem-27-5-15-i}
It is a very disheartening that our estimate deteriorated here. However it may be that indeed the better estimate does not hold due to the \index{entanglement}\emph{entanglement\/} of the singularity and the regular zone via self-generated magnetic field. Still we did not loose Scott correction term as long as $\kappa\beta h\ll 1$;

\item\label{rem-27-5-15-ii}
Recall that in the Section~\ref{book_new-sect-26-3} we already had an entanglement of different singularities which obviously remains with us as $M\ge 2$. Surely both of these entanglement matter only if we are looking for the remainder estimate better than $O(\kappa h^{-2})$. Otherwise we can just pick up $A'=0$ near singularities and Scott correction term $2h^{-2}S(0)z^2$;

\item\label{rem-27-5-15-iii}
The silver lining is that we do not need all these non-degeneracy conditions for such bad estimate and we expect that our arguments would work for $M\ge 2$.
Still for $M\ge 2$ we will need to decouple singularities and to do this we will need to estimate $\|\partial A'\|_{\sL^2(\{\ell (x) \asymp d\})}$ where $d$ is the minimal distance between nuclei (see (\ref{27-5-19})).
\end{enumerate}
\end{remark}

As $\beta \ll 1$, $\kappa \ll h^{\frac{1}{3}}|\log h|^{-\frac{1}{3}}$ we can recover even Schwinger correction term:

\begin{theorem}\label{thm-27-5-16}
Let $V$ be a Thomas-Fermi potential $W^\TF_B+\lambda$ rescaled as $B\ll Z$, $N\asymp Z$ and $M=1$. Let respectively $\beta=BZ^{-1}$, $h=Z^{-\frac{1}{3}}$, and $\kappa=\alpha Z\le \kappa^*$. Then
\begin{multline}
|\E^*_\alpha -\bigl(\cE^*_0+ 2 S(\alpha Z)Z^2 + \Schwinger\bigr) |\\
\le C Z^{\frac{5}{3}}\bigl(Z^{-\delta}+B^{\delta}Z^{-\delta}\bigr) +
 C\alpha |\log \alpha Z|^{\frac{1}{3}}Z^{\frac{25}{9}}
\label{27-5-47}
\end{multline}
where $\Schwinger$ is a Schwinger correction term.
\end{theorem}

\begin{proof}
The proof is standard like in Section~\ref{book_new-sect-26-3}: we invoke propagation arguments in the zone
$\{x:\,(\beta +h)^{\sigma}\le \ell (x)Z^{\frac{1}{3}}\le
(\beta +h)^{-\sigma}\}$.
\end{proof}

\section{$\N$-term asymptotics and $\D$-term estimate}
\label{sect-27-5-5}

Consider now $\N$-terms assuming that $A'$ is a minimizer.

\subsection{Case $M=1$}
\label{sect-27-5-5-1}

Assume first that $M=1$.

\medskip\noindent
(a) Consider first singular and regular zones (but not the boundary zone). Then after rescaling contribution of $\ell$-element to the remainder does not exceed
\begin{equation}
C \bigl(h_1^{-2}+h_1^{-\frac{5}{3}}\nu_1^{\frac{2}{3}}\bigr)
\label{27-5-48}
\end{equation}
and summation over zone $\{x:\,\ell(x)\le 1\}$ results in the value of this as $\ell=1$ i.e.
\begin{equation}
C\bigl(h^{-2}+h^{-\frac{5}{3}}\nu^{*\,\frac{2}{3}}\bigr),\qquad
\nu^*=(\kappa \beta )^{\frac{10}{9}}h^{\frac{4}{9}}|\log h|^K.
\label{27-5-49}
\end{equation}
Recall that $\ell \ge 1$ we have $h_1=h\ell$, $\beta_1=\beta \ell^3$ and
$\nu_1$ is defined by (\ref{27-5-31}) which is sufficient even without invoking Remark~\ref{rem-27-5-9} as all powers of $\ell$ in (\ref{27-5-48}) become negative. Therefore contribution of regular exterior zone also does not exceed (\ref{27-5-49}).

\medskip\noindent
(b) Consider now boundary zone. Let us repeat arguments used in the proof of Proposition~\ref{prop-27-5-12}: contribution of $\gamma$-element does not exceed
$C\beta_2 \bigl(h_2^{-1} + h_2^{-\frac{2}{3}}\nu_2^{\frac{2}{3}}\bigr)$ and plugging $\beta_2=\beta_1\gamma^{-1}$, $h_2=h_1\gamma^{-3}$ and $\nu_2=\nu_1\gamma^{\frac{3}{2}+\delta}$ we get
$C\beta_1 \bigl( h_1^{-1} \gamma^{2} + h_1^{-\frac{2}{3}}\nu_1^{\frac{2}{3}}\gamma^{2+\delta}\bigr)$ and therefore the total contribution of the boundary zone to the remainder does not exceed
\begin{equation}
C\beta_1 \int \bigl(h_1^{-1} \gamma^{-1} + h_1^{-\frac{2}{3}}\nu_1^{\frac{2}{3}}\gamma^{-1+\delta}\bigr)\,dx\asymp
C\bigl(h_1^{-2} |\log h| + h_1^{-\frac{5}{3}}\nu_1^{\frac{2}{3}}\bigr)
\label{27-5-50}
\end{equation}
since $\beta_1=h_1^{-1}$ and from Subsection~\ref{book_new-sect-25-5-1} we also know that
\begin{equation}
\D(\gamma ^{-1+s},\gamma ^{-1+s})\le Cs ^{-2}\qquad \text{for\ \ }s>0.
\tag{\ref*{book_new-25-5-13}}\label{25-5-13xx}
\end{equation}
even in the general case. Here we integrate over
$\gamma\ge \bar{\gamma}=h_1^{\frac{1}{3}}$ while contribution of zone $\{x:\,\gamma(x)\le \bar{\gamma}$ does not exceed $Ch_1^{-2}$.

If $\bar{\gamma}\ge h_1^{\frac{1}{3}}$ (i.e. $(Z-N)_+\gtrsim B^{\frac{5}{12}}$ we invoke in zone $\{x:\,\gamma(x)\le \bar{\gamma}$ strong non-degeneracy assumption (\ref{27-3-60}) fulfilled as $M=1$ and estimate its contribution by the same expression (\ref{27-5-50}) albeit without logarithmic term. Similar arguments work also as $(Z-N)_+\ge B^{\frac{3}{4}}$.

\medskip\noindent
(c) In (\ref{27-5-50}) logarithmic factor is mildly annoying. However we can get rid off it using our standard propagation arguments like in Subsection~\ref{book_new-sect-25-6-3}; we leave easy details to the reader.

Note that so far in the estimate we also have $P'_{Bh}(V)$ rather than
$P'_{\beta h}$; observe however that in the virtue of non-degeneracy assumption (\ref{27-3-60}) fulfilled as $M=1$ the error when we replace $P'_{Bh}(V)$ by $P'_{\beta h}(V)$ does not exceed $C\beta_1 h_1^{-1}$ on the regular elements and $Ch_2^{-2}$ on the border elements and summation in both cases results in
$O(\beta h^{-1}\bar{\ell}^2)=O(\beta^{\frac{1}{2}}h^{-\frac{3}{2}})$. Therefore in contrast to the trace estimate there is no deterioration.

Scaling back, we arrive to estimate (\ref{27-5-51}) below.

\medskip\noindent
(d) In the framework of Theorem~\ref{thm-27-5-16} we invoke our standard propagation arguments in the zone
$\{x:\,(\beta +h)^{\sigma}\le \ell (x)\le (\beta +h)^{-\sigma}\}$; considering $\D$-terms we invoke these arguments as both elements in the pair belong to this zone. Again, we leave easy details to the reader. Scaling back, we arrive to estimate (\ref{27-5-53}) below.

\medskip\noindent
(e) We deal with $\D$-term in our usual manner considering double partition and different pairs of partition elements--disjoint when we just apply above arguments as the kernel $|x-y|^{-1}$ is smooth there and non-disjoint when we appeal to the local estimates of $\D$-term. We we leave easy details to the reader.

\begin{proposition}\label{prop-27-5-17}
Let $V=W^\TF_B+\lambda$ be a Thomas-Fermi potential as $B\le Z^{\frac{4}{3}}$, $N\asymp Z$ and $M=1$. Then
\begin{enumerate}[label=(\roman*), fullwidth]
\item\label{prop-27-5-17-i}
\begin{gather}
|\int \bigl( \tr e(x,x,0) -P'_{B}(V(x)\bigr)\,dx|\le C R
\label{27-5-51}\\
\shortintertext{and}
\D \bigl( \tr e(x,x,0) -P'_{B}(V(x),\,\tr e(x,x,0) -P'_{B}(V(x)\bigr)
\le C Z^{\frac{1}{3}}R^2
\label{27-5-52}\\
\shortintertext{with}
R=R_0\Def Def Z^{\frac{2}{3}}+ Z^{\frac{5}{9}}\nu^{*\,\frac{2}{3}},\qquad
\nu^*=(\alpha B )^{\frac{10}{9}}Z^{-\frac{4}{27}}|\log Z|^K
\label{27-5-53}
\end{gather}

\item\label{prop-27-5-17-ii}
In the framework of Theorem~\ref{thm-27-5-16} estimates \textup{(\ref{27-5-51})} and \textup{(\ref{27-5-52})} hold with
\begin{equation}
R\Def C Z^{\frac{2}{3}}\bigl(Z^{-\delta}+ B^{\delta}Z^{-\delta}+(\alpha Z)^{\delta}\bigr).
\label{27-5-54}
\end{equation}
\end{enumerate}
\end{proposition}

\subsection{Case $M\texorpdfstring{\ge}{\textge} 2$}
\label{sect-27-5-5-2}

Consider now $M\ge 2$. In comparison with case $M=1$ we should get some extra terms because
\begin{enumerate}[fullwidth, label=(\alph*)]
\item
First, in the regular zone as $\ell(x)\ge \epsilon d$ strong non-degeneracy assumption (\ref{27-3-60}) is replaced by strong assumption (\ref{27-3-65})

\item
In the boundary zone as $\gamma (x)\le \bar{\gamma}$ (with $\bar{\gamma}\ge h_1^{\frac{1}{3}}$) there is no non-degeneracy assumption at all.
\end{enumerate}

Consider a regular zone first. According to (\ref{27-3-86}) we get an extra term
\begin{equation}
C\beta_1 h_1^{-1} \nu_1^{\frac{1}{2}}\asymp
C\beta h^{-1}(\kappa \beta)^{\frac{5}{9}}\ell^{\frac{20}{9}}|\log h|^K
\label{27-5-55}
\end{equation}
as other extra terms are smaller; summation over regular zone results in the value as $\ell=\bar{r}$; as $\bar{r}=(\beta h)^{-\frac{1}{4}}$ we get
$C\kappa^{\frac{5}{9}} \beta h^{-\frac{14}{9}}|\log h|^K$. Scaling back we arrive to
\begin{equation}
CR''=C\alpha^{\frac{5}{9}} B Z^{\frac{2}{27}}|\log Z|^K.
\label{27-5-56}
\end{equation}
Using $\gamma$-partition in the boundary zone and plugging
$h_2=h_1\gamma^{-3}$, $\beta_2= \beta_1\gamma^{-1}$, $\nu_2=\nu_1\gamma^{\frac{3}{2}+\delta}$ (see Remark~\ref{rem-27-5-11}) and using (\ref{27-5-52}) we prove easily that as $\bar{\gamma}=h_1^{\frac{1}{3}}$ (i.e. $(Z-N)_+\le B^{\frac{5}{12}}$) the contribution of the boundary zone is smaller than smaller than $CR''$.

On the other hand, if $ B^{\frac{5}{12}}\le (Z-N)_+\le B^{\frac{3}{4}}$ we need to add $C\beta_2 h_2^{-\frac{3}{2}}\bar{\gamma}^{-2}(1+|\log \bar{\gamma}|)$. Plugging $h_2=h_1\bar{\gamma}^{-3}$, $\beta_2= \beta_1\bar{\gamma}^{-1}$,
$\bar{\gamma}= (Z-N)_+^{\frac{1}{4}}B^{-\frac{3}{16}}$ and scaling back we arrive to
\begin{equation}
CR'''=C(Z-N)_+^{\frac{3}{8}}B^{\frac{11}{32}}
\bigl(1+|\log (Z-N)_+B^{-\frac{3}{4}}|\bigr).
\label{27-5-57}
\end{equation}

Finally as $(Z-N)_+\ge B^{\frac{3}{4}}$ we get $C\beta_1h_1^{-\frac{3}{2}}$ with $\ell=(Z-N)_+^{-\frac{1}{3}}Z^{\frac{1}{3}}$ i.e.
\begin{equation}
CR'''=C(Z-N)_+^{-\frac{1}{2}}B.
\label{27-5-58}
\end{equation}

Thus we arrive to Proposition~\ref{prop-27-5-18}\ref{prop-27-5-18-i}. The similar arguments work for $\D$-term and we arrive to Proposition~\ref{prop-27-5-18}\ref{prop-27-5-18-i}

\begin{proposition}\label{prop-27-5-18}
Let $V=W^\TF_B+\lambda$ be a Thomas-Fermi potential as $B\le Z^{\frac{4}{3}}$, $N\asymp Z$ and $M\ge 2$. Let $d\ge Z^{-\frac{1}{3}}$. Then
\begin{enumerate}[label=(\roman*), fullwidth]
\item\label{prop-27-5-18-i}
Estimate \textup{(\ref{27-5-51})} holds with
$R=R_0+R''$ and $R''$ defined by \textup{(\ref{27-5-56})} as
$(Z-N)_+\le B^{\frac{5}{12}}$, $R=R_0+R''+R'''$ and $R'''$ defined by \textup{(\ref{27-5-57})} as $B^{\frac{5}{12}}\le (Z-N)_+\le B^{\frac{3}{4}}$,
$R=R_0+R'''$ and $R'''$ defined by \textup{(\ref{27-5-58})} as
$(Z-N)_+\ge B^{\frac{3}{4}}$.

Furthermore, as $B\le Z$ estimate \textup{(\ref{27-5-51})} holds with
\begin{equation}
R= Z^{\frac{2}{3}}\bigl[ Z^{-\delta} + B^{\delta}Z^{-\delta}+ (dZ^{\frac{1}{3}})^{-\delta}+(\alpha Z)^{\delta}\bigr].
\label{27-5-59}
\end{equation}

\item\label{prop-27-5-18-ii}
The left-hand expression of \textup{(\ref{27-5-57})} does not exceed
$CZ^{\frac{1}{3}}R_0^2 + CB^{\frac{1}{4}} R^{\prime\prime\,2}$
as $(Z-N)_+\le B^{\frac{5}{12}}$,
$CZ^{\frac{1}{3}}R_0^2 + CB^{\frac{1}{4}} (R''+R''')^2$ and  $R'''$ defined by \textup{(\ref{27-5-57})} as $B^{\frac{5}{12}}\le (Z-N)_+\le B^{\frac{3}{4}}$,
$CZ^{\frac{1}{3}}R_0^2 + (Z-N)_+^{\frac{1}{3}}R^{\prime\prime\prime\,2}$ and $R'''$ defined by \textup{(\ref{27-5-58})} as $(Z-N)_+\ge B^{\frac{3}{4}}$.

Furthermore, as $B\le Z$ the left-hand expression of \textup{(\ref{27-5-57})} does not exceed $CZ^{\frac{1}{3}}R^2$ with $R$ defined by \textup{(\ref{27-5-59})}.
\end{enumerate}
\end{proposition}

\section{More estimates to minimizer}
\label{sect-27-5-6}

Now we want to provide different kinds of estimates to the minimizer as
$\ell(x) \ge Z^{-\frac{1}{3}}$ in the original scale. More precisely, we are looking for
\begin{equation}
\alpha^{-1} \int \phi_r(x) | \partial A'|^2\,dx
\label{27-5-60}
\end{equation}
with $\phi_r$ supported in $\{x:\, \ell(x)\asymp r\}$ as it will appear as an error when we decouple singularities as $M\ge 2$ (in this case we should take $r\asymp d$. Due to equation (\ref{26-2-14xx}) it is $\D$-type term as well: namely, with integral taken over $\bR^3$ it would be equal to
$\alpha Z^{\frac{5}{3}}\D (\phi \Phi_j, \phi \Phi_j)$ calculated in the rescaled coordinates with $\Phi_j$ defined by (\ref{26-2-14xx}); however with the cut-off the integral kernel $|x-y|^{-1}$ needs to be modified. Recall that corresponding Weyl expression is $0$.

First, using (\ref{27-5-16}) and decomposition like in the proof of Proposition~\ref{book_new-prop-26-3-9} we can rewrite (\ref{27-5-60}) as
\begin{equation}
I\Def \int K(z; x, y) \Phi (x) \Phi(y)\,dxdydz
\label{27-5-61}
\end{equation}
multiplied by $\alpha Z^{\frac{5}{3}}$; Here $K(z; x, y)$ is supported in
$\{z:\, \ell(z)\asymp r\}$, singular as $x=z$ and $y=z$ and satisfies
\begin{equation}
|K(z; x, y)|\le |x-z|^{-2}|y-z|^{-2} |x||y| (|x|+r)^{-1}(|y|+r)^{-1}.
\label{27-5-62}
\end{equation}
Here we temporarily replaced $r$ by $Z^{\frac{1}{3}}r$.

Let us make a double $\ell$-admissible partition of unity and consider pairs of elements with $ \ell(x)\asymp r_1$ and $\ell (y)\asymp r_2$. There are three cases:$\le r_1\le \epsilon r$, $r_1\asymp r$ and $r_1\ge c r$ and so for $y$ and we can consider only pure pairs.

\subsection{Case $M=1$}
\label{sect-27-5-6-1}

Assume first that $M=1$. Then contribution of each pair with $r_j\le \epsilon r$ (assuming that they belong to regular zone) does not exceed
\begin{equation}
C \zeta_1 r_1 \bigl(h_1^{-2}+h_1^{-\frac{5}{3}}\nu_1^{\frac{2}{3}}\bigr)\times
\zeta_2 r_2 \bigl(h_2^{-2}+h_2^{-\frac{5}{3}}\nu_2^{\frac{2}{3}}\bigr) \times r^{-3}
\label{27-5-63}
\end{equation}
where $\zeta_j=r_j^{-2}$, $h_j=h r_j$,
$\nu_j=(\kappa \beta)^{\frac{10}{9}}h^{\frac{4}{9}}r_j^{\frac{4}{9}}$ as
$r_j\ge 1$ and $\zeta_j=r_j^{-\frac{1}{2}}$, $h_j=h r_j^{-\frac{1}{2}}$,
$\nu_j=(\kappa \beta)^{\frac{10}{9}}h^{\frac{4}{9}}$ as $r_j\le 1$. Double summation returns its value as $r_1=r_2=1$ i.e.
\begin{equation}
C h^{-4} \bigl(1 +h^{\frac{2}{3}}\nu^{*\,\frac{4}{3}}\bigr) r^{-3}.
\label{27-5-64}
\end{equation}

Further, using Fefferman-de Llave decomposition one can prove easily that contribution of pairs with $r_1\asymp r_2 \asymp r$ (the only case when we have a singular kernel does not exceed (\ref{27-5-63}) calculated as $r_1=r_2=r$ which is decaying function of $r$ and therefore does not exceed (\ref{27-5-64}).

Furthermore, contribution of each pair with $r_j\ge C r$ (assuming that they belong to regular zone) does not exceed
\begin{equation}
C \zeta_1 r_1 ^{-2} \bigl(h_1^{-2}+h_1^{-\frac{5}{3}}\nu_1^{\frac{2}{3}}\bigr)\times
\zeta_2 r_2^{-2} \bigl(h_2^{-2}+h_2^{-\frac{5}{3}}\nu_2^{\frac{2}{3}}\bigr)\times r^3
\label{27-5-65}
\end{equation}
and the double summation returns its value as $r_1=r_2=r$ which is the same as (\ref{27-5-63}) calculated as $r_1=r_2=r$ and again does not exceed (\ref{27-5-64}).

Finally, considering $r_1\asymp r_2\asymp \bar{r}$ we apply if
$(Z-N)_+\le B^{\frac{3}{4}}$ secondary partitions with respect to $x$ and $y$ and using our standard arguments we estimate contribution of this zone by (\ref{27-5-65}) calculated as $r_1=r_2=\bar{r}$.

Therefore we estimated expression (\ref{27-5-61}) by (\ref{27-5-64}). In particular, if $\kappa \beta \le h^{-\frac{17}{20}}|\log h|^{-K}$
i.e. $\alpha B \le Z^{\frac{17}{60}}|\log Z|^{-K}$) then $\nu^*\le h^{-\frac{1}{2}}$ and expression (\ref{27-5-61}) does not exceed $Ch^{-4}r^{-3}$. Plugging $h=Z^{-\frac{1}{3}}$, multiplying by $\alpha Z^{\frac{5}{3}}$ and replacing $r$ by $Z^{\frac{1}{3}} r$ we get $\alpha Z^2r^{-3}$ thus proving
Proposition~\ref{prop-27-5-19}\ref{prop-27-5-19-i}.

On the other hand, as $\kappa \beta \ge h^{-\frac{17}{20}}|\log h|^{-K}$ estimate (\ref{27-5-61}) by (\ref{27-5-64}) could be improved. Indeed, let us apply all the above arguments only as $r_j\ge t$ with $1\le t\le \epsilon r$. Then we get expression (\ref{27-5-63}) as $r_1=r_2=t$ i.e.
\begin{equation}
C h^{-4} \times \nu^{*\frac{4}{3}}h^{\frac{2}{3}} t^{-\frac{128}{27}} r^{-3}
\label{27-5-66}
\end{equation}
where we consider only term possibly exceeding $Ch^{-4}r^{-3}$.

To estimate contribution of zone $\{x,y:\,\ell(x)\le t, \ell(y)\le t\}$ we replace $\Phi_j$ by $(\kappa h^{-2})\Delta A'_j$ and using standard estimate for operator norm in $\sL^2$ we conclude that the corresponding part of
expression (\ref{27-5-61}) does not exceed
$C\kappa ^{-2}h^{-4} \|\partial A'\|^2\times t^3r^{-3}\le Ch^{-4}t^3 r^{-3}$ as
$\|\partial A'\|\le C\kappa$.

Adding to (\ref{27-5-66}) and minimizing by $t\lesssim r$ we get
$Ch^{-4}(\nu^{*\,2}h)^{\frac{54}{209}}r^{-3}$ provided
$r\ge (\nu^{*\,2}h)^{\frac{18}{209}}$.

Plugging $\nu^*$ and $h$, multiplying by $\alpha Z^{\frac{5}{3}}$ and replacing $r$ by $Z^{\frac{1}{3}} r$ we arrive to
Proposition~\ref{prop-27-5-19}\ref{prop-27-5-19-ii}.

\begin{proposition}\label{prop-27-5-19}
Let $V=W^\TF_B+\lambda$ be a Thomas-Fermi potential as $B\le Z^{\frac{4}{3}}$, $N\asymp Z$ and $M=1$. Then
\begin{enumerate}[label=(\roman*), fullwidth]
\item\label{prop-27-5-19-i}
Minimizer satisfies
\begin{gather}
\alpha^{-1} \int_{\{x:\, \ell(x)\ge r\}} | \partial A'|^2\,dx \le
CT_0r^{-3}=C \alpha Z^2 r^{-3}
\label{27-5-67}\\
\shortintertext{as $r\ge r_*=Z^{-\frac{1}{3}}$ holds provided}
\alpha B \le Z^{\frac{17}{60}}|\log Z|^{-K};
\label{27-5-68}
\end{gather}
\item\label{prop-27-5-19-ii}
Otherwise minimizer satisfies
\begin{gather}
\alpha^{-1} \int_{\{x:\, \ell(x)\ge r\}} | \partial A'|^2\,dx \le
CT_0 r^{-3}=C \alpha Z^3 r_*^3 r^{-3}\qquad\text{as\ \ } r\ge r_*
\label{27-5-69}\\
\shortintertext{with}
r_* \Def (\alpha B)^{\frac{40}{209}}Z^{-\frac{81}{209}}|\log Z|^K
\gtrsim Z^{-\frac{1}{3}}.
\label{27-5-70}\end{gather}
\end{enumerate}
\end{proposition}

\begin{remark}\label{rem-27-5-20}
\begin{enumerate}[label=(\roman*), fullwidth]
\item\label{rem-27-5-20-i}
Assumption (\ref{27-5-68}) means exactly that $r_*\le Z^{-\frac{1}{3}}$;

\item\label{rem-27-5-20-ii}
Our usual approach implies that for $B\le Z$ Tauberian error estimate could be slightly improved but it has no implications here because in contrast to
$e(x,x,\tau)$ where the main term in the formal asymptotic decomposition is
$h^{-3}P'_{Bh}$ and the next one is $\asymp \beta^2h^{-2}$, in $\Phi_j$ the main term is $0$ and the next one is $\eta h^{-2}$ with the coefficient $\eta$ depending on $A'$ and trying to calculate it and plug the corresponding term into (\ref{27-5-61}) instead of $\Phi_j$ will certainly result in the identity;

\item\label{rem-27-5-20-iii}
It may happen that $r_*\ge \bar{r}$; then it follows from the proof that $r_*$ should be truncated to $\bar{r}$ in (\ref{27-5-69}). Moreover, estimate
(\ref{27-5-69}) with $r_*=\bar{r}$ holds even as $M\ge 2$ as no non-degeneracy assumption is required.
\end{enumerate}
\end{remark}

\subsection{Case $M\texorpdfstring{\ge}{\textge} 2$}
\label{sect-27-5-6-2}

Assume now that $M\ge 2$ and that (\ref{27-5-19}) is fulfilled. Then we need to take into account excess terms in our estimates. In the regular zone such excess term is
\begin{align}
&C\zeta_1r_1 \beta_1 h_1^{-1}\nu_1^{\frac{1}{2}} \times
\zeta_2r_2 \beta_2 h_2^{-1}\nu_2^{\frac{1}{2}} \times r^{-3}
\qquad\text{as\ \ } r_1\le r, r_2\le r
\label{27-5-71}\\
\shortintertext{and}
&C\zeta_1r_1 ^{-2}\beta_1 h_1^{-1}\nu_1^{\frac{1}{2}} \times
\zeta_2r_2^{-2} \beta_2 h_2^{-1}\nu_2^{\frac{1}{2}} \times r^{3}
\qquad\text{as\ \ } r_1\ge r, r_2\ge r.
\label{27-5-72}
\end{align}
Plugging $\beta_j,h_j$ and $\nu_j$ one observes easily that the former is a growing and the latter is a decaying function of $r_j$ and these expressions coincide as $r_1=r_2=r$. To decouple singularities we need to consider
$r\le\epsilon d$ where $d$ is the minimal distance between singularities; so we will assume this. Observe that extra terms appear only as $r_j\ge \epsilon d$, so we need to consider only (\ref{27-5-72})

Therefore as $d\le \bar{r}$ summation results in expression (\ref{27-5-72}) calculated at $r_1=r_2=d$ which is
\begin{gather}
C\beta^2 h^{-\frac{14}{9}}(\kappa \beta)^{\frac{10}{9}} |\log h|^K d^{-\frac{32}{9}} r^3;
\notag\\
\intertext{which results in the original settings in\footnotemark}
CT'r^3=C(\alpha B)^{\frac{19}{9}}B d^{-\frac{32}{9}} r^3;
\label{27-5-73}
\end{gather}
\footnotetext{\label{foot-27-33} I.e. after we plug $\kappa,\beta,h$, replace $d$ and $r$ by $Z^{\frac{1}{3}}d$ and $Z^{\frac{1}{3}}r$ and multiply by $\alpha Z^{frac{5}{3}}$.}
recall that we plug $\beta$, $h$, $\kappa$,replace $d$ and $r$ by $Z^{\frac{1}{3}}d$ and $Z^{\frac{1}{3}}r$, and multiply by
$\alpha Z^{\frac{5}{3}}$.

Meanwhile using our standard arguments one can prove easily that contribution of the boundary zone is less than this provided $(Z-N)_+\le B^{\frac{5}{12}}$.

On the other hand, as $B^{\frac{5}{12}}\le (Z-N)_+\le B^{\frac{3}{4}}$ an extra contribution of the boundary zone does not exceed
\begin{gather}
C\beta^2h^{-3}\bar{\gamma}^{13}\bar{r}^{-5}r^3 (1+|\log \bar{\gamma}|)^2
\notag\\
\intertext{which results in the original settings in\footref{foot-27-33}}
CT''r^3=C\alpha (Z-N)_+^{\frac{13}{4}}B^{\frac{13}{16}}
\bigl(1+|\log (Z-N)_+B^{-\frac{3}{4}}|\bigr)^2 r^3.
\label{27-5-74}
\end{gather}

Finally, as $ (Z-N)_+\ge B^{\frac{3}{4}}$ an extra contribution of the boundary zone does not exceed
\begin{gather}
C\beta^2h^{-3}\bar{r}^{-5}r^3
\notag\\
\intertext{which results in the original settings in\footref{foot-27-33}}
CT''r^3=C\alpha (Z-N)_+^{\frac{5}{3}}B^{2}.
\label{27-5-75}
\end{gather}

\begin{proposition}\label{prop-27-5-21}
Let $V=W^\TF_B+\lambda$ be Thomas-Fermi potential as
$B\le Z^{\frac{4}{3}}$, $N\asymp Z_1\asymp Z_2\asymp \ldots \asymp Z_M$ and $M\ge 2$. Let $r_*\ll r\ll d \le \epsilon \bar{r}$.

Then expression
\begin{equation}
\alpha^{-1} \int_{\{x:\, \ell(x)\asymp r\}} | \partial A'|^2\,dx
\label{27-5-76}
\end{equation}
does not exceed $C\bigl(T_0 r^{-3}+(T'+T'') r^3\bigr)$ with $T_0$ defined by \textup{(\ref{27-5-68})} or \textup{(\ref{27-5-69})} and
$T'$ defined by \textup{(\ref{27-5-73})}, and $T''$ either $0$ (as $(Z-N)_+\le B^{\frac{5}{12}}$) or defined by \textup{(\ref{27-5-74})}, or \textup{(\ref{27-5-75})}.
\end{proposition}

\begin{remark}\label{rem-27-5-22}
\begin{enumerate}[label=(\roman*), fullwidth]
\item\label{rem-27-5-22-i}
Recall that the decoupling error between singularity and a regular part does not exceed $C\alpha B Z^{\frac{5}{3}}$;

\item\label{rem-27-5-22-ii}
Meanwhile in these settings $\Scott-\Scott_0= O(\alpha Z^3)$ and if decoupling error is greater than this there is no point in decoupling.

\item\label{rem-27-5-22-iii}
Obviously we need to assume that $r_*\le \epsilon\bar{r}$ which implies that
\begin{equation}
(\alpha Z)^{40} B^{\frac{369}{4}}\le \epsilon Z^{121}
\label{27-5-77}
\end{equation}
which is just tiny bit stronger than $\alpha Z\lesssim 1$, $B\lesssim Z^{\frac{4}{3}}$. As $(Z-N)_+ \le B^{\frac{3}{4}}$ this condition is also sufficient.

\item\label{rem-27-5-22-iv}
Decoupling singularities we get an error (\ref{27-5-76}) with integration over $\{x:\, \ell(x)\asymp r\}$; therefore as $r_*\ll d\le \bar{r}$ minimizing
$T_0r^{-3}+T'r^3$ by $r: r_*\le r\le d$ we get
\begin{equation}
T_*\Def \bigl(T_0 (T'+T'')\bigr)^{\frac{1}{2}}+ (T'+T'') r_*^3+T_0 d^{-3}.
\label{27-5-78}
\end{equation}
\end{enumerate}
\end{remark}

\section{Endgame: \texorpdfstring{$M\ge 2$}{M \textge 2}}
\label{sect-27-5-4-5}

As $M\ge 2$ we have two rather different results. In the first (\ref{27-5-39}) we appeal to the sum of localized trace terms
$\sum_{1\le m\le M} \Tr (\psi_m H^-_{A,V_m} \psi_m)$ where $\psi_m$
is supported in $\{x:\, |x-\y_m|\le \frac{1}{3}d\}$ (recall that $d$ is the minimal distance between singularities).

In the second one we want to use $2h^{-2}\sum S(\alpha Z_m)Z_m^2$ instead. If $A'=0$ then transition would be immediate. However in our case we need to ``decouple'' singularities. Therefore in the estimate from below we need results from the previous Subsubsection:

\begin{theorem}\label{thm-27-5-23}
Let $V=W^\TF_B+\lambda$ be a Thomas-Fermi potential as $B\le Z^{\frac{4}{3}}$, $N\asymp Z$ and $M\ge 2$. Let $\kappa=\alpha Z\le \kappa^*$. Then as
$r_*\le d\le \bar{r}$
\begin{equation}
|\E^*_\alpha + \int P_{B} (V) \,dx -
2\sum S(\alpha Z_m)Z_m^2- \Schwinger|
\label{27-5-79}
\end{equation}
does not exceed $C(Q + T)$ where $Q$ is the trace estimate obtained in Proposition~\ref{prop-27-5-12}\ref{prop-27-5-12-iii}--\ref{prop-27-5-12-iv}
and $T_*$ is an estimate for expression \textup{(\ref{27-5-76})} given by
\textup{(\ref{27-5-78})}.
\end{theorem}

\begin{proof}
\begin{enumerate}[label=(\roman*), fullwidth]
\item\label{proof-27-5-23-i}
In the estimate from below we just replace $A'$ in $\E_\alpha (A)$ by\newline
$\sum_{1\le m\le M} A'\psi_m$ with $\psi_m$ supported in
$\{x:\,|x-\y_m|\le \frac{1}{3} r\}$ and equal $1$ in
$\{x:\,|x-\y_m|\le \frac{1}{4} r\}$ where $r$ is the minimal distance between singularities and observe that $\alpha^{-1} \|\partial A'\|^2$ increased by no more than $CT$;

\item\label{proof-27-5-23-ii}
In the estimate from above we just plug into $\E_\alpha(A)$ \ $A'=\sum_{1\le m\le M} A'_m\psi_m$ with $A'_m$ minimizers for a single-singularity potential $V_m$.
\end{enumerate}
\end{proof}

\begin{remark}\label{rem-27-5-24}
\begin{enumerate}[label=(\roman*), fullwidth]
\item\label{rem-27-5-24-i}
Theorem~\ref{thm-27-5-23} makes sense only as $r_*\ll d\le \bar{r}$ and
$T_*\ll \alpha Z^3$; if any of these assumptions fails we observe that $\E_\alpha (A')$ is greater than $\cE_0+\Schwinger-CQ-C\alpha Z^3$ and in this case we can replace $S(\alpha Z_m)$ by $S(0)$; in this case in the upper estimate we pick up $A'=0$;

\item\label{rem-27-5-24-ii}
In the estimate from above $T=\max_m T_m$ with $T_m$ an estimate for a single-singularity potential $V_m$ delivered by Proposition~\ref{prop-27-5-19}; thus $T=T_0 d^{-3}$ as $r_*\le d\le \bar{r}$. Still it is at least
$C\alpha B^{\frac{3}{4}} Z^2$ while decoupling error of singularity and the regular zone is $C\alpha B Z^{\frac{5}{3}}$ which is smaller.
\end{enumerate}
\end{remark}

\chapter{Global trace asymptotics in the case of Thomas-Fermi potential: \texorpdfstring{$Z^{\frac{4}{3}}\le B\le Z^3$}{Z\textfoursuperior\textthreesuperior \textle Z\textle Z\textthreesuperior}}
\label{sect-27-6}

\section{Trace estimates}
\label{sect-27-6-1}

In this Section we consider the case of $Z^{\frac{4}{3}}\le B\le Z^3$ (corresponding to $\beta h\ge 1$ after rescaling\footnote{\label{foot-27-34} Recall that as $Z^{\frac{4}{3}}\lesssim B\lesssim Z^3$ the scaling is
$x\mapsto B^{\frac{2}{5}}Z^{-\frac{1}{5}}x$ (and the original distance between nuclei is at least $B^{-\frac{2}{5}}Z^{\frac{1}{5}}$),
$\tau \mapsto B^{-\frac{2}{5}}Z^{\frac{4}{5}}\tau$,
$\beta = B^{\frac{2}{5}}Z^{-\frac{1}{5}}$, $h=B^{\frac{1}{5}}Z^{-\frac{3}{5}}$ and $B\lesssim Z^3$.}). We start with

\begin{remark}\label{rem-27-6-1}
\begin{enumerate}[label=(\roman*), fullwidth]

\item\label{rem-27-6-1-i}
In contrast to the previous Section~\ref{sect-27-5} in this case the remainder estimate will be at least $C\kappa h^{-2}$ and therefore there will be no difficulty to decouple between singularities or between singularities and a regular zone and the Scott correction term will be either $\sum 2S(0)Z_m^2$ or even absent.

Therefore in the estimate from above we just pick up $A'=0$ both here and in the multiparticle problem and we will need only $\N$-term and $\D$-terms with
$A'= 0$ referring to Chapter~\ref{book_new-sect-25};

\item\label{rem-27-6-1-ii}
As $\beta h\ge 1$ we have a major dichotomy unrelated to the self-generated magnetic field:

\begin{enumerate}[label=(\alph*), fullwidth]
\item\label{rem-27-6-1-iia}
 $\beta h^2\le 1$ (i.e. $Z^{\frac{4}{3}}\le B\le Z^{\frac{7}{4}}$). In this case Scott correction term could be larger than the contribution of zone $\{x:\, \ell(x)\asymp 1\}$ to the remainder estimate which is no better than $O(\beta)$ and one probably needs to include Scott correction term in the final trace asymptotics;

\item\label{rem-27-6-1-iib}
$\beta h^2\ge 1$ (i.e. $Z^{\frac{7}{4}}\le B\le Z^3$). In this case the opposite is true; then one does not need to include Scott correction term for sure. Recall that condition $C\le Z^3$ is also unrelated to self-generated magnetic field;
\end{enumerate}
\item\label{rem-27-6-1-iii}
Recall that we need to impose condition $\kappa \beta h^2 |\log \beta|^K\le 1$ which is equivalent to
\begin{equation}
\alpha B^{\frac{4}{5}}Z^{-\frac{2}{5}}|\log Z|^K\le 1.
\label{27-6-1}
\end{equation}
\end{enumerate}
\end{remark}

\begin{theorem}\footnote{\label{foot-27-35} Cf. Proposition~\ref{prop-27-5-2}.}\label{thm-27-6-2}
Let $V$ be Thomas-Fermi potential $W^\TF_B+\lambda$ as
$Z^{\frac{4}{3}}\le B\le Z^3$, $N\asymp Z_1\asymp Z_2\asymp \ldots \asymp Z_M$ and $N\le Z$. Let $\alpha Z\le \kappa^*$ and assumption \textup{(\ref{27-6-1})} be fulfilled. Let $A'$ be a minimizer. Then

\begin{enumerate}[label=(\roman*), fullwidth]
\item\label{thm-27-6-2-i}
As \underline{either} $(Z-N)_+\lesssim B^{\frac{4}{15}}Z^{\frac{1}{5}}$ \underline{or} $M=1$ and $\alpha B^{\frac{3}{5}}Z^{\frac{1}{5}} \gtrsim 1$ expression
\begin{gather}
|\E_\kappa (A') - \Scott_0 + \int \int P_{B} (V(x))\,dx|\label{27-6-2}\\
\intertext{does not exceed}
C\Bigl( B^{\frac{1}{3}}Z^{\frac{4}{3}}+ B^{\frac{4}{5}}Z^{\frac{3}{5}} + \alpha Z^3 + \alpha^{\frac{16}{9}}B^{\frac{82}{45}}Z^{\frac{58}{45}}|\log Z|^K\Bigr);
\label{27-6-3}
\end{gather}

\item\label{thm-27-6-2-ii}
As $M= 1$, $(Z-N)_+\gtrsim B^{\frac{4}{15}}Z^{\frac{1}{5}}$ and
$\alpha B^{\frac{3}{5}}Z^{\frac{1}{5}} \lesssim 1$ expression \textup{(\ref{27-6-2})} does not exceed
\begin{multline}
C\Bigl( B^{\frac{1}{3}}Z^{\frac{4}{3}}+ B^{\frac{4}{5}}Z^{\frac{3}{5}} +
\alpha Z^3 + \alpha^{\frac{16}{9}}B^{\frac{82}{45}}Z^{\frac{49}{45}}|\log Z|^K\\
+ \alpha^{\frac{40}{27}}B^{\frac{74}{45}}Z^{\frac{131}{540}}
(Z-N)_+^{\frac{85}{108}} |\log Z|^K\Bigr);
\label{27-6-4}
\end{multline}

\item\label{thm-27-6-2-iii}
As $M\ge 2$ and $(Z-N)_+\gtrsim B^{\frac{4}{15}}Z^{\frac{1}{5}}$ expression \textup{(\ref{27-6-2})} does not exceed
\begin{multline}
C\Bigl( B^{\frac{1}{3}}Z^{\frac{4}{3}}+ B^{\frac{4}{5}}Z^{\frac{3}{5}} +
\alpha Z^3 + \alpha^{\frac{16}{9}}B^{\frac{82}{45}}Z^{\frac{49}{45}}|\log Z|^K\\
 +
B^{\frac{7}{10}} Z^{\frac{11}{40}}(Z-N)_+^{\frac{5}{8}}
|\log (Z-N)_+Z^{-1}|\Bigr).
\label{27-6-5}
\end{multline}
\end{enumerate}
\end{theorem}

\begin{proof}
Observe first that we need to prove only the estimate from below for expression (\ref{27-6-2}) without absolute value as in the estimate from above we just pick $A'=0$ and apply results of Chapter~\ref{book_new-sect-25} producing estimate $CB^{\frac{4}{5}}Z^{\frac{3}{5}}$.

Proof of the estimate from below repeats one of Proposition~\ref{prop-27-5-2}. Namely we apply an appropriate partition and on each element $\psi_\iota^2$ estimate from below
\begin{equation}
\Tr^- (\psi_\iota H_{A,V}\psi_\iota) + (C_0\alpha )^{-1}\int |\partial A'|^2\,dx
+\int \int P_{B} (V(x))\psi_\iota ^2(x)\,dx
\label{27-6-6}
\end{equation}
as $\iota \ge 1$ and
\begin{multline}
\Tr^- (\psi_0 H_{A,V}\psi_0) + (C_0\alpha )^{-1}\int |\partial A'|^2\,dx \\
+\int \int P_{B} (V(x))\psi_0 ^2(x)\,dx - \Scott_0.
\label{27-6-7}
\end{multline}
As $B\le Z^2$ we separate zone
$\cX_0\Def\{x:\, \ell(x)\le r_*= B^{-\frac{2}{3}}Z^{\frac{1}{3}}\}$ in which after rescaling $x\mapsto r_*^{-1}x$, $\tau\mapsto r_*Z^{-1}\tau$, we have
$\beta=1$, $h=B^{\frac{1}{3}}Z^{-\frac{2}{3}}$ and $\kappa = \alpha Z$. Then for corresponding partition element $\psi_0^2$ expression (\ref{27-6-7}) in virtue of Chapter~\ref{book_new-sect-26} does not exceed (by an absolute value)
\begin{equation*}
C\bigl(h^{-1} + \kappa |\log \kappa|^{\frac{1}{3}}h^{-\frac{4}{3}}\bigr)\times
Zr_*^{-1}
\end{equation*}
which does not exceed (\ref{27-6-3}) without the last term.

As $B\ge Z^2$ we separate zone
$\cX_0\Def \{x:\, \ell(x)\le r_*=Z^{-1}\}$ and after rescaling $x\mapsto r_*^{-1}x$, $\tau\mapsto Z^{-1} r_*\tau$ we have $h=1$,
$\beta= BZ^{-2}$ and apply variational estimates of Appendix~\ref{sect-27-A-1} here. Then the contribution of $\cX_0$ to the remainder does not exceed
$C\beta \times Zr_*^{-1}=CB \le CB^{\frac{4}{5}}Z^{\frac{3}{5}}$.

\medskip\noindent
(b) Further, contribution of each regular element with
$r_*\le \ell(x)\le \epsilon \bar{r}$ (recall that $\bar{r}=B^{-\frac{2}{5}}Z^{\frac{1}{5}}$) does not exceed
\begin{align}
& C\zeta ^2 h_1^{-1}\Bigl(1 +
(\kappa_1\beta_1)^{\frac{40}{27}}h_1^{\frac{34}{27}}|\log h_1|^K\Bigr)
&&\text{as\ \ } \beta_1h_1\le C_0,
\label{27-6-8}\\
&C\zeta^2\beta_1\ \Bigl(1+ (\kappa_1\beta_1)^{\frac{16}{9}}h_1^{\frac{14}{9}}|\log h_1|^K\Bigr)
&&\text{as\ \ } \beta_1 h_1\ge C_0
\label{27-6-9}
\end{align}
with $\zeta=Z^{\frac{1}{2}}\ell^{-\frac{1}{2}}$, $\beta_1=BZ^{-\frac{1}{2}}\ell^{\frac{3}{2}}$, $h_1=Z^{-\frac{1}{2}}\ell^{-\frac{1}{2}}$, $\kappa_1=\alpha Z$. Indeed, as $\beta_1h_1\ge C_0$ and
$\ell(x)\le \epsilon \bar{r}$ super-strong non-degeneracy assumption (\ref{27-4-4}) is fulfilled.

Observe that the first term in (\ref{27-6-8}) has $\ell$ in the negative power and therefore sums to its value as $\ell=r_*$ while the second term has $\ell $ in the positive power and therefore sums to its value as $\beta_1h_1 = 1$ (i.e. $\ell=B^{-1}Z$, $h_1=B^{\frac{1}{2}}Z^{-1}$, $\beta_1=B^{-\frac{1}{2}}Z$); one can see easily that it is less than $\alpha Z^3$. Actually this zone appears only as $B\le Z^2$.

On the other hand, both terms in (\ref{27-6-9}) have $\ell$ in the positive power and thus sum to their values as $\ell=\bar{r}$,
$\beta_1=B^{\frac{2}{5}}Z^{-\frac{1}{5}}$, $h_1=B^{\frac{1}{5}}Z^{-\frac{3}{5}}$ and $\kappa_1=\alpha Z$ which are exactly the second and the fourth terms in (\ref{27-6-2}).

\medskip\noindent
(c) \emph{Boundary zone\/} $\{x:\,\epsilon \le \ell(x) \le c\}$ is treated in the same way albeit with $\zeta=B^{\frac{2}{5}}Z^{\frac{4}{5}}\gamma^2$, $\beta_1=B^{\frac{2}{5}}Z^{-\frac{1}{5}} \gamma^{-1}$,
$h_1=B^{\frac{1}{5}}Z^{-\frac{3}{5}}\gamma^{-3}$ and
$\kappa_1=\alpha Z \gamma^5$ as long as $\gamma \ge C_0\bar{\gamma}$ (with $\bar{\gamma}= (Z-N)_+^{\frac{1}{4}} Z^{-\frac{1}{4}}$ but reset to
$B^{\frac{1}{15}}Z^{-\frac{1}{5}}$ if the latter is larger. Observe that plugging into (\ref{27-6-9}) we get in both terms $\gamma$ in the power greater than $2$; therefore after summation b with respect to partition elements we get expression (\ref{27-6-9}) with $\gamma=1$, $\beta_1=B^{\frac{2}{5}}Z^{-\frac{1}{5}}$, $h_1=B^{\frac{1}{5}}Z^{-\frac{3}{5}}$ and $\kappa_1=\alpha Z$.

This proves the lower estimate (\ref{27-6-2}) in the framework of the first clause of Statement~\ref{thm-27-6-2-i} as contribution of the zone
$\gamma \le B^{\frac{1}{15}}Z^{-\frac{1}{5}}$ is estimated easily; we leave it to the reader.

\medskip\noindent
(d) Assume now that $(Z-N)_+\ge B^{\frac{4}{15}}Z^{\frac{1}{5}}$. We do not partition zone $\{x:\, \gamma (x)\le C_0\bar{\gamma}\}$ further. In this case we need to take
\begin{gather}
\nu_1 = \bigl((\kappa_1\beta_1)^{\frac{4}{3}}h_1^{\frac{2}{3}}+
(\kappa_1\beta_1)^{\frac{10}{9}}h_1^{\frac{4}{9}}\bigr)|\log h_1|^K.
\label{27-6-10}\\
\intertext{As $M=1$ we should plug it into}
C\zeta^2 \beta_1
\bigl(1+ h_1^{\frac{2}{3}}\nu_1^{\frac{4}{3}}\bigr)\gamma^{-2}
\label{27-6-11}
\end{gather}
with $\zeta=B^{\frac{2}{5}}Z^{\frac{4}{5}}\gamma^2$ and $\gamma=\bar{\gamma}$.

As $\kappa \beta h\gtrsim 1$ we estimate it by the same expression with $\bar{\gamma}$ replaced by $1$ but then in $\nu_1$ dominates the first term and we arrive to the lower estimate (\ref{27-6-2}) in the framework of the second clause of Statement~\ref{thm-27-6-2-i}.

\medskip\noindent
(e) As $M=1$ and $\kappa \beta h\lesssim 1$ we need to take into account term (\ref{27-6-11}) with
$\nu _1= (\kappa_1 \beta_1)^{\frac{10}{9}}h_1^{\frac{4}{9}}|\log h_1|^K$ which results in the last term in (\ref{27-6-4}). Indeed, as
$\gamma(x)\le C_0\bar{\gamma}$ super-strong non-degeneracy condition is not fulfilled.

\medskip\noindent
(f) As $M\ge 2$, we need to take into account term
$C\zeta^2 \beta_1 h_1^{-\frac{1}{2}} \gamma^{-2}(1+|\log \gamma|)$ with $\gamma=\bar{\gamma}$ which results in the last term in (\ref{27-6-5}).
This concludes estimate from below.
\end{proof}

\begin{corollary}\label{cor-27-6-4}
In the framework of Theorem~\ref{thm-27-6-2}\ref{thm-27-6-2-i}, \ref{thm-27-6-2-ii}, \ref{thm-27-6-2-iii} expression $\|\partial A'\|^2$ does not exceed expressions \textup{(\ref{27-6-3})}, \textup{(\ref{27-6-4})} and \textup{(\ref{27-6-5})} respectively multiplied by $\alpha$.
\end{corollary}

\begin{proof}
Indeed, the same estimates hold with $\alpha$ replaced by $2\alpha$.
\end{proof}

The same methods lead us to a similar result as $B\lesssim Z^{\frac{4}{3}}$:

\begin{theorem}\label{thm-27-6-4}
Let $V$ be Thomas-Fermi potential $W^\TF_B+\lambda$ as
$B\le Z^{\frac{4}{3}}$, $N\asymp Z_1\asymp Z_2\asymp \ldots \asymp Z_M$ and $N\le Z$. Let $\alpha Z\le \kappa^*$. Let $A'$ be a minimizer. Then

\begin{enumerate}[label=(\roman*), fullwidth]
\item\label{thm-27-6-4-i}
As \underline{either} $(Z-N)_+\lesssim B^{\frac{5}{12}}$ \underline{or} $M=1$ expression
\begin{gather}
|\E_\kappa (A') - \Scott_0 + \int \int P_{B} (V(x))\,dx|\label{27-6-12}\\
\intertext{does not exceed}
C\Bigl( B^{\frac{1}{3}}Z^{\frac{4}{3}}+ Z^{\frac{5}{3}} + \alpha Z^3 \Bigr);
\label{27-6-13}
\end{gather}

\item\label{thm-27-6-4-ii}
As $M\ge 2$ and $ B^{\frac{5}{12}}\lesssim (Z-N)_+\lesssim B^{\frac{3}{4}}$ expression \textup{(\ref{27-6-12})} does not exceed
\begin{equation}
C\Bigl( B^{\frac{1}{3}}Z^{\frac{4}{3}}+ Z^{\frac{5}{3}} +
\alpha Z^3 +
B^{\frac{29}{32}} (Z-N)_+^{\frac{5}{8}}|\log (Z-N)_+B^{-\frac{3}{4}}|\Bigr).
\label{27-6-14}
\end{equation}

\item\label{thm-27-6-4-iii}
As $M\ge 2$ and $(Z-N)_+\gtrsim B^{\frac{3}{4}}$ expression \textup{(\ref{27-6-12})} does not exceed
\begin{equation}
C\Bigl( B^{\frac{1}{3}}Z^{\frac{4}{3}}+ Z^{\frac{5}{3}} +
\alpha Z^3 + B (Z-N)_+^{\frac{1}{2}}\Bigr).
\label{27-6-15}
\end{equation}
\end{enumerate}
\end{theorem}

\section{Estimates to minimizer}
\label{sect-27-6-2}

Observe that only terms $B^{\frac{1}{3}}Z^{\frac{4}{3}}$ and $\alpha Z^3$ are associated with the singularities and they are definitely smaller than $B^{\frac{4}{5}}Z^{\frac{3}{5}}$ as $B\gtrsim Z^{\frac{7}{4}}$. Therefore as
$B\gtrsim Z^{\frac{7}{4}}$ we do not expect estimate for $\D(\rho_\Psi-\rho_B,\,\rho_\Psi-\rho_B)$ better than expressions (\ref{27-6-3})--(\ref{27-6-5}).

However for $B\lesssim Z^{\frac{7}{4}}$ to improve such estimate we need to study a minimizer. We assume that the remainder in Theorem~\ref{thm-27-6-4} does not exceed $C\bigl(B^{\frac{1}{3}}Z^{\frac{4}{3}}+ \alpha Z^3\bigr)$ and therefore
\begin{gather}
\|\partial A'\| \le C\alpha^{\frac{1}{2}}B^{\frac{1}{6}}Z^{\frac{2}{3}}+
C\alpha Z^{\frac{3}{2}}
\label{27-6-16}\\
\intertext{or after our usual scaling}
\|\partial A'\| \le \varsigma \Def
C\bigl(\kappa +\kappa^{\frac{1}{2}}\beta^{\frac{1}{6}}h^{\frac{1}{2}}\bigr).
\label{27-6-17}
\end{gather}

Observe that as $Z^{\frac{4}{3}}\le B \le Z^2$ we have all zones.

\begin{proposition}\label{prop-27-6-5}
Let $V$ be $W^\TF_B+\lambda$ rescaled as $Z^{\frac{4}{3}}\le B\le Z^3$,
$N\asymp Z_1\asymp Z_2\asymp \ldots \asymp Z_M$ and $N\le Z$. Further, let $\alpha Z\lesssim 1$ and
$\alpha B^{\frac{2}{5}} Z^{-\frac{2}{5}}|\log \beta|^K\lesssim 1$. Then under assumption \textup{(\ref{27-6-17})} the minimizer $A'$ satisfies
\begin{enumerate}[label=(\roman*), fullwidth]
\item \label{prop-27-6-5-i}
As $\ell\le r_*=h^2$ (i.e. $h_1\ge 1$)
\begin{equation}
|\partial^2 A'|\le C \varsigma h^{-5};
\label{27-6-18}
\end{equation}
\item \label{prop-27-6-5-ii}
As $r_*\le \ell \le c(\beta h)^{-1}$ (i.e. $h_1\le 1$, $\beta_1h_1\le c$)
\begin{multline}
|\partial^2 A'|\le C\kappa\Bigl( \varsigma \ell^{-\frac{5}{2}}+
\min \bigl( \beta^{\frac{3}{2}}h^{\frac{1}{2}} \ell^{-1},\,
\beta^{\frac{1}{2}}\ell^{-\frac{3}{4}}\bigr)\Bigr)|\log \ell/r_*| \\
+ C (\kappa \beta)^{\frac{10}{9}}h^{\frac{4}{9}} \ell^{-\frac{19}{9}}
|\log \ell/r_*|^K;
\label{27-6-19}
\end{multline}
\item \label{prop-27-65-iii}
As $c(\beta h)^{-1}\le \ell \le c $
\begin{multline}
|\partial^2 A'|\le C\varsigma \ell^{-\frac{5}{2}}+
C\kappa \beta^{\frac{1}{2}}\ell^{-\frac{7}{4}}|\log \ell/r_*| \\
+
C \Bigl((\kappa \beta)^{\frac{10}{9}}h^{\frac{4}{9}}\ell^{-\frac{19}{9}}
+(\kappa \beta)^{\frac{4}{3}} h^{\frac{2}{3}} \ell^{-\frac{5}{6}}\Bigr)
|\log \ell/r_*|^K;
\label{27-6-20}
\end{multline}
\end{enumerate}
\end{proposition}

\begin{proof}
The proof of these two propositions repeats our standard arguments and is left to the reader.
\end{proof}

These propositions may not provide the best $\D$-term estimate as
$\kappa \beta h\le 1$ (i.e. $\alpha B^{\frac{3}{5}}Z^{\frac{1}{5}}\le 1$) and could be improved in virtue of the super-strong non-degeneracy assumption fulfilled at regular elements with $\ell\ge c(\beta h)^{-1} $ and at border elements with $\gamma \ge C_0\bar{\gamma}$. We want to improve term containing
$(\kappa \beta )^{\frac{10}{9}} h^{\frac{4}{8}}\ell^{-\frac{19}{9}}$ in (\ref{27-6-20}). Assume now that $\beta h^2\le 1$ and $\kappa \beta h\le 1$ (case we need to cover). Let us consider zone
$\{x:(\epsilon_0 \beta h \ge V(x)\ge C_0 |\eta|\}$ where in the corresponding scale super-strong non-degeneracy condition is fulfilled and
$\eta=\lambda B^{-\frac{2}{5}}Z^{\frac{1}{5}}$.

\begin{proposition}\label{prop-27-6-6}
Let conditions of Proposition~\ref{prop-27-6-5} be fulfilled. Then

\begin{enumerate}[label=(\roman*), fullwidth]
\item\label{prop-27-6-6-i}
Estimate
\begin{gather}
|\partial^2 A'(x)| \le
C\Bigl(\varsigma \ell^{-\frac{5}{2}}+
\kappa \beta ^{\frac{1}{2}}\ell(x)^{-\frac{7}{4}} +
(\kappa\beta)^{\frac{4}{3}} h^{\frac{2}{3}} \ell(x)^{-\frac{5}{6}}\Bigr)
|\log h|^K
\label{27-6-21}\\
\shortintertext{holds as}
\beta h |\log h|^{-\delta}\ge V(x)\ge |\eta|\cdot |\log h|^\delta
\label{27-6-22}
\end{gather}
with arbitrarily small exponent $\delta>0$;
\item\label{prop-27-6-6-ii}
Furthermore, if $|\eta|\le |\log h|^{-\delta}$ then estimate
\begin{equation}
|\partial^2 A'(x)| \le
C\Bigl(\varsigma+ \kappa \beta ^{\frac{1}{2}} +
(\kappa\beta)^{\frac{4}{3}} h^{\frac{2}{3}} +
(\kappa\beta)^{\frac{10}{9}} h^{\frac{4}{9}}\bar{\gamma}^{\frac{28}{9}}\Bigr) |\log h|^K
\label{27-6-23}
\end{equation}
holds as $|V(x)|\lesssim 1$, $\bar{\gamma}=|\eta|^{\frac{1}{4}}$.
\end{enumerate}
\end{proposition}

\begin{proof}
(i) Let
\begin{equation*}
\nu_n(t)=\sup _{\cX_n(t)} |\partial^2 A'(x)|\ell(x)^{\frac{5}{2}},\qquad
\cX_n(t)= \{e^{-\epsilon n-1} t\le V(x)\le e^{\epsilon n+1} t\}.
\end{equation*}
Here $\epsilon>0$ is arbitrarily small (but constants may depend on it). Assume that
\begin{equation}
e^{\epsilon n} t\le \epsilon_0,\qquad
e^{-\epsilon n} t \ge C_0 |\eta|, \qquad e^{\epsilon n} \le |\log h|.
\label{27-6-24}
\end{equation}
Here first two conditions assure that in $\cX_n(t)$ the super-strong non-degeneracy assumption is fulfilled after rescaling and the last condition assures that $\ell(x)$ remains the same (modulo logarithmic factor) here. Then
\begin{equation}
\nu_n(t) \le C\Bigl(\varsigma + \kappa \beta_1^{\frac{1}{2}} +
 C\kappa_1 \beta_1 h_1^{\frac{1}{4}} (\nu_{n+1}(t))^{\frac{1}{4}}\Bigr)
 |\log h|^{K_1}
 \label{27-6-25}
\end{equation}
with $\kappa_1=\kappa$, $\beta_1=\beta r^{\frac{3}{2}}$, $h_1=hr^{-\frac{1}{2}}$, $r= \min(t^{-1}, 1)$. Indeed, $|\Delta A'|$ in $\cX_{n+1/2}(t)$ does not exceed the right-hand expression (without term $\varsigma$ and without logarithmic factor) multiplied by $r^{-\frac{5}{2}}$ and $C\varsigma $ estimates $\sL^2$-norm of $\partial A'$ (and we scale it properly). Recall that we scale $x\mapsto x/r$ as $r\le \epsilon_0$ and $x\mapsto x/\gamma$ as $r\asymp 1$ and in the latter case
$\beta_1=\beta \gamma^{-1}$, $h_1=h\gamma^{-3}$ and $\kappa_1=\kappa \gamma^5$; the uncertainty due to $r$ or $\gamma$ defined modulo logarithmic factor compensates by $|\log h|^{K_1}$ in the right-hand expression of (\ref{27-6-25}).

Therefore
\begin{gather*}
F_n(t)\le C\Bigl(\varsigma r^{-\frac{5}{3}}+
C\kappa \beta^{\frac{1}{2}}r^{-\frac{11}{12}} +
C\kappa \beta h^{\frac{1}{2}}
\times (F_{n+1}(t))^{\frac{1}{4}}\Bigr) |\log h|^{K_1}\\
\intertext{for $F_n(t)=\nu_n(t) r^{-\frac{5}{3}}$. Iterating we see that}
F_0(t)\le
C\Bigl(\varsigma r^{-\frac{5}{3}}+
\kappa \beta^{\frac{1}{2}} r^{-\frac{11}{12}}\Bigr)|\log h|^K
+C(\kappa \beta h^{\frac{1}{2}}\bigr)^{\frac{4}{3}}|\log h|^K
\times \underbracket{(F_{n+1}(t))^{\frac{1}{4^n}}}.
\end{gather*}
Since $F_{n+1}(r)\le h^{-L}$ the last factor is bounded by a constant as
$2^n \ge |\log h|$ and we can satisfy this and (\ref{27-6-24}) as
long as (\ref{27-6-22}) holds. This proves Statement~\ref{prop-27-6-6-i}.

\medskip\noindent
(ii) Consider remaining zone
$\cY=\{V(x)\le |\eta|\cdot |\log h|^{\delta}\}$. Let
$\nu = \sup_{\cY} |\partial^2 A'|$. Observe that in $\cY$ \ $|\Delta A'|$ does not exceed
\begin{equation*}
C\Bigl(\kappa \beta ^{\frac{1}{2}} +
\kappa \beta h^{\frac{1}{2}}|\nu|^{\frac{1}{4}}+
\kappa \beta h^{\frac{2}{5}}\bar{\gamma}^{\frac{14}{5}} \nu ^{\frac{1}{10}}
\Bigr) |\log h|^K
\end{equation*}
and on its border (\ref{27-6-23}) is fulfilled. It implies that (\ref{27-6-23}) is fulfilled in $\cY$ as well. This proves Statement~\ref{prop-27-6-6-ii}.
\end{proof}

\begin{remark}\label{rem-27-6-7}
If $|\eta| \ge |\log h|^{-\delta}$ then Proposition~\ref{prop-27-6-5} is sufficiently good in the remaining zone $\cY$.
\end{remark}

\section{$\N$-term asymptotics and $\D$-term estimates}
\label{sect-27-6-3}

We leave to the reader not complicated but rather tedious and error-prone task

\begin{problem}\label{problem-27-6-8}
Estimate remainder in $\N$-term
\begin{gather}
|\int \bigl(e(x,x,\lambda') - P_{B}(V(x)+\lambda')\bigr)\,dx|
\label{27-6-26}\\
\intertext{and $\D$-term}
\D \bigl( e(x,x,\lambda') - P_{B}(V(x)+\lambda'),\,
e(x,x,\lambda') - P_{B}(V(x)+\lambda')\bigr).
\label{27-6-27}
\end{gather}
\end{problem}

After usual rescaling one needs to consider the following zones:

\begin{enumerate}[label=(\alph*), fullwidth]
\item
Zone $\{x:\, \ell (x)\le (\beta h)^{-1}|\log h|^\delta\}$. In this zone one should use $\beta_1=\beta \ell^{\frac{3}{2}}$, $h_1=h \ell^{-\frac{1}{2}}$ and
$\nu _1= (\kappa \beta_1)^{\frac{10}{9}} h_1^{\frac{4}{9}}|\log h_1|^K$ (other terms are not important here); then its contributions to expressions (\ref{27-6-26}) and (\ref{27-6-27}) do not exceed respectively
$C\bigl(h_1^{-2}+ h_1^{-\frac{5}{3}}\nu_1^{\frac{2}{3}}\bigr)$ and
$C\bigl(h_1^{-2}+ h_1^{-\frac{5}{3}}\nu_1^{\frac{2}{3}}\bigr)^2\ell ^{-1}$ (as
$\ell (x)\lesssim (\beta h)^{-1}$ but slight extension just adds some logarithmic factor). In the final tally
$\ell =(\beta h)^{-1}|\log h|^{\delta}$.

\item
Zone
$\{x:\, (\beta h)^{-1}|\log h|^\delta \le \ell (x)\le |\log h|^{-\delta}\}$.
In this zone we have the same expressions for $h_1$ and $\beta_1$ and
$\nu _1= (\kappa \beta_1)^{\frac{4}{3}} h_1^{\frac{2}{3}}|\log h_1|^K$;
then its contributions to expressions (\ref{27-6-26}) and (\ref{27-6-27}) do not exceed respectively
$C\beta_1\bigl( h_1^{-1}+ h_1^{-\frac{2}{3}}\nu_1^{\frac{2}{3}}\bigr)$ and
$C\beta_1^2\bigl(h_1^{-1}+ h_1^{-\frac{2}{3}} \nu_1^{\frac{2}{3}}\bigr)^2 \ell^{-1}$. In the final tally $\ell=1$;

\item
Zone $\{x:\,\ell (x)\ge |\log h|^{-\delta},\,
\gamma (x)\ge \bar{\gamma} |\log h^\delta\}$. In this zone $h_1=h\gamma^{-3}$ and $\beta_1=\beta\gamma^{-1}$ and
$\nu = (\kappa \beta)^{{4}{3}}h^{\frac{2}{3}}|\log h|^K$ but we use according to Remark~\ref{rem-27-5-9} instead
$\nu_1= \kappa_1 \beta_1 h_1^{\frac{1}{4}} \nu^{\frac{1}{4}}$ with
$\kappa_1=\kappa\gamma^5$. Then

\begin{enumerate}[label=(c$_\arabic*$), fullwidth]
\item
As $M=1$ contributions of $\gamma$-element to expressions (\ref{27-6-26}) and (\ref{27-6-27}) do not exceed
$C\beta_1
\bigl( h_1^{-1}+ h_1^{-\frac{2}{3}}\nu_1^{\frac{2}{3}}\bigr)$ and
$C\beta_1^2\bigl(h_1^{-1}+ h_1^{-\frac{2}{3}} \nu_1^{\frac{2}{3}}\bigr)^2 $ respectively.

\item
As $M\ge 2$ contributions of $\gamma$-element to expressions (\ref{27-6-26}) and (\ref{27-6-27}) do not exceed respectively
$C\beta_1 h_1^{-1}\nu_1^{\frac{1}{2}}$ and $C\beta_1^2h_1^{-2} \nu_1 $.
\end{enumerate}

In the final tally $\gamma=1$;

\item
Then as $M=1$ contributions of $\bar{\gamma}$-element to expressions (\ref{27-6-26}) and (\ref{27-6-27}) do not exceed respectively (in comparison to what we have already)
$C\beta_1 h_1^{-\frac{2}{3}}\nu_1^{\frac{2}{3}}$ and
$C\beta_1^2h_1^{-\frac{4}{3}}\nu_1^{\frac{4}{3}}$ with $\nu_1 =
(\kappa \beta)^{\frac{10}{9}}h_1^{\frac{4}{9}}\bar{\gamma}^{\frac{58}{27}}
|\log h|^K$.

On the other hand, as $M\ge 2$ contributions of $\bar{\gamma}$-element to expressions (\ref{27-6-26}) and (\ref{27-6-27}) do not exceed respectively (in comparison to what we have already) $C\beta_1 h_1^{-\frac{3}{2}}$ and $C\beta_1^2 h_1^{-3}$.
\end{enumerate}

Let us partially summarize what we got. Let $\kappa \beta h \gtrsim 1$. Then if $M=1$ the final results are
$C\beta
\bigl(h^{-1}+ h^{-\frac{2}{9}} (\kappa \beta)^{\frac{8}{9}}|\log h|^K\bigr)$  and the same expression squared. If $M\ge 2$ the final results from all zones except $\{x:\,\gamma(x)\le C_0\bar{\gamma}\}$ are
$C\beta h^{-\frac{1}{3}}(\kappa \beta)^{\frac{2}{3}}|\log h|^K$ and the same expression squared.

\smallskip
On the other hand, let $\kappa \beta h \lesssim 1$. Then if $M=1$ the final results do not exceed
$C\beta
\bigl(h^{-1}+ h^{-\frac{10}{27}} (\kappa \beta)^{\frac{20}{27}}|\log h|^K\bigr)$  and the same expression squared.

\chapter{Applications to ground state energy}
\label{sect-27-7}

\section{Preliminary remarks}
\label{sect-27-7-1}

Recall that we are looking for
\begin{equation}
\langle \sfH \Psi,\Psi\rangle + \frac{1}{\alpha} \int |\partial A'|^2\,dx
\label{27-7-1}
\end{equation}
which should be minimized by $\Psi\in \fH$ and $A'$.

We know (see f.e. Subsection~\ref{book_new-24-2-1}) that
\begin{multline}
\langle \sfH \Psi,\Psi\rangle \ge \Tr ^- (H_{A,W+\lambda '}) + \lambda ' N
+\frac{1}{2} \D(\rho_\Psi-\rho,\,\rho_\Psi-\rho)\\
 - \frac{1}{2} \D(\rho,\,\rho)
-C\int \rho_\Psi^{\frac{4}{3}}\,dx
\label{27-7-2}
\end{multline}
where $W=V-|x|^{-1}* \rho$, $V$ is a Coulomb potential of nuclei, $\rho$ and $\lambda\le 0$ are arbitrary.

Therefore to derive estimate from below for expression (\ref{27-7-1}) we
just need to pick up $\rho$ and $\lambda' $ but we cannot pick up $A'$. Let us select $\rho$ and $\lambda$ equal to Thomas-Fermi density $=\rho^\TF_B$ and chemical potential $\lambda$ respectively (but if $N\approx Z$ it is beneficial to pick up $\lambda '=0$), add $\alpha^{-1}\int |\partial A'|^2\,dx$, and apply trace asymptotics without any need to consider $\N$- or $\D$-terms and we have an estimate from below which includes also ``bonus term''
$\frac{1}{2} \D(\rho_\Psi-\rho,\,\rho_\Psi-\rho)$ and is as good as a remainder estimate in the trace asymptotics rescaled--provided we estimate properly
the last term in (\ref{27-7-2}).

Thus here we are missing only estimate for
$\int \rho_\Psi^{\frac{4}{3}}\,dx$ or more sophisticated estimate if we are interested in Dirac and Schwinger terms. We will prove in Appendix~\ref{sect-27-A-2} that in the electrostatic inequality for near ground-state one can replace this term $-C\int \rho_\Psi^{\frac{4}{3}}\,dx$ by
$-CZ^{\frac{5}{3}}$ as $B\le Z^{\frac{4}{3}}$ and $-CB^{\frac{4}{5}}Z^{\frac{3}{5}}$ as $Z^{\frac{4}{3}}\le B\le Z^3$; further, as $B\le Z$ one can replace it by $\Dirac -C Z^{\frac{5}{3}-\delta}$ thus proving Bach-Graf-Solovej estimate in our current settings.

Therefore due to these arguments estimates from below in Theorems~\ref{thm-27-7-4} and \ref{thm-27-7-5} follow immediately from Theorems~\ref{thm-27-5-14},~\ref{thm-27-5-16} and~\ref{thm-27-5-23} as $B\le Z^{\frac{4}{3}}$ and estimates from below in Theorem~\ref{27-7-11} follow immediately from Theorems~\ref{thm-27-6-2} and~\ref{thm-27-6-4} as $Z^{\frac{4}{3}}\le B\le Z^3$.

On the other hand, estimate from above involves picking up $\rho$ (which we select to be $\rho^\TF_B$ again) and picking $A'$ as well--which we choose as in upper estimates of Section~\ref{sect-27-5} (rescaled) and also picking up $\Psi$ which we select $\phi_1(x_1,\varsigma_1)\cdots \phi_N(x_N,\varsigma_N)$ anti-symmetrized by $(x_1,\varsigma_1),\ldots ,(x_N,\varsigma_N)$ but we do not select $\lambda'$ in the trace asymptotics which \emph{must\/} be equal to $\lambda_N$ which is $N$-th eigenvalue of $H_{W,A}$ or to $0$ if there are less than $N$ negative eigenvalues of $H_{W,A}$.

In this case we need to estimate $|\lambda' -\lambda|$ and also
\begin{equation}
\D\bigl( \tr e(x,x,\lambda')- \rho^\TF_B,\,
\tr e(x,x,\lambda')- \rho^\TF_B\bigr)
\label{27-7-3}
\end{equation}
which required some efforts in Chapters~\ref{book_new-sect-24}--\ref{book_new-sect-26} but here we will do it rather easily because in Section~\ref{sect-27-5} we took either $A'$ equal to one for Coulomb potential and without external magnetic field (as $\beta h^2\le 1$) or just $0$ (as $\beta h^2\ge 1$).

\section{Estimate from above: \texorpdfstring{$B\le Z^{\frac{4}{3}}$}{B\textle Z\textfoursuperior\textthreesuperior}}
\label{sect-27-7-2}

Recall that now we can select $\rho$ and $A'$ but cannot select $\lambda'$.

As $B\le Z^{\frac{4}{3}}$ (or $\beta h\le 1$ after rescaling) we select $A'$ as a minimizer for one-particle operator with Coulomb potential and without external magnetic field. Then after rescaling $x\mapsto Z^{\frac{1}{3}}x$,
$\tau\mapsto Z^{-\frac{4}{3}}\tau$, $h=1\mapsto h=Z^{-\frac{1}{3}}$,
$B\mapsto \beta = BZ^{-1}$, $\alpha \mapsto \kappa =\alpha Z$
\begin{align}
&|\partial A'|\le C\kappa \ell^{-\frac{3}{2}}, \qquad
&&|\partial ^2 A'|\le C\kappa \ell^{-\frac{5}{2}}|\log (\ell/h^2)|
\label{27-7-4}\\
\shortintertext{or before it}
&|\partial A'|\le C\alpha Z^{\frac{5}{3}} \ell^{-\frac{3}{2}}, \qquad
&&|\partial ^2 A'|\le C\alpha Z^{\frac{5}{3}}\ell^{-\frac{5}{2}}
|\log (Z\ell)|.
\label{27-7-5}
\end{align}

Let us start from the easy case $M=1$. We need the following

\begin{proposition}\label{prop-27-7-1}
Let $M=1$, $N\asymp Z$, $B\lesssim Z^{\frac{4}{3}}$ and $\alpha Z\le \kappa^*$. Assume that $A'$ satisfies \textup{(\ref{27-7-5})}. Then
\begin{enumerate}[label=(\roman*), fullwidth]
\item\label{prop-27-7-1-i}
The remainder in the trace asymptotics does not exceed
\begin{equation}
\left\{\begin{aligned}
&C\bigl(Z^{\frac{5}{3}} +
\alpha |\log (\alpha Z)|^{\frac{1}{3}}Z^{\frac{25}{9}}\bigr)\qquad
&&\text{as\ \ } B \le Z,\\
& C\bigl(B^{\frac{1}{3}} Z^{\frac{4}{3}} +
\alpha |\log (\alpha Z)|^{\frac{1}{3}}B^{\frac{2}{9}}Z^{\frac{23}{9}}\bigr)
 \qquad &&\text{as\ \ } B \ge Z;
\end{aligned}\right.
\label{27-7-6}
\end{equation}

\item\label{prop-27-7-1-ii}
The remainder in $\N$-term asymptotics does not exceed $CZ^{\frac{2}{3}}$;

\item\label{prop-27-7-1-iii}
$\D$-term does not exceed $CZ^{\frac{5}{3}}$.
\end{enumerate}
\end{proposition}

\begin{proof}
We cannot directly apply previous results as $A'$ is now generated by much slower decaying Coulomb potential. The good news however is that $A'$ is generated without presence of the external magnetic field.
Let us scale as we mentioned above.

\medskip\noindent
(i) Consider the remainder estimate in the trace asymptotics.

\medskip\noindent
(ia) Contribution of the near singularity zone
$\{x:\, \ell(x)\le \ell^*= \epsilon \min (\beta^{-\frac{2}{3}},1)\}$ does not exceed
\begin{equation}
\left\{\begin{aligned}
&C(h^{-1} + \kappa |\log \kappa|^{\frac{1}{3}}h^{-\frac{4}{3}})\qquad
&&\text{as\ \ } \beta \le 1,\\
& C(\beta^{\frac{1}{3}} h^{-1} +
\kappa |\log \kappa|^{\frac{1}{3}}\beta^{\frac{2}{9}} h^{-\frac{4}{3}}) \qquad &&\text{as\ \ } \beta \ge 1.
\end{aligned}\right.
\label{27-7-7}
\end{equation}

\medskip\noindent
(ib) Contribution of $\ell$-element in the regular zone
$\{x:\, \ell(x)\ge \ell^*\}$ does not exceed
$C\zeta^2 \bigl(h_1^{-1} + h_1^{-\frac{1}{3}} \nu_1^{\frac{2}{3}}\bigr)$.
We plug $h_1=h/(\zeta \ell)$, and in virtue of (\ref{27-7-4})
$\nu_1 =\kappa \ell^{-\frac{5}{2}} \cdot \ell^3\zeta^{-1}=
\kappa \ell^{\frac{1}{2}}\zeta^{-1}$. Taking a sum we arrive to the same expression as $\ell= \ell^*$ i.e. expression (\ref{27-7-7}). Scaling back we arrive to (\ref{27-7-6}).

\medskip\noindent
(ic) We leave analysis of the boundary zone to the reader. It requires just repetition of the corresponding arguments of Section~\ref{sect-27-5}.

\medskip\noindent
(ii) Next let us consider a remainder estimate in the $\N$-term asymptotics.

\medskip\noindent
(iia) Contribution of $\ell$-element in the regular zone does not exceed
$C\bigl( h_1^{-2}+ h_1^{-\frac{5}{3}}\nu_1^{\frac{1}{3}}\bigr)$ and plugging $h_1$, $\nu_1$ we arrive after summation to the same expression as $\ell=1$ i.e.
$Ch^{-2}$. One can prove easily that the contribution of the near singularity zone is $O(h^{-2+\delta})$.

Consider contribution of the boundary zone. Note that this zone appears only if
$B\ge (Z-N)_+^{\frac{4}{3}}$. Let us scale $x\mapsto x\bar{\ell}^{-1}$, $\tau\mapsto \bar{\ell}^{-4} \tau$, with $\bar{\ell}=(\beta h)^{-\frac{1}{4}}$. Then $h\mapsto h_1= h^{\frac{3}{4}}\beta^{-\frac{1}{4}}$,
$\beta\mapsto \beta_1= h_1^{-1}$ and after scaling $A'$ satisfies
$|\partial^2 A'|\le \nu_1= C\kappa \bar{\ell}^{\frac{3}{2}}|\log \kappa|$.

\medskip\noindent
(iib) Consider contribution of $\gamma$-element; scaling again
$x\mapsto x\gamma^{-1}$, $\tau\mapsto \gamma^{-4}\tau$,
$h_1\mapsto h_2=h_1\gamma^{-3}$, $\beta_1\mapsto \beta \gamma^{-1}$, $\nu_1\mapsto\nu_2=\nu_1$ we see that it does not exceed
$C\beta_2\bigl( h_2^{-1}+ h_2^{-\frac{2}{3}} \nu_2^{\frac{2}{3}}\bigr)=
C\beta_1\bigl(\gamma^2 h_1^{-1}+
\gamma h_1^{-\frac{2}{3}} \nu_1^{\frac{2}{3}}\bigr)$. This expression must be divided by $\gamma^2$ and summed resulting in
\begin{equation}
C\bigl(h_1^{-2}|\log \bar{\gamma}|+
\bar{\gamma}^{-1} h_1^{-\frac{5}{3}} \nu_1^{\frac{2}{3}}\bigr)
\label{27-7-8}
\end{equation}
with $\bar{\gamma}\ge h_1^{\frac{1}{3}}$ (equality is achieved as $(Z-N)_+$ is small enough, otherwise partitioning may be cut-off by a chemical potential). One can get rid off the logarithmic factor by our standard propagation arguments; the second term in (\ref{27-7-8}) does not exceed $Ch_1^{-2}\nu_1^{\frac{2}{3}}$ and plugging $h_1$, $\nu_1$ we get $O(h^{-2})=O(Z^{\frac{2}{3}})$.

\medskip\noindent
(iii) $\D$-term is analyzed in the same way.
\end{proof}

\begin{proposition}\label{prop-27-7-2}
In the framework of Proposition~\ref{prop-27-7-1} assume that $B\le Z$. Then
\begin{enumerate}[label=(\roman*), fullwidth]
\item\label{prop-27-7-2-i}
The remainder in the trace asymptotics (with the Schwinger term) does not exceed
\begin{equation}
CZ^{\frac{5}{3}} \bigl(Z^{-\delta}+ (BZ^{-1})^{\delta}\bigr)+
\alpha |\log (\alpha Z)|^{\frac{1}{3}}Z^{\frac{25}{9}};
\label{27-7-9}
\end{equation}

\item\label{prop-27-7-2-ii}
The remainder in $\N$-term asymptotics does not exceed\\ ${CZ^{\frac{2}{3}}\bigl(Z^{-\delta}+ (BZ^{-1})^{\delta}+(\alpha Z)^{\delta} \bigr)}$;

\item\label{prop-27-7-2-iii}
$\D$-term does not exceed $CZ^{\frac{5}{3}}
\bigl(Z^{-\delta}+ (BZ^{-1})^{\delta}+(\alpha Z)^{\delta}\bigr)$.
\end{enumerate}
\end{proposition}

\begin{proof}
Proof includes improved (due to the standard arguments of propagation of singularities) estimates of the contributions of \emph{threshold zone\/}
$\{x:\,h^{\delta'}\le \ell(x)\le h^{-\delta'}\}$ after rescaling. We leave details to the reader.
\end{proof}

As we now have exactly the same $\N$-term asymptotics and the same remainder term estimate as in Subsection~\ref{book_new-sect-25-6-3} as if there was no self-generated magnetic field we immediately arrive to the same estimates of $|\lambda_N-\lambda|$ as there and therefore to the same estimates for $|\lambda_N-\lambda|\cdot N$ and not only for two $\D$-terms
\begin{gather*}
\D\bigl( \tr e(x,x,\tau) - P'_B(W+\tau),\,\tr e(x,x,\tau) - P'_B(W+\tau)\bigr)\\
\intertext{with $\tau=\lambda$ and $\tau=\lambda_N$ but also for the third one}
\D\bigl(P'_B(W+\lambda_N) - P'_B(W+\lambda),\,
P'_B(W+\lambda_N) - P'_B(W+\lambda)\bigr).
\end{gather*}
The trace term however is different--with Scott correction term
$2S(\alpha Z)Z^2$ instead of $2S(0)Z^2$ and the remainder estimate here includes an extra term related to $\alpha$.

It concludes the proof of the estimate from above for $\E^*_N$. Combined with estimate from below it concludes the proof of Theorem~\ref{thm-27-7-4}.

Consider now case $M\ge 2$. Since we need to decouple singularities in this case we need sufficiently fast decaying magnetic field and thus potential generating it. So we will take $A'=\sum_m A'_m \phi_m$ with $A'_m$ defined by $V=W^TF_m$ (where $\W^\TF_m$ corresponds to a single nucleus) without any magnetic field and with $N=Z$ (i.e. $\lambda=0$) and $\phi_m$ is supported in
$\{x:\,|x-\y_m\le \frac{1}{3}d\}$ and equals $1$ in
$\{x:\,|x-\y_m\le \frac{1}{4}d\}$. However everywhere else we take $V=W^\TF_B$.

Assume first that $(Z-N)_+$ is sufficiently small and we take $Z=N$ even in the definition of $W^\TF_B$. Then one can prove easily that
\begin{multline}
|\Tr ^- (H_{A,V}) +\alpha^{-1} \int |\partial A'|^2\,dx + \int P_B(V)\,dx
-\Scott|\\
\le \textup{(\ref{27-7-6})} + C\alpha Z^2 d^{-3}
\label{27-7-10}
\end{multline}
where the last term is due to decoupling; indeed $|\partial A'|$ and
$|\partial ^2 A'|$ decay as $\ell^{-3}$ and $\ell^{-4}$ as
$\ell \ge Z^{-\frac{1}{3}}$.

Moreover, as $B\le Z$ one can replace (\ref{27-7-6}) by
\begin{equation}
CZ^{\frac{5}{3}}
\bigl(Z^{-\delta}+ (BZ^{-1})^{\delta} +
(\alpha Z)^{\delta}+(dZ^{\frac{1}{3}})^{-\delta}\bigr)+
C\alpha |\log (\alpha Z)|^{\frac{1}{3}}Z^{\frac{25}{9}}
\label{27-7-11}
\end{equation}
while including $\Schwinger$ into left-hand expression. This estimate is at least as good as what we got in the estimate from below but probably even better since magnetic field admits now better estimates.

Let us estimate $|\lambda_N|$ if $\lambda_N<0$. To do this consider
\begin{equation}
\int \bigl[e(x,x,\lambda)-e(x,x,\lambda_N)\bigr]\,dx
\label{27-7-12}
\end{equation}
with non-negative integrand. Then contribution of $\ell$-element into the main part of this expression, namely
\begin{gather}
\int \bigl[e(x,x,\lambda)-e(x,x,\lambda_N)\bigr]\psi_\iota^2\,dx
\label{27-7-13}\\
\shortintertext{is}
\int \bigl[P'(x,V(x)+\lambda)-P'(x,V(x)+\lambda_N) \bigr]\psi_\iota^2\,dx
\label{27-7-14}
\end{gather}
does not depend on $A'$. On the other hand, since now $|\partial ^2 A'|$ admits so good estimate, contribution of this element to the remainder estimated as if there was no self-generated magnetic field.

The same is true for the boundary elements as well.

But then $|\lambda_N|$ is estimated exactly as if there was no self-generated magnetic field, i.e. as in Section~\ref{book_new-sect-25-6}. But then all components of the estimate from above, with exception of the trace term, namely
$|\lambda_N|\cdot N$ and all three $\D$-terms are estimated as in as in Section~\ref{book_new-sect-25-6}. Under assumption small $(Z-N)_+$ all of them do not exceed $CZ^{\frac{5}{3}}$ which could be improved to
\begin{equation}
CZ^{\frac{5}{3}}
\bigl(Z^{-\delta}+ (BZ^{-1})^{\delta} +
(\alpha Z)^{\delta}+(dZ^{\frac{1}{3}})^{-\delta}\bigr)
\label{27-7-15}
\end{equation}
as $B\le Z$.

Similarly, as $(Z-N)_+$ is larger than the corresponding threshold we need to consider two cases: $0>\lambda >\lambda_N$ and $0>\lambda_N >\lambda$ and estimate from below
\begin{gather}
\int \bigl(e(x,x,\lambda)-e(x,x,\lambda_N)\bigr)\psi_\iota^2\,dx
\label{27-7-16}\\
\shortintertext{is}
\int \bigl(e(x,x,\lambda_N)-e(x,x,\lambda)\bigr)\psi_\iota^2\,dx
\label{27-7-17}
\end{gather}
respectively leading to estimate of $|\lambda_N-\lambda|$ and then
$|\lambda_N-\lambda|\cdot N$ and all three $\D$-terms and again here these terms are estimated as if there was no self-generating magnetic field.

This concludes the proof of Theorem~\ref{thm-27-7-5}.

\begin{remark}\label{rem-27-7-3}
As $M\ge 2$ one could be concerned about term coming from
$C\beta h^{-\frac{1}{2}}$ in the trace term and $C\beta h^{-\frac{3}{2}}$ or its square in $\N$- and $\D$-terms but assuming that $d\lesssim \bar{r}$ we simply have $A'=0$ due to decoupling there and therefore apply theory of Chapter~\ref{book_new-sect-25} without any modification.
\end{remark}

\section{Main Theorems: \texorpdfstring{$B\le Z^{\frac{4}{3}}$}{B\textle Z\textfoursuperior\textthreesuperior}}
\label{sect-27-7-3}

\subsection{Ground state energy asymptotics}
\label{sect-27-7-3-1}

\begin{theorem}\label{thm-27-7-4}
Let $M=1$, $N\asymp Z$, $B\le Z^{\frac{4}{3}}$ and $\alpha \le \kappa^* Z^{-1}$ with small constant $\kappa^*$. Then

\begin{enumerate}[fullwidth, label=(\roman*)]
\item\label{thm-27-7-4-i}
As $B\le Z$
\begin{equation}
\E^*_N = \cE^\TF_N + \Scott +
O\Bigl(Z^{\frac{5}{3}} + \alpha |\log (\alpha Z)|^{\frac{1}{3}} Z^{\frac{25}{9}} \Bigr)
\label{27-7-18}
\end{equation}
with $\Scott=2Z^2 S(\alpha Z)$;

\item\label{thm-27-7-4-ii}
Moreover as $B\ll Z$ and $\alpha |\log Z| ^{\frac{1}{3}}\ll Z^{-\frac{10}{9}}$ this estimate could be improved to
\begin{multline}
\E^*_N = \cE^\TF_N + 2Z^2 S(\alpha Z) + \Schwinger + \Dirac\\[3pt]
+O\Bigl(Z^{\frac{5}{3}} \bigl[Z^{-\delta}+ B^{\delta}Z^{-\delta}\bigr]+
\alpha |\log (\alpha Z)|^{\frac{1}{3}} Z^{\frac{25}{9}}\Bigr);
\label{27-7-19}
\end{multline}

\item\label{thm-27-7-4-iii}
As $Z\le B\le Z^{\frac{4}{3}}$
\begin{multline}
\E^*_N =\cE^\TF_N + 2Z^2 S(\alpha Z) \\
+
O\Bigl(B^{\frac{1}{3}}Z^{\frac{4}{3}} + \alpha |\log (\alpha Z)|^{\frac{1}{3}} B^{\frac{2}{9}} Z^{\frac{23}{9}}+  \alpha BZ ^{\frac{5}{3}}\Bigr).
\label{27-7-20}
\end{multline}
\end{enumerate}
\end{theorem}

\begin{theorem}\label{thm-27-7-5}
Let $M\ge 2$, $N\asymp Z_1\asymp Z_2\asymp \ldots \asymp Z_M$,
$B\lesssim Z^{\frac{4}{3}}$. Let $\alpha\le \kappa^* Z^{-1}$,
$d\ge Z^{-\frac{1}{3}}$ be a minimal distance between nuclei capped by $\bar{r}=\min\bigl(B^{-\frac{1}{4}},(Z-N)_+^{-\frac{1}{3}}\bigr)$. Then
\begin{enumerate}[label=(\roman*), fullwidth]
\item\label{thm-27-7-5-i}
As $B\le Z$
\begin{equation}
\E^*_N = \cE^\TF_N + \Scott +
O\Bigl(Z^{\frac{5}{3}}+\alpha \bigl[|\log (\alpha Z)|^{\frac{1}{3}} Z^{\frac{25}{9}} +Z^2d^{-3}\bigr]\Bigr)
\label{27-7-21}
\end{equation}
with $\Scott= 2\sum _{1\le m\le M} Z_m^2S(\alpha Z_m)$ and moreover
\begin{multline}
\E^*_N = \cE^\TF_N + \Scott + \Schwinger +\Dirac\\+
O\Bigl(Z^{\frac{5}{3}}\bigl[Z^{-\delta}+ (BZ^{-1})^{\delta} + (dZ^{\frac{1}{3}})^{-\delta}\bigr] +
\alpha \bigl[|\log (\alpha Z)|^{\frac{1}{3}} Z^{\frac{25}{9}} + Z^2d^{-3}\bigr]\Bigr);
\label{27-7-22}
\end{multline}
\item\label{thm-27-7-5-ii}
As $B\ge Z$, $\alpha B\le Z^{\frac{17}{60}}|\log Z|^{-K}$ and
$(Z-N)_+\le B^{\frac{5}{12}}$
\begin{multline}
\E^*_N = \cE^\TF_N + \Scott\\
+ O\Bigl(B^{\frac{1}{3}}Z^{\frac{4}{3}}+
\alpha\bigl[|\log \alpha|^{\frac{1}{3}} B^{\frac{2}{9}}Z^{\frac{23}{9}}
+  BZ^{\frac{5}{3}}+ Z^2d^{-3}\bigr]\Bigr);
\label{27-7-23}
\end{multline}
\item\label{thm-27-7-5-iii}
As $B\ge Z^{\frac{77}{60}} |\log Z|^{-K}$,
$\alpha\ge B^{-1} Z^{\frac{17}{60}}|\log Z|^{-K}$,
$(Z-N)_+\le B^{\frac{5}{12}}$ and
$d\ge \bar{d}=(\alpha B)^{\frac{40}{209}}Z^{-\frac{81}{209}}|\log Z|^K$
\begin{multline}
\E^*_N = \cE^\TF_N + \Scott\\
O\Bigl(B^{\frac{1}{3}}Z^{\frac{4}{3}}+
\alpha \bigl[|\log \alpha|^{\frac{1}{3}} B^{\frac{2}{9}}Z^{\frac{23}{9}}
+  BZ^{\frac{5}{3}}
+ (\alpha B)^{\frac{120}{209}} Z^{\frac{384}{209}}d^{-3}|\log Z|^K\bigr]\Bigr);
\label{27-7-24}
\end{multline}
\item\label{thm-27-7-5-iv}
In any case
\begin{equation}
\E^*_N = \cE^\TF_N + \Scott_0 +
O\bigl(B^{\frac{1}{3}}Z^{\frac{4}{3}}+\alpha Z^3\bigr)
\label{27-7-25}
\end{equation}
with $\Scott_0= 2\sum _{1\le m\le M} Z_m^2S(0)$.
\end{enumerate}
\end{theorem}

Recall that in the free nuclei model excess energy is $\asymp d^{-7}$ (as
$d\le \epsilon B^{-\frac{1}{4}}$) and therefore an error must be greater than $\epsilon d^{-7}$. One can see easily that $d\ge \min (Z^{-\frac{5}{21}}, \, B^{-\frac{1}{4}})$ as $B\le Z$ and then in estimates (\ref{27-7-21}) and (\ref{27-7-22}) the last terms (with $d^{-3}$) could be skipped.

On the other hand, in (\ref{27-7-23})--(\ref{27-7-25}) we can either skip the last terms (with $d^{-3}$) or assume that $d\asymp B^{-\frac{1}{4}}$ and these terms should be calculated under this assumption and we arrive to

\begin{theorem}\label{thm-27-7-6}
Let $M\ge 2$, $N\asymp Z_1\asymp Z_2\asymp \ldots \asymp Z_M$,
$B\lesssim Z^{\frac{4}{3}}$ and $(Z-N)_+\lesssim B^{\frac{1}{2}}$. Consider a free nuclei model. Then
\begin{enumerate}[label=(\roman*), fullwidth]
\item\label{thm-27-7-6-i}
As $B\le Z$
\begin{equation}
\hat{\E}^*_N = \hat{\cE}^\TF_N + \Scott +
O\Bigl(Z^{\frac{5}{3}}+\alpha \bigl[|\log (\alpha Z)|^{\frac{1}{3}} Z^{\frac{25}{9}}\bigr]\Bigr)
\label{27-7-26}
\end{equation}
and moreover
\begin{multline}
\hat{\E}^*_N = \hat{\cE}^\TF_N + \Scott + \Schwinger +\Dirac\\+
O\Bigl(Z^{\frac{5}{3}}\bigl[Z^{-\delta}+ (BZ^{-1})^{\delta}\bigr] +
\alpha \bigl[|\log (\alpha Z)|^{\frac{1}{3}} Z^{\frac{25}{9}} \bigr]\Bigr);
\label{27-7-27}
\end{multline}
\item\label{thm-27-7-6-ii}
As $B\ge Z$,
$\alpha \le B^{-1}Z^{\frac{17}{60}}|\log Z|^{-K}$
\begin{equation}
\hat{\E}^*_N = \hat{\cE}^\TF_N + \Scott
+ O\Bigl(B^{\frac{1}{3}}Z^{\frac{4}{3}}+
\alpha\bigl[|\log \alpha|^{\frac{1}{3}} B^{\frac{2}{9}}Z^{\frac{23}{9}}
+  Z^2B^{\frac{3}{4}}\bigr]\Bigr);
\label{27-7-28}
\end{equation}
\item\label{thm-27-7-6-iii}
As $B\ge Z^{\frac{77}{60}} |\log Z|^{-K}$,
$ B^{-1} Z^{\frac{17}{60}}|\log h|^{-K}\le \alpha \le B^{-\frac{369}{160}}Z^{\frac{81}{40}}|\log Z|^{-K}$
\begin{multline}
\E^*_N = \cE^\TF_N + \Scott\\
O\Bigl(B^{\frac{1}{3}}Z^{\frac{4}{3}}+
\alpha \bigl[|\log \alpha|^{\frac{1}{3}} B^{\frac{2}{9}}Z^{\frac{23}{9}}
+ (\alpha B)^{\frac{120}{209}} Z^{\frac{384}{209}} B^{\frac{3}{4}}
|\log Z|^K\bigr]\Bigr);
\label{27-7-29}
\end{multline}
\item\label{thm-27-7-6-v}
In any case
\begin{equation}
\hat{\E}^*_N = \hat{\cE}^\TF_N + \Scott_0 +
O\bigl(B^{\frac{1}{3}}Z^{\frac{4}{3}}+\alpha Z^3\bigr).
\label{27-7-30}
\end{equation}
\end{enumerate}
\end{theorem}

We leave to the reader the following

\begin{Problem}\label{Problem-27-7-7}
In the frameworks of fixed nuclei and free nuclei models consider the case
$M\ge 2$, $B\ge Z$ and $(Z-N)_+\ge B^{\frac{5}{12}}$. Use results of Sections~\ref{sect-27-5} and~\ref{book_new-sect-25-6}. Recall that there are two cases:
$B^{\frac{5}{12}} \le (Z-N)_+\le B^{\frac{3}{4}}$ and
$(Z-N)_+\ge B^{\frac{3}{4}}$.

In particular find out for which $B$ this assumption could be skipped without deterioration of the remainder estimates.
\end{Problem}

\subsection{Ground state density asymptotics}
\label{sect-27-7-3-2}

Consider now asymptotics of $\rho_\Psi$. Apart of independent interest one needs them for estimate of excessive negative charge and estimate or asymptotics of the ionization energy.

\begin{theorem}\label{thm-27-7-8}
Let $M=1$, $N\asymp Z$, $B\lesssim Z^{\frac{4}{3}}$ and
$\alpha \le \kappa^*Z^{-1}$, with small constant $\kappa^*$. Then

\begin{enumerate}[label=(\roman*), fullwidth]
\item\label{thm-27-7-8-i}
As $B\le Z$
\begin{gather}
\D(\rho_\Psi-\rho^\TF,\,\rho_\Psi-\rho^\TF) = O\bigl(Z^{\frac{5}{3}} \bigr)
\label{27-7-31}\\
\intertext{and moreover this estimate could be improved to}
\D(\rho_\Psi-\rho^\TF,\,\rho_\Psi-\rho^\TF) =
O\Bigl(Z^{\frac{5}{3}} \bigl[Z^{-\delta}+ B^{\delta}Z^{-\delta} +
(\alpha Z)^{\delta}\bigr]\Bigr);
\label{27-7-32}
\end{gather}

\item\label{thm-27-7-8-ii}
As $Z\le B\le Z^{\frac{4}{3}}$
\begin{equation}
\D(\rho_\Psi-\rho^\TF,\,\rho_\Psi-\rho^\TF) \\
= O\Bigl(Z^{\frac{5}{3}} +
(\alpha B)^{\frac{40}{27}} Z^{\frac{101}{81}} |\log Z|^K \Bigr).
\label{27-7-33}
\end{equation}
\end{enumerate}
\end{theorem}

\begin{proof}
We need to consider only the case when errors in the estimates for $\E^*_N$ exceed those announced in (\ref{27-7-31})--(\ref{27-7-33}). Otherwise an estimate for $\D(\rho_\Psi-\rho^\TF,\,\rho_\Psi-\rho^\TF)$ follow from the estimates from above and below for $\E^*_N$ as estimates from below contain the ``bonus term'' $\D(\rho_\Psi-\rho^\TF,\,\rho_\Psi-\rho^\TF)$.

Let in the estimate from below pick up $\lambda'=\lambda_N$ and in the estimate from above pick up $A'$ as a minimizer for a potential $W^\TF_B+\lambda'$ with $\lambda'=\lambda_N$; we do not calculate asymptotics of the trace terms as these terms in both estimates coincide\footnote{\label{foot-27-36} Due to the matching choices of $A'$ and $\lambda'$.}; then we arrive to estimate
\begin{multline}
\D( \rho_\Psi -\rho^\TF,\,\rho_\Psi -\rho^\TF)\\
\le
C\D( \tr e(x,x,\lambda_N)- \rho^\TF,\, \tr e(x,x,\lambda_N)- \rho^\TF) +
CZ^{\frac{5}{3}}
\label{27-7-34}
\end{multline}
and we need to estimate the first term in the right-hand expression.

Let us scale as usual. Then $\beta h\lesssim 1$ as $B\le Z^{\frac{4}{3}}$ and our standard arguments estimate this term by
$C\bigl(h^{-2}+ h^{-\frac{5}{3}}\nu^{\frac{2}{3}}\bigr)^2$ with
$\nu = (\kappa \beta)^{\frac{10}{9}}h^{\frac{4}{9}}|\log h|^K$.

Plugging $\beta = BZ^{-1}$, $h=Z^{-\frac{1}{3}}$, $\kappa=\alpha Z$ and multiplying by $Z^{\frac{1}{3}}$ due to the spatial scaling we arrive exactly to (\ref{27-7-31}) and (\ref{27-7-33}).

Furthermore, as $\beta \ll 1$ let us consider contribution of the \emph{main zone\/} $\{x:\, h^{\delta} \le \ell(x)\le
h^{-\delta}+(\beta +\kappa)^{\delta}\}$ and use propagation arguments and improved electrostatic inequality; then we arrive to estimate
$Ch^{-4}( h + \beta +\kappa)^{\delta}$ which after rescaling becomes (\ref{27-7-32}).

We leave all easy details to the reader.
\end{proof}

Consider now case $M\ge 2$. Then no matter what is the distance between nuclei (as long as it is greater than $\epsilon Z^{-\frac{1}{3}}$) we need to add one or two more extra terms.

\medskip\noindent
(a) The first one always appears and it is what becomes from
$C\beta_1^2 h_1^{-2} \nu_1 \ell^{-1}$ as we plug $\beta_1=\beta\ell^3$,
$h_1=h\ell$, $\kappa_1=\kappa \ell^{-3}$,
$\nu_1=(\kappa_1 \beta_1)^{\frac{10}{9}} h_1^{\frac{4}{9}}|\log h_1|^K\asymp
(\kappa \beta)^{\frac{10}{9}} h^{\frac{4}{9}}\ell^{\frac{4}{9}}|\log h|^K$, and multiply by $\ell^{-1}$ we get $\ell$ in the positive power and therefore we must plug the largest possible $\ell$ which in case $(Z-N)_+\le B^{\frac{3}{4}}$ is $\ell=(\beta h)^{-\frac{1}{4}}$. Also plugging $\kappa=\alpha Z$,
$\beta =BZ^{-1}$, $h=Z^{-\frac{1}{3}}$,
$\ell =(\beta h)^{-\frac{1}{4}}=B^{-\frac{1}{4}}Z^{\frac{1}{3}}$ and multiplying by $Z^{\frac{1}{3}}$ due to scaling (with a possible improvement as $B\ll Z$) we arrive to $\alpha^{\frac{40}{27}}B^{frac{9}{4}}|\log Z|^K$. One can see easily that this term is larger than the second term in (\ref{27-7-33}).

As $(Z-N)_+\ge B^{\frac{3}{4}}$ this term will be smaller than the second extra term.

\medskip\noindent
(b) The second extra term appears only if $(Z-N)_+\gtrsim B^{\frac{5}{12}}$.

\begin{enumerate}[label=(b$_\arabic*$), fullwidth]
\item
As $B^{\frac{5}{12}}\le (Z-N)_+\le B^{\frac{3}{4}}$ it is what becomes from
$\beta_2^2 h_2^{-3}\bar{\gamma}^{-4}|\log \bar{\gamma}|^2$ with substitutions
$\beta_2=\beta_1\bar{\gamma}^{-1}$, $h_2=h_1\bar{\gamma}^{-3}$,
$\bar{\gamma}= (Z-N)_+^{\frac{1}{4}}B^{-\frac{3}{16}}$ and with $\beta_1,h_1,\ell$ defined above (we still need to multiply by $Z^{\frac{1}{3}}$).

\item
As $ (Z-N)_+\ge B^{\frac{3}{4}}$ it is what becomes out
$\beta_1^2 h_1^{-3}\ell^{-1}$ with $\ell=(Z-N)_+^{-\frac{1}{3}}$.
\end{enumerate}

Thus we arrive to

\begin{theorem}\label{thm-27-7-9}
Let $M\ge 2$, $N\asymp Z_1\asymp \ldots\asymp Z_M$, $B\lesssim Z^{\frac{4}{3}}$. Further, let $\alpha \le \kappa^*Z^{-1}$, $d\gtrsim Z^{-\frac{1}{3}}$. Then

\begin{enumerate}[label=(\roman*), fullwidth]

\item\label{thm-27-7-9-i}
As $B\le Z$ estimate \textup{(\ref{27-7-31})} holds and moreover it could be improved to
\begin{multline}
\D(\rho_\Psi-\rho^\TF,\,\rho_\Psi-\rho^\TF) \\
=
O\Bigl(Z^{\frac{5}{3}} \bigl[Z^{-\delta}+ B^{\delta}Z^{-\delta} +
(\alpha Z)^{\delta}+ (dZ^{\frac{1}{3}})^{-\delta}\bigr]\Bigr);
\label{27-7-35}
\end{multline}
\item\label{thm-27-7-9-ii}
As $ B\ge Z$ and $(Z-N)_+\le B^{\frac{5}{12}}$
\begin{equation}
\D(\rho_\Psi-\rho^\TF,\,\rho_\Psi-\rho^\TF)
= O\Bigl(Z^{\frac{5}{3}} +
\alpha^{\frac{10}{9}}B^{\frac{9}{4}}|\log Z|^K\Bigr);
\label{27-7-36}
\end{equation}

\item\label{thm-27-7-9-iii}
As $ B\ge Z$ and $B^{\frac{5}{12}}\le (Z-N)_+\le B^{\frac{3}{4}}$
\begin{multline}
\D(\rho_\Psi-\rho^\TF,\,\rho_\Psi-\rho^\TF)\\
= O\Bigl(Z^{\frac{5}{3}} +
\alpha^{\frac{10}{9}}B^{\frac{9}{4}} |\log Z|^K + B^{\frac{15}{16}}(Z-N)_+^{\frac{3}{4}}(1+|\log (Z-N)_+B^{-\frac{1}{3}}|)^2\Bigr);
\label{27-7-37}
\end{multline}

\item\label{thm-27-7-9-iv}
As $ B\ge Z$ and $(Z-N)_+\le B^{\frac{3}{4}}$
\begin{equation}
\D(\rho_\Psi-\rho^\TF,\,\rho_\Psi-\rho^\TF)
= O\Bigl(Z^{\frac{5}{3}} + B^2(Z-N)_+^{-\frac{2}{3}} \Bigr);
\label{27-7-38}
\end{equation}
\end{enumerate}
\end{theorem}

\begin{corollary}\label{cor-27-7-10}
Estimates \textup{(\ref{27-7-31})}, \textup{(\ref{27-7-35})}-- \textup{(\ref{27-7-38})} hold for a free nuclei model.
\end{corollary}

\section{Main theorems: \texorpdfstring{$Z^{\frac{4}{3}}\le B\le Z^3$}{Z\textfoursuperior\textthreesuperior \textle Z\textle Z\textthreesuperior}}
\label{sect-27-7-4}

\subsection{Ground state energy asymptotics}
\label{sect-27-7-4-1}

As $Z^{\frac{4}{3}}\le B\le Z^3$ we select $A'=0$ in the estimate from above and therefore just apply an upper estimate $\E^*_N$ from Subsection~\ref{book_new-sect-25-6-3}. Combined with estimate from below provided by Theorem~\ref{thm-27-6-2} it implies the following

\begin{theorem}\label{thm-27-7-11}
Let  $Z^{\frac{4}{3}}\lesssim B\lesssim Z^3$, $Z_1\asymp \ldots \asymp Z_M\asymp N$, and
$\alpha \le \kappa^*Z^{-1}$ with small constant $\kappa^*$, and also
$\alpha \le B^{-\frac{4}{5}}Z^{\frac{2}{5}} |\log Z|^{-K}$. Then

\begin{enumerate}[fullwidth, label=(\roman*)]
\item\label{thm-27-7-11-i}
As $M=1$ and either $(Z-N)_+\lesssim B^{\frac{4}{15}}Z^{\frac{1}{5}}$ or
$\alpha B^{\frac{3}{5}} Z^{\frac{1}{5}}\gtrsim 1$
\begin{gather}
|\E^*_N -\cE^\TF_N - \Scott_0|
\label{27-7-39}\\
\shortintertext{does not exceed}
C\Bigl( B^{\frac{1}{3}}Z^{\frac{4}{3}}+ B^{\frac{4}{5}}Z^{\frac{3}{5}} +
\alpha Z^3 + \alpha^{\frac{16}{9}}B^{\frac{82}{45}}Z^{\frac{49}{45}}|\log Z|^K \Bigr);
\tag{\ref{27-6-3}}\label{27-6-3xx}
\end{gather}

\item\label{thm-27-7-11-ii}
As $M=1$ and $(Z-N)_+\gtrsim B^{\frac{4}{15}}Z^{\frac{1}{5}}$ and
$\alpha B^{\frac{3}{5}} Z^{\frac{1}{5}}\lesssim 1$ expression \textup{(\ref{27-7-39})} does not exceed
\begin{multline}
C\Bigl( B^{\frac{1}{3}}Z^{\frac{4}{3}}+ B^{\frac{4}{5}}Z^{\frac{3}{5}} +
\alpha Z^3 + \alpha^{\frac{16}{9}}B^{\frac{82}{45}}Z^{\frac{49}{45}}|\log Z|^K\\
+ \alpha^{\frac{40}{27}}B^{\frac{13}{15}}Z^{\frac{139}{135}}
(Z-N)_+^{\frac{28}{27}} |\log Z|^K\Bigr);
\tag{\ref{27-6-4}}\label{27-6-4xx}
\end{multline}

\item\label{thm-27-7-11-iii}
As $M\ge 2$ and $(Z-N)_+\lesssim B^{\frac{4}{15}}Z^{\frac{1}{5}}$ expression \textup{(\ref{27-7-39})} does not exceed
\begin{equation}
C\Bigl( B^{\frac{1}{3}}Z^{\frac{4}{3}}+
B^{\frac{4}{5}}Z^{\frac{3}{5}} (1+|\log BZ^{-3}|)^2+
\alpha Z^3 + \alpha^{\frac{16}{9}}B^{\frac{82}{45}}Z^{\frac{49}{45}}|\log Z|^K \Bigr);
\label{27-7-40}
\end{equation}

\item\label{thm-27-7-11-iv}
As $M\ge 2$ and $(Z-N)_+\gtrsim B^{\frac{4}{15}}Z^{\frac{1}{5}}$ expression \textup{(\ref{27-7-39})} does not exceed
\begin{multline}
C\Bigl( B^{\frac{1}{3}}Z^{\frac{4}{3}}+ B^{\frac{4}{5}}Z^{\frac{3}{5}}
(1+|\log BZ^{-3}|)^2 +
\alpha Z^3 + \alpha^{\frac{16}{9}}B^{\frac{82}{45}}Z^{\frac{49}{45}}|\log Z|^K\\
 + B^{\frac{2}{3}}(Z-N)_+ (1+|\log (Z-N)_+Z^{-1}|)^2\Bigr).
\label{27-7-41}
\end{multline}
\end{enumerate}
\end{theorem}

\subsection{Ground state density asymptotics}
\label{sect-27-7-4-2}

Consider now asymptotics of $\rho_\Psi$. Apart of independent interest we need them for estimate of excessive negative charge and estimate or asymptotics of the ionization energy. We are interested as usual in
$\D\bigl(\rho_\Psi-\rho^\TF_B,\, \rho_\Psi-\rho^\TF_B\bigr)$ and we know that

\begin{corollary}\label{cor-27-7-12}
In the framework of Theorem~\ref{thm-27-7-11}\ref{thm-27-7-11-i}, \ref{thm-27-7-11-ii}, \ref{thm-27-7-11-iii}, \ref{thm-27-7-11-iv}
$\D\bigl(\rho_\Psi-\rho^\TF_B,\, \rho_\Psi-\rho^\TF_B\bigr)$
does not exceed corresponding remainder estimate \textup{(\ref{27-6-3xx})},
\textup{(\ref{27-6-4xx})}, \textup{(\ref{27-7-40})}, \textup{(\ref{27-7-41})}.
\end{corollary}

If we want to get rid off $\alpha Z^3+ B^{\frac{1}{3}}Z^{\frac{4}{3}}$ terms (which may be dominant only for $B\le Z^{\frac{7}{4}}$ and $B\le Z^{\frac{11}{7}}$ respectively) we need not to find asymptotics of the trace term but to have trace term in the estimates from above and below more consistent. The only explored option is to take in the estimate from above the same $A'$ as in the estimate from below, which is a minimizer for the corresponding one particle problem.

\begin{theorem}\label{thm-27-7-13}
Let  $Z^{\frac{4}{3}}\lesssim B\lesssim Z^3$, $Z_1\asymp \ldots \asymp Z_M\asymp N$, and $\alpha \le \kappa^*Z^{-1}$ with small constant $\kappa^*$, and also $\alpha \le B^{-\frac{4}{5}}Z^{\frac{2}{5}} |\log Z|^{-K}$.

\begin{enumerate}[label=(\roman*), fullwidth]
\item\label{thm-27-7-13-i}
Let $M=1$.  Then
$\D\bigl(\rho_\Psi-\rho^\TF_B,\,\rho_\Psi-\rho^\TF_B\bigr)$ does not exceed
\begin{multline}
CB^{\frac{4}{5}}Z^{\frac{3}{5}}\\[3pt]
+ C\Bigl(
\alpha ^{\frac{16}{9}} B^{\frac{82}{45}} Z^{\frac{49}{45}} + \alpha^{\frac{40}{27}}B^{-\frac{10}{9}} Z^{\frac{127}{27}}+
\alpha^{\frac{40}{27}}B^{\frac{74}{45}} Z^{\frac{11}{90}} (Z-N)_+^{\frac{29}{54}}\Bigr)|\log Z|^K;
\label{27-7-42}
\end{multline}
\item\label{thm-27-7-13-ii}
Let $M=2$ and the minimal distance between nuclei $d\gtrsim B^{-\frac{2}{5}}Z^{\frac{1}{5}}$.  Then
$\D\bigl(\rho_\Psi-\rho^\TF_B,\,\rho_\Psi-\rho^\TF_B\bigr)$ does not exceed
\begin{multline}
CB^{\frac{4}{5}}Z^{\frac{3}{5}}\\[3pt]
+ C\Bigl(
\alpha ^{\frac{4}{3}} B^{\frac{22}{15}} Z^{\frac{19}{15}} + \alpha^{\frac{40}{27}}B^{-\frac{10}{9}} Z^{\frac{127}{27}}+
B^{\frac{3}{5}} Z^{\frac{9}{20}} (Z-N)_+^{\frac{3}{4}}\Bigr)|\log Z|^K;
\label{27-7-43}
\end{multline}
\end{enumerate}
\end{theorem}

\begin{proof}
We will use Propositions~\ref{prop-27-6-5} and~\ref{prop-27-6-6} to estimate (\ref{27-6-27}). Let us consider for each partition element
\begin{equation}
|\int \bigl( e(x,x,\lambda') - P_{B}(V(x)+\lambda')\bigr)\psi (x)\,dx|.
\label{27-7-44}
\end{equation}

\medskip\noindent
(a) Zone $\{x:\, \ell(x)\le B^{-1}Z |\log h|^{\delta}\}$. Here for each $\ell$-element expression (\ref{27-7-44}) does not exceed
$R_a=C(h_1^{-2}+\beta_1 h_1^{-1}) (1+ \nu_1^{\frac{2}{3}}h_1^{\frac{1}{3}})$ with $\beta_1=\beta \ell^{\frac{3}{2}}$, $h_1=h\ell^{-\frac{1}{2}}$ and $\nu_1$ defined according to Proposition~\ref{prop-27-6-5}\ref{prop-27-6-5-ii}. As usual
 $\beta =B^{\frac{2}{5}}Z^{-\frac{1}{5}}$ and $h=B^{\frac{1}{5}}Z^{-\frac{3}{5}}$.

Then one can prove easily that the total contribution of this zone to \begin{equation}
\D\bigl(e(x,x,\lambda') - P_{B}(V(x)+\lambda'),\,
e(x,x,\lambda') - P_{B}(V(x)+\lambda')\bigr)
\label{27-7-45}
\end{equation}
does not exceed $CB Z^{-1}R_a^2$\,\footnote{\label{foot-27-37} Calculated for $\ell$ or $\gamma$ on its maximum.} which is the second term in the parenthesis of (\ref{27-7-42}) and (\ref{27-7-43})\,\footnote{\label{foot-27-38} Modulo term not exceeding $CB^{\frac{4}{5}}Z^{\frac{3}{5}}$.}.

\medskip\noindent
(b) Zone $\{x:\, B^{-1}Z |\log h|^{\delta}\le \ell(x)\le \epsilon B^{-\frac{2}{5}}Z^{\frac{1}{5}}\}$ (with the exception of the case
$\bar{\gamma}\ge |\log h|^{-\delta}$ which we leave to the reader).
Here for each $\ell$-element expression (\ref{27-7-44}) does not exceed
$R_b=C\beta_1 h_1^{-1}(1+ \nu_1^{\frac{2}{3}}h_1^{\frac{1}{3}})$ with $\beta_1=\beta \ell^{\frac{3}{2}}$, $h_1=h\ell^{-\frac{1}{2}}$ and $\nu_1$ defined according to Proposition~\ref{prop-27-6-6}\ref{prop-27-6-6-i}.

Then one can prove easily that the total contribution of this zone to (\ref{27-7-45}) does not exceed $CB^{\frac{2}{5}} Z^{-\frac{1}{5}}R_b^2$\,\footref{foot-27-37} which is  the first term in the parenthesis of (\ref{27-7-42})\,\footref{foot-27-38}.

\medskip\noindent
(c) Zone $\{x:\, \bar{\gamma}|\log h|^{\delta}\le \gamma(x)\le C_0\}$.
Here for each $\gamma$-element expression (\ref{27-7-44}) does not exceed $R_c$ where $R_c=C\beta_2 h_2^{-1}(1+ \nu_2^{\frac{2}{3}}h_2^{\frac{1}{3}})$ (as $M=1$) and $R_c=C\beta_2 h_2^{-1}\nu_2^{\frac{1}{2}}$ (as $M \ge 2$)
with $\beta_2=\beta \gamma^{-1}$, $h_2=h\gamma^{-3}$ and $\nu_2$ defined by Proposition~\ref{prop-27-6-6}\ref{prop-27-6-6-i} and redefined by Remark~\ref{rem-27-5-9}.

Then one can prove easily that the total contribution of this zone to
(\ref{27-7-45}) does not exceed $CB^{\frac{2}{5}} Z^{-\frac{1}{5}}R_c^2$\,\footref{foot-27-37} which is the first term in the parenthesis of (\ref{27-7-42}) and (\ref{27-7-43}) for $M=1$ and $M\ge 2$ respectively\,\footref{foot-27-38}.

\medskip\noindent
(d) Zone $\{x:\, \gamma(x) \le \bar{\gamma}|\log h|^{\delta}\}$.
Here for each $\gamma$-element expression (\ref{27-7-44}) does not exceed
$R_d=C\beta_2 h_2^{-1}(1+ \nu_2^{\frac{2}{3}}h_2^{\frac{1}{3}})$ (as $M=1$) and
$R_d=C\beta_2 h_2^{-\frac{3}{2}}$ (as $M \ge 2$)
with $\beta_2=\beta \gamma^{-1}$, $h_2=h\gamma^{-3}$ and $\nu_2$ defined by Proposition~\ref{prop-27-6-6}\ref{prop-27-6-6-ii}.

Then one can prove easily that the total contribution of this zone to (\ref{27-7-45}) does not exceed
$CB^{\frac{2}{5}} Z^{-\frac{1}{5}}R_d^2\bar{\gamma}^{-4}$\,\footref{foot-27-37} which is the third term in the parenthesis of (\ref{27-7-42}) and (\ref{27-7-43}) for $M=1$ and $M\ge 2$ respectively\,\footref{foot-27-38}.
\end{proof}

\begin{subappendices}
\chapter{Appendices}
\label{sect-27-A}

\section{Generalization of Lieb-Loss-Solovej estimate}
\label{sect-27-A-1}

\begin{proposition}\label{prop-27-A-1}
Consider operator $H$ defined by \textup{(\ref{book_new-26-1-1})} with $A= A'+A''$,
$A'= (A'_1(x'), A'_2(x'),0)$, $x'=(x_1,x_2)$, and
$A''=(A''_1(x), A''_2(x),A''_3(x))$ on $\Omega$. Assume that
\begin{equation}
\int B^{\prime\,2}\,dx \ge \int B^{\prime\prime\,2}\,dx
\label{27-A-1}
\end{equation}
with $B=|\nabla \times A|$, $B'=|\nabla \times A'|$, $B''=|\nabla \times A''|$.
Then
\begin{multline}
-\Tr (H_{A,V}^-) \le \\
C\int V_+ ^{\frac{5}{2}}(x)\,dx +
C\Bigl(\int B^2\,dx\Bigr)^{\frac{1}{2}}
\Bigl(\int B^{\prime\prime\,2}\,dx+\int V^2\,dx\Bigr)^{\frac{1}{4}}
\Bigl(\int V^4\,dx\Bigr)^{\frac{1}{4}}+\\
C\Bigl(\int B^2\,dx\Bigr)^{\frac{3}{8}}\Bigl(\int V^2\,dx\Bigr)^{\frac{3}{8}}
\Bigl(\int V^4\,dx\Bigr)^{\frac{1}{4}}.
\label{27-A-2}
\end{multline}
\end{proposition}

\begin{proof}
Without any loss of the generality we can assume that $V\ge 0$.
We apply ``moving frame technique'' of Lieb-Loss-Solovej \cite{lieb:loss:solovej}.
\begin{multline*}
-\Tr (H_{A,V}^-)=\int_0^\infty \N^- (H_{A,V} +\lambda)\,d\lambda=
\int_0^\infty \N^- (H_{A,0} +\lambda-V)\,d\lambda\\
\le \int_0^\infty \N^- \bigl(H_{A,0} +(\lambda-V)\phi(\lambda)\bigr)\,d\lambda
\end{multline*}
with $\phi(\lambda)=\max (1,\lambda\mu^{-1})$ since $H_{A,0}\ge 0$. Since \begin{gather*}
H_{A,0}=\bigl( P\cdot \boldupsigma\bigr)^2=
\bigl( P'\cdot \boldupsigma'\bigr)^2 + P_3^2 +\sum_{j=1,2} [P_j,P_3] \, [\upsigma_j,\upsigma_3] \ge H'_{A,0} + P_3^2 - B''\\
\shortintertext{with}
H'_{A,0}= \bigl( P'\cdot \boldupsigma'\bigr)^2=
\bigl( \sum_{j=1,2}P_j\cdot \upsigma_j\bigr)^2
\end{gather*}
$P_j=D_j-A_j$ we conclude that $-\Tr (H_{A,V}^-)$ does not exceed
\begin{equation}
\int_0^\infty
\N^- \bigl(H'_{A,0} + P_3^2 -B'' +(\lambda-V)\phi(\lambda)\bigr)\,d\lambda;
\label{27-A-3}
\end{equation}
consider this integral over $(\mu,\infty)$; it is
\begin{equation}
\int_\mu^\infty
\N^- \bigl(H'_{A,0} + P_3^2 -B'' +(\lambda-V)\lambda \mu^{-1}\bigr)\,d\lambda
\label{27-A-4}
\end{equation}
and since $H'_{A,0}\ge 0$ this integral does not exceed
\begin{equation*}
\int_\mu^\infty
\N^- \bigl(H'_{A,0} + a [P_3^2 -B''+(\lambda-V)\lambda \mu^{-1}]\bigr)\,d\lambda
\end{equation*}
with $a\ge 1$; since $H'_{A,0}\ge P_1^2+P_2^2 - |B|$ , the latter integral does not exceed
\begin{equation*}
\int_\mu^\infty
\N^-\bigl(P_1^2+P_2^2-B + a [P_3^2 -B''+(\lambda-V)\lambda \mu^{-1}]\bigr) \,d\lambda
\end{equation*}
which can be estimated due to LCR inequality after rescaling
$x_3\mapsto a^{\frac{1}{2}}x_3$, $P_3\mapsto a^{-\frac{1}{2}}P_3$ by
\begin{multline*}
C\int \int_\mu^\infty a^{-\frac{1}{2}}
\bigl(B +a[B''+(V-\lambda)\lambda \mu^{-1}]\bigr)_+^{\frac{3}{2}}\,d\lambda dx\\ \begin{aligned}
\le &C\int \int_\mu^\infty a^{-\frac{1}{2}}
\bigl(B -\frac{1}{3} a\lambda ^2\mu^{-1}\bigr)_+^{\frac{3}{2}} \,d\lambda dx\\
+&C\int \int_\mu^\infty a
\bigl(B''- \frac{1}{3}\lambda^2 \mu^{-1}\bigr)_+^{\frac{3}{2}}\,d\lambda dx\\
+&C\int \int_\mu^\infty a (\lambda \mu^{-1})^{\frac{3}{2}}
(V-\frac{1}{3}\lambda)_+^{\frac{3}{2}}\,d\lambda dx
\end{aligned}\\
\le Ca^{-1}\mu^{\frac{1}{2}} \int B^2\,dx +
Ca \mu^{\frac{1}{2}}\int B^{\prime\prime2}\,dx +
Ca \mu^{-\frac{3}{2}} \int V^4\,dx
\end{multline*}
where we integrated over $[0,\infty]$. Optimizing with respect to $a\ge 1$ we get
\begin{multline}
C\Bigl( \int B^2\,dx \Bigr)^{\frac{1}{2}}
\Bigl(\mu \int B^{\prime\prime2}\,dx +
\mu^{-1} \int V^4\,dx\Bigr)^{\frac{1}{2}}\\
+C \mu^{\frac{1}{2}}\int B^{\prime\prime2}\,dx+
C\mu^{-\frac{3}{2}}\int V^4\,dx.
\label{27-A-5}
\end{multline}
Therefore integral (\ref{27-A-4}) does not exceed (\ref{27-A-5}).

Consider (\ref{27-A-3}) over $[0,\mu]$; it is
\begin{equation}
\int_0^\mu
\N^- \bigl(H'_{A,0} + P_3^2 -B'' +(\lambda-V)\bigr)\,d\lambda
\label{27-A-6}
\end{equation}
and exactly as before it does not exceed
\begin{multline*}
C\int \int_0^\mu a^{-\frac{1}{2}}
\bigl(B -\frac{1}{3} a\lambda \bigr)_+^{\frac{3}{2}} \,d\lambda dx
+C\int \int_0^\mu a
\bigl(B''- \frac{1}{3}\lambda\bigr)_+^{\frac{3}{2}}\,d\lambda dx\\
+C\int \int_0^\mu a
(V-\frac{1}{3}\lambda)_+^{\frac{3}{2}}\,d\lambda dx
\end{multline*}
with $a\ge 1$; in first two integrals we replace (in parenthesis) $\lambda$ by $\lambda^2\mu^{-1}$ and expand integral to $[0,\infty]$ arriving to
\begin{equation*}
Ca^{-1}\mu^{\frac{1}{2}} \int B^2\,dx +
Ca \mu^{\frac{1}{2}}\int B^{\prime\prime2}\,dx
+C\int a V^{\frac{3}{2}}\min (V,\mu)\ dx;
\end{equation*}
optimizing with respect to $a\ge 1$ we get
\begin{multline}
C\Bigl( \int B^2\,dx \Bigr)^{\frac{1}{2}}
\Bigl( \mu \int B^{\prime\prime2}\,dx
+\mu^{\frac{1}{2}} \int V^{\frac{3}{2}}\min (V,\mu)\ dx\Bigr)^{\frac{1}{2}}\\
+ C \mu^{\frac{1}{2}}\int B^{\prime\prime2}\,dx
+C\int  V^{\frac{3}{2}}\min (V,\mu)\ dx.
\label{27-A-7}
\end{multline}
Therefore integral (\ref{27-A-6}) does not exceed (\ref{27-A-7}) and the whole expression does not exceed
\begin{multline*}
C\Bigl( \int B^2\,dx \Bigr)^{\frac{1}{2}}
\Bigl(\mu \int B^{\prime\prime2}\,dx +
\mu^{-1} \int V^4\,dx+ \mu^{\frac{1}{2}} \int V^{\frac{3}{2}}\min (V,\mu)\ dx \Bigr)^{\frac{1}{2}}\\[3pt]
+C \mu^{\frac{1}{2}}\int B^{\prime\prime2}\,dx+
C\mu^{-\frac{3}{2}}\int V^4\,dx
+C\int  V^{\frac{3}{2}}\min (V,\mu)\, dx;
\end{multline*}
replacing $\min (V,\mu)$ by $V^{\frac{1}{2}}\mu^{\frac{1}{2}}$ and $V$ in the first and second lines respectively we get
\begin{multline*}
C\Bigl( \int B^2\,dx \Bigr)^{\frac{1}{2}}
\Bigl(\mu \int B^{\prime\prime2}\,dx +
\mu^{-1} \int V^4\,dx+ \underbracket{\mu \int V^2 \, dx }\Bigr)^{\frac{1}{2}}\\[3pt]
+C \mu^{\frac{1}{2}}\int B^{\prime\prime2}\,dx+
C\mu^{-\frac{3}{2}}\int V^4\,dx
+C\int  V^{\frac{5}{2}}\, dx.
\end{multline*}
We skip the last term as it is already in (\ref{27-A-2}); temporarily skip monotone increasing by $\mu$ selected term; optimizing the rest by $\mu>0$ we get
\begin{equation*}
C\Bigl(\int B^2\,dx\Bigr)^{\frac{1}{2}}
\Bigl(\int B^{\prime\prime\,2}\,dx\Bigr)^{\frac{1}{4}}
\Bigl(\int V^4\,dx\Bigr)^{\frac{1}{4}}.
\end{equation*}
Now we are left with
\begin{equation*}
C\mu ^{\frac{1}{2}} \Bigl(\int B^2\,dx\Bigr)^{\frac{1}{2}}
\Bigl(\int V^2\,dx\Bigr)^{\frac{1}{2}}+
C\mu ^{-\frac{1}{2}} \Bigl(\int B^2\,dx\Bigr)^{\frac{1}{2}}
\Bigl(\int V^4\,dx\Bigr)^{\frac{1}{2}}+
C\mu^{-\frac{3}{2}}\int V^4\,dx
\end{equation*}
optimizing by $\mu>0$ we get
\begin{equation*}
C \Bigl(\int B^2\,dx\Bigr)^{\frac{1}{2}}
\Bigl(\int V^2\,dx\Bigr)^{\frac{1}{4}}
\Bigl(\int V^4\,dx\Bigr)^{\frac{1}{4}}+
C\Bigl(\int B^2\,dx\Bigr)^{\frac{3}{8}}\Bigl(\int V^2\,dx\Bigr)^{\frac{3}{8}}
\Bigl(\int V^4\,dx\Bigr)^{\frac{1}{4}}
\end{equation*}
which concludes the proof.
\end{proof}

\section{Electrostatic inequality}
\label{sect-27-A-2}

\begin{proposition}\label{prop-27-A-2}
\begin{enumerate}[fullwidth, label=(\roman*)]
\item\label{prop-27-A-2-i}
Let $B\le Z^{3}$, $\alpha \le \kappa^*Z^{-1}$, $c^{-1}Z\le N\le cZ$. Further, if
$B\ge Z^{\frac{4}{3}}$ then
$\alpha B^{\frac{4}{5}}Z^{-\frac{2}{5}} \le \epsilon$. Then
\begin{multline}
\sum_{1\le j< k\le M}
\int |x_j-x_k|^{-1}|\Psi (x_1,\ldots,x_N)|^2\,sdx_1\cdots dx_N\\
\ge
\D (\rho_\Psi,\rho_\Psi) -
C \bigl(Z^{\frac{5}{3}}+B^{\frac{2}{5}}Z^{\frac{17}{15}}+B\bigr);
\label{27-A-8}
\end{multline}
\item\label{prop-27-A-2-ii}
Further, as $B\le Z$ one can replace the last term in \textup{(\ref{27-A-8})} by $\Dirac- CZ^{\frac{5}{3}-\delta}$.
\end{enumerate}
\end{proposition}

\begin{remark}\label{rem-27-A-3}
Without self-generated magnetic field the last term was
$-C(Z^{\frac{5}{3}}+B^{\frac{2}{5}}Z^{\frac{17}{15}})$ and probably it holds here but does not give us any advantage; for $B\ge Z$ we need only
$C \bigl(Z^{\frac{5}{3}}+B^{\frac{4}{5}}Z^{\frac{3}{5}}\bigr)$ estimate.
\end{remark}

\begin{proof}
As we prove estimate from below we replace first
\begin{gather*}
\langle \sum_{1\le j\le N} (H_{A,V})_{x_j}\Psi,\,\Psi\rangle
\shortintertext{by}
\langle \sum_{1\le j\le N} (H_{A,W})_{x_j}\Psi,\,\Psi\rangle+
\int (W-V)\rho_\Psi\,dx
\end{gather*}
without changing anything else and then we estimate the first term here from below by $\Tr (H^-_{A,W})$ ; then in
$\Tr (H^-_{A,W})+\alpha^{-1}\|\partial A'\|^2$ we replace $A'$ by a minimizer for this expression (rather than for the original problem) only decreasing this expression. So we can now consider $A'$ a minimizer of Section~\ref{sect-27-5}.

Then we follow arguments of Appendix~\ref{25-A-1} but now we need to justify Magnetic Lieb-Thirring estimate (\ref{book_new-5-A-12})
\begin{equation}
\Tr (H_{A,W}^-) \ge -C \int P_B (W)\,dx
\tag{\ref*{book_new-25-A-12}}\label{25-A-12x}
\end{equation}
in the current settings and with $W: CP'(W)=\rho_\Psi$ and then
$W\asymp \min (B^{-2}\rho_\Psi^2; \rho_\Psi^{\frac{2}{3}})$.

Estimate (\ref{25-A-12x}) has been proven in L.~Erd\"os~\cite{erdos:magnetic} (Theorem 2.2) under assumption that intensity of the magnetic field $\vec{B}(x)$ has a constant direction which was the case in Chapter~\ref{book_new-sect-25} but not here.

However we actually we do not need (\ref{25-A-12x}); we need this estimate but with an extra term $-C R$ in the right-hand expression where in \ref{prop-27-A-2-i} $R= (Z^{\frac{5}{3}} + B^{\frac{4}{5}}Z^{\frac{3}{5}})$ is the last term in (\ref{27-A-8}) and in \ref{prop-27-A-2-ii} $R=CZ^{\frac{5}{3}-\delta}$.

Further, the same paper L.~Erd\"os~\cite{erdos:magnetic}) provides an alternative version of Theorem 2.2: as long as $|\partial A'|\le B$ it is sufficient to estimate $|\partial^2 A|\le c B^{\frac{3}{2}}$.

One can check easily that this happens as \underline{either} $B\le Z^2$ and
$\ell(x)\ge r_*\Def B^{-\frac{3}{2}}Z^{\frac{1}{3}}$ \underline{or} $Z^2 \le B\le Z^3$ and $\ell(x)\ge r_*= Z^{-1}$. Introducing partition into two zones
$\{x:\, \ell (x)\ge r_*\}$ and $\{x:\,\ell(x)\le 2r_*\}$ adds
$\ell^{-2}\phi (x)$ with
$\phi (x)=\boldsymbol{1}_{\{x:\,r_*\le \ell(x)\le 2r_*\}}$ which adds $-CR$ to the right-hand expression of (\ref{25-A-12x}).

Therefore we need to deal with zone $\{x:\, \ell(x)\le 2r_*\}$. In this zone however we can neglect an external field; indeed, as in Remark~\ref{rem-26-4-1} we get the same estimate (\ref{book_new-26-4-25}) but with $\mathsf{B}$ intensity of the combined field; however $\int B^2 \,dx$ over this zone does not exceed $C R$. This concludes proof of Statement~\ref{prop-27-A-2-i}.

Statement~\ref{prop-27-A-2-ii} is proven in the same manner as in Appendix~\ref{25-A-1}. We leave details to the reader.
\end{proof}

\section{Estimates for \texorpdfstring{$(hD_{x_j}-\mu x_j)e(x,y,\tau)|_{x=y}$}{\textGamma\textxinferior((hD\textjinferior -\textmu x\textjinferior)e(.,.,\texttau))} for pilot model operator}
\label{sect-27-A-3}

We will use here notations of Subsection~\ref{book_new-sect-26-5-1}.

\subsection{Calculations}
\label{sect-27-A-3-1}

Let us calculate the required expressions as $X=\bR^3$ and $A(x)$ and $V(x)$ are linear. To do this we can consider just Schr\"odinger operator (acting on vector-functions) and then replace $V$ by $V\pm \mu h $ where $\mu$ is the magnetic intensity; since $\upsigma_j\upsigma_k+\upsigma_k\upsigma_j=2\updelta_{jk}$ we have to consider scalar a Schr\"odinger operator. Let us apply calculations of Subsection~\ref{sect-16-5-1} with operator
\begin{equation}
H=h^2D_1^2 + (hD_2-\mu x_1)^2 + hD_3^2 - 2\alpha x_1 -2\beta x_3
\label{27-A-9}
\end{equation}
where without any loss of the generality we assume that $\alpha\ge 0$, $\beta\ge 0$.

After rescaling $x\mapsto \mu x$,
$y\mapsto \mu y$, $t\mapsto \mu t$, $\mu \mapsto 1$ (but we will need to use old $\mu$ in calculations), $h\mapsto \hbar=\mu h$ we have $U(x,y,t)= U_{(1)}(x_3,y_3,t) U_{(2)}(x',y',t)$ where from (\ref{book_new-16-5-4})
\begin{multline}
U_{(1)}(x_3,y_3,t)=\\
\frac{1}{2} \mu (2\pi \hbar |t|)^{-\frac{1}{2}}\exp \Bigl(i\hbar^{-1}\bigl(\mu^{-1}\beta t (x_3+y_3)+\frac{1}{8}t^{-1}(x_3-y_3)^2 +\frac{1}{3}\mu^{-2} \beta^2t^3\bigr)\Bigr);
\label{27-A-10}
\end{multline}
and repeating (\ref{book_new-16-1-9})--(\ref{27-A-11}) we get
\begin{equation}
U_{(2)}(x,y,t)=i(4\pi\hbar)^{-1} \mu^2
\csc(t) \,e^{i\hbar^{-1}\bar{\phi}_{(2)}(x',y',t)}
\label{27-A-11}
\end{equation}
with
\begin{align}
\bar{\phi}_{(2)}\Def\label{27-A-12}
&-\frac{1}{4}\cot(t) (x_1-y_1)^2\\
&+\frac{1}{2} (x_1+y_1+2\alpha\mu^{-1})(x_2-y_2+2t\alpha\mu^{-1})\notag\\
&-\frac{1}{4}\cot(t)(x_2-y_2+2t\alpha\mu^{-1})^2 - t\alpha^2\mu^{-2}.\notag
\end{align}
Then
\begin{equation}
U(x,y,t)=i(2\pi h)^{-\frac{3}{2}} |t|^{-\frac{1}{2}} \mu^{\frac{3}{2}}
\csc(t) \,e^{i\hbar^{-1}\bar{\phi}(x,y,t)}
\label{27-A-13}
\end{equation}
with
\begin{multline}
\bar{\phi}\Def
-\frac{1}{4}\cot(t) (x_1-y_1)^2\\
+\frac{1}{2} (x_1+y_1+2\alpha\mu^{-1})(x_2-y_2+2t\alpha\mu^{-1})
-\frac{1}{4}\cot(t)(x_2-y_2+2t\alpha\mu^{-1})^2 - t\alpha^2\mu^{-2}+ \\
\mu^{-1}\beta t (x_3+y_3)+\frac{1}{8}t^{-1}(x_3-y_3)^2 +
\frac{1}{3}\mu^{-2} \beta^2t^3\bigr);
\label{27-A-14}
\end{multline}

Therefore applying first $\hbar D_{x_1}$, $\hbar D_{x_2}-\mu x_1$, or $\hbar D_{x_3}$ and setting after this $x=y=0$ we conclude that
\begin{phantomequation}\label{27-A-15}\end{phantomequation}
\begin{align}
&\bigl(\hbar D_{x_1}U)|_{x=y=0}=
i\alpha \mu^{-1}t \!\times\! (2\pi h)^{-\frac{3}{2}} |t|^{-\frac{1}{2}} \mu^{\frac{3}{2}}e^{i\hbar^{-1}\varphi(t)},
\tag*{$\textup{(\ref*{27-A-15})}_1$}\\
&\bigl(\hbar D_{x_2}U)|_{x=y=0}=
i\alpha \mu^{-1}(1-t\cot(t)) \!\times\! (2\pi h)^{-\frac{3}{2}} |t|^{-\frac{1}{2}} \mu^{\frac{3}{2}}e^{i\hbar^{-1}\varphi(t)},
\tag*{$\textup{(\ref*{27-A-15})}_2$}\\
&\bigl(\hbar D_{x_3}U)|_{x=y=0}=
i\beta \mu^{-1}t \!\times\! (2\pi h)^{-\frac{3}{2}} |t|^{-\frac{1}{2}} \mu^{\frac{3}{2}}e^{i\hbar^{-1}\varphi(t)}
\tag*{$\textup{(\ref*{27-A-15})}_3$}
\end{align}
with
\begin{equation}
\varphi(t)=\alpha^2\mu^{-2}t -\alpha^2\mu^{-2}t^2\cot(t)+
\frac{1}{3}\mu^{-2} \beta^2t^3.
\label{27-A-16}
\end{equation}
In other words, in comparison with $U|_{x=y=0}$ calculated in Subsection~\ref{sect-16-5-1} expressions
$\bigl(\hbar D_{x_1}U)|_{x=y=0}$, $\bigl(\hbar D_{x_2}U)|_{x=y=0}$ and
$\bigl(\hbar D_{x_3}U)|_{x=y=0}$ acquire factors $\mu^{-1}\alpha t$,
$\mu^{-1}\alpha(1-t\cot(t) )$ and $\mu^{-1}\beta t $ respectively.

Recall that we had 2 cases: $\mu ^2h \le \alpha$ and $\mu ^2h \ge \alpha$.

\subsection{Case \texorpdfstring{$\alpha \ge \mu^2 h$}{\textalpha\textge\textmu\texttwosuperior}}
\label{sect-27-A-3-2}

Then for each $k$, $1\le |k|\le C_0\mu \alpha^{-1}$, the $k$-th tick contributed no more than
\begin{equation}
C\mu h^{-1}\underbracket{
(\mu^2 h/ \alpha|k|)^{\frac{1}{2}}}\times (\mu/h|k| )^{\frac{1}{2}}
\label{27-A-17}
\end{equation}
to $F_{t\to \hbar^{-1}\tau}U|_{x=y=0}$ (see Subsection~\ref{sect-16-5-2}) and then it contributed no more than this multiplied by $|t_k|^{-1}$ i.e.
\begin{equation}
C\mu h^{-1} |k|^{-1}
\underbracket{(\mu^2 h/ \alpha |k|)^{\frac{1}{2}}}\times
(\mu / h|k| )^{\frac{1}{2}}
\label{27-A-18}
\end{equation}
to the corresponding Tauberian expression. Even as we multiply by $\mu^{-1}|k|$ we get (\ref{27-A-17}) again proportional to $|k|^{-1}$; then summation with respect to $k$,
$1\le |k|\le k^*\Def C_0\mu (\alpha+\beta)^{-1}$\,\footnote{\label{foot-27-39} As $|t|\ge k^*$ we have $\phi'(t)\ge c_1$ and then integrating by parts there we can recover factor $(t/k^*)^{-n}$ thus effectively confining us to integration over $\{t:\,|t|\le k^*\}$. This observation can also improve some results of Sections \ref{sect-16-5}--\ref{sect-16-9}.}
returns its value as $k=1$ i.e. $\mu^{\frac{3}{2}}h^{-1}\alpha^{-\frac{1}{2}}$ multiplied by logarithm
$(1+|\log k^*|)$ and therefore we arrive to proposition \ref{prop-27-A-4} below for $j=1,3$.

Let $j=2$. Since $t_k/\cot(t_k)\asymp \alpha^{-1}\mu $ we conclude that contribution of $k$-th tick does not exceed
\begin{equation}
C\mu h^{-1} |k|^{-1}
\underbracket{(\mu^2 h/ \alpha |k|)^{\frac{1}{2}}}\times
(\mu / h|k| )^{\frac{1}{2}}
\label{27-A-19}
\end{equation}
and summation by $|k|\ge 1$ returns its value as $|k|=1$ i.e. $C\mu^{\frac{5}{2}}h^{-1}\alpha^{-\frac{1}{2}}$ and therefore we arrive to proposition \ref{prop-27-A-4} below for $j=2$.

\begin{proposition}\label{prop-27-A-4}
Let $\mu h\le \epsilon_0$, $\tau\asymp 1$, $\alpha \ge \mu^2h$.

Then
\begin{enumerate}[fullwidth, label=(\roman*)]
\item
Expression $(h D_{x_2}-\mu x_1) e(x,y,\tau)|_{x=y=0}$ does not exceed
$C\mu^{\frac{5}{2}}h^{-1}\alpha^{-\frac{1}{2}}$;
\item
Expression $h D_{x_j} e(x,y,\tau)|_{x=y=0}$ does not exceed\\
$C \mu^{\frac{3}{2}}h^{-1}\alpha^{\frac{1}{2}}
(1+|\log \mu (\alpha+\beta)^{-1}|)$, and
$C \mu^{\frac{3}{2}}h^{-1}\beta \alpha^{-\frac{1}{2}}
(1+|\log \mu (\alpha+\beta)^{-1}|)$ for $j=1,3$ respectively.
\end{enumerate}
\end{proposition}

\subsection{Case \texorpdfstring{$\alpha \le \mu^2 h$}{\textalpha\textle\textmu\texttwosuperior}}
\label{sect-27-A-3-3}

If $\mu ^2h \ge \alpha$ then the same arguments work only for
$\bar{k}\Def \mu^2h \alpha^{-1}\le |k|\le k^*$ resulting in contributions
$C \mu^{\frac{3}{2}}h^{-1}\alpha^{\frac{1}{2}}(1+|\log k^*\bar{k}^{-1}|)$,
$C\mu^{\frac{1}{2}}h^{-2}\alpha^{\frac{1}{2}}$, and
$C\mu^{\frac{3}{2}}h^{-1}\beta\alpha^{-\frac{1}{2}}(1+|\log k^*\bar{k}^{-1}|)$ for $j=1,2,3$ respectively as $\bar{k}\le k^*$ (i.e. $\mu h\beta \le \alpha$) or $0$ otherwise.

Let $1\le |k|\le \bar{k}$. We mainly consider the most difficult case $j=2$ and (as $|t|\ge \epsilon_0$) only term arising from $-\alpha\mu^{-1}t\cot(t)$ factor, namely
\begin{equation}
\alpha \mu^{-1} \times \mu^{\frac{3}{2}}h^{-\frac{3}{2}} \int |t|^{-\frac{1}{2}}\cos(t)(\sin(t))^{-2} e^{i\hbar ^{-1}(\varphi(t)-t\tau)}\,dt
\label{27-A-20}
\end{equation}
where we took into account that we need to divide by $t$ and skipped a constant factor.

Consider first (\ref{27-A-20}) with integration over interval
$\{t:\, |t-t_k|\le s_k\}$ near $t_k$. Observe that
\begin{equation}
\phi'(t)= (\sin(t))^{-2} \alpha^2 \mu^{-2}t^2 -
2t(\sin(t))^{-1}\alpha^2 \mu^{-2}t^2 +\beta^2 t^2
\label{27-A-21}
\end{equation}
and transform (\ref{27-A-20}) into
\begin{multline}
\alpha ^{-1}\mu^{\frac{7}{2}}h^{-\frac{1}{2}}
\int_{t_k-s_k}^{t_k+s_k}  |t|^{-\frac{5}{2}}\cos(t)
\partial_t \bigl[e^{i\hbar ^{-1}(\varphi(t)-t\tau)}\bigr]\,dt\\
+\alpha ^{-1} \mu^{\frac{5}{2}}h^{-\frac{3}{2}} \int_{t_k-s_k}^{t_k+s_k} |t|^{-\frac{5}{2}}\cos(t)
\bigl[2t(\sin(t))^{-1}\alpha^2 \mu^{-2} -\beta^2 t^2+\tau\bigr]e^{i\hbar ^{-1}(\varphi(t)-t\tau)}
\label{27-A-22}
\end{multline}
Integrating the first term by parts we get a non-integral term
\begin{equation}
\alpha ^{-1}\mu^{\frac{7}{2}}h^{-\frac{1}{2}}|t|^{-\frac{5}{2}}\cos(t)
e^{i\hbar ^{-1}(\varphi(t)-t\tau)}\bigl|_{t=t_k-s_k}^{t=t_k+s_k}
\label{27-A-23}
\end{equation}
and we get an integral term
\begin{multline}
\alpha ^{-1} \mu^{\frac{5}{2}}h^{-\frac{3}{2}} \int_{t_k-s_k}^{t_k+s_k} \Bigl[\mu h \partial_t \bigl[|t|^{-\frac{5}{2}}\cos(t)\bigr]
\\[3pt]
\shoveright{+ |t|^{-\frac{5}{2}}\cos(t)
\bigl[2t(\sin(t))^{-1}\alpha^2 \mu^{-2} -\beta^2 t^2+\tau\bigr]\Bigr]
e^{i\hbar ^{-1}(\varphi(t)-t\tau)}\,dt}\\[3pt]
= 2\alpha \mu^{\frac{1}{2}}h^{-\frac{3}{2}} \int_{t_k-s_k}^{t_k+s_k}
|t|^{-\frac{3}{2}}\cot(t)
e^{i\hbar ^{-1}(\varphi(t)-t\tau)}\,dt +
O\bigl(\alpha ^{-1}\mu^{\frac{5}{2}}h^{-\frac{3}{2}}s_k|k|^{-\frac{5}{2}}\bigr).
\label{27-A-24}
\end{multline}
Repeating the same trick we can eliminate the first term in the right-most expression. Therefore we arrive to (\ref{27-A-23}) with
$O\bigl(\alpha^{-1} \mu^{\frac{5}{2}}h^{-\frac{3}{2}}s_k|k|^{-\frac{5}{2}}\bigr)$ error.As $s_k\asymp \alpha^2\mu^{-3}h^{-1} k^2$ we get
\begin{equation}
C\mu^{\frac{3}{2}}h^{-\frac{3}{2}} \times (\alpha/\mu^2 h) |k|^{-\frac{1}{2}}
\label{27-A-25}
\end{equation}
error.

On the other hand, consider integral over
$[t_k+s_k, t_{k+1}-s_{k+1}]$, $k\ne 0$. Decomposing
$e^{i\mu^{-1}h^{-1}(\phi(t)-t\tau)}$ into Taylor series with respect to
$\alpha^2 h^{-1} \mu^{-3}\cot(t)$ one can prove easily that expression in question is
\begin{equation*}
\alpha ^{-1}\mu^{\frac{7}{2}}h^{-\frac{1}{2}}|t|^{-\frac{5}{2}}\cos(t)
\bigl(e^{i\hbar ^{-1}(\varphi(t)-t\tau)}-
e^{i\hbar ^{-1}(\varphi_{(1)}(t)-t\tau)}\bigr)
\bigl|_{t=t_k+s_k}^{t=t_{k+1}-s_{k+1}}
\end{equation*}
with $\varphi_{(1)}(t)=\frac{1}{3}\mu^{-2} \beta^2t^3$ and with error not exceeding (\ref{27-A-25}) multiplied by $(1+|\log s_k|)$:
\begin{equation}
C\mu^{\frac{3}{2}}h^{-\frac{3}{2}} \times (\alpha/\mu^2 h) |k|^{-\frac{1}{2}}\times (1+|\log (\alpha^2 \mu^{-3}h^{-1}k^2)|).
\label{27-A-26}
\end{equation}
So non-integral terms with $\varphi$ cancel one another because by the similar arguments we can also cover $[0, t_1-s_1]$ and $[ t_{-1}+s_{-1}]$ and due to non-singularity of $t^{-1}(1-t\cot(t))\csc (t)$ at $t=0$ there will be no non-integral terms with $k=0$. So we are left with
\begin{equation*}
-\alpha ^{-1}\mu^{\frac{7}{2}}h^{-\frac{1}{2}}|t|^{-\frac{5}{2}}\cos(t)
e^{i\hbar ^{-1}(\varphi_{(1)}(t)-t\tau)}\bigr)
\bigl|_{t=t_k-s_k}^{t=t_{k}+s_{k}}
\end{equation*}
and their absolute values do not exceed (\ref{27-A-25}).

Finally, summation of (\ref{27-A-26}) by $k: 1\le |k|\le \min(\bar{k},k^*)$ returns
\begin{multline}
C\mu^{\frac{3}{2}}h^{-\frac{3}{2}} (\alpha/\mu^2 h)^{\frac{1}{2}} \\
\times
\left\{\begin{aligned}
&(1+|\log (\mu h)|)\qquad
&&\beta \mu h\le \alpha,\\
&(\alpha/\beta \mu h)^{\frac{1}{2}}
(1+|\log (\mu h)|+|\log (\alpha/\beta \mu h)|)
\qquad
&&\beta \mu h\ge \alpha.
\end{aligned}\right.
\label{27-A-27}
\end{multline}
and we arrive to Proposition~\ref{prop-27-A-5} below as $j=2$:

\begin{proposition}\label{prop-27-A-5}
Let $\mu h\le \epsilon_0$, $\tau\asymp 1$, and $\alpha\le \mu^2h$.

Then
\begin{enumerate}[fullwidth, label=(\roman*)]
\item
Expression $(h D_{x_2}-\mu x_1) e(x,y,\tau)|_{x=y=0}$ does not exceed
\textup{(\ref{27-A-27})};
\item
Expression $h D_{x_3} e(x,y,\tau)|_{x=y=0}$ does not exceed
\begin{equation}
C\left\{\begin{aligned}
&\mu^{\frac{3}{2}}h^{-1} \beta \alpha^{-\frac{1}{2}}
(1+|\log (\alpha/\beta\mu h)|)\qquad
&&\beta \mu h\le \alpha,\\
&\mu h^{-\frac{3}{2}}\beta^{\frac{1}{2}}
\qquad
&&\beta \mu h\ge \alpha;
\end{aligned}\right.
\label{27-A-28}
\end{equation}
\item
Expression $h D_{x_1} e(x,y,\tau)|_{x=y=0}$ does not exceed
\begin{equation}
C\left\{\begin{aligned}
&\mu^{\frac{3}{2}}h^{-1} \alpha^{\frac{1}{2}}
(1+|\log (\alpha/\beta\mu h)|)\qquad
&&\beta \mu h\le \alpha,\\
&\mu h^{-\frac{3}{2}}\alpha^{\frac{1}{2}}
\qquad
&&\beta \mu h\ge \alpha;
\end{aligned}\right.
\label{27-A-29}
\end{equation}
\end{enumerate}
\end{proposition}

\subsection{Case \texorpdfstring{$\mu h\ge \epsilon_0$}{\textmu h \textge \textepsilon\textzeroinferior}}
\label{sect-27-A-3-4}

As $\mu h\ge \epsilon_0$ we consider a different representation: namely (\ref{book_new-16-1-15}) for a spectral projector in dimension $2$ (again after rescaling where we scale $e_*$ as functions rather than Schwartz kernels):
\begin{multline}
e_{(2)}(x',y',\tau)=\\
(2\pi)^{-1}\mu h^{-1}\sum_{m\in \bZ^+}\int
\upsilon_m \bigl(\eta+\mu^{-\frac{1}{2}}h^{-\frac{1}{2}}(x_1-y_1)\bigr)
\upsilon_m \bigl(\eta-\mu^{-\frac{1}{2}}h^{-\frac{1}{2}}(x_1-y_1)\bigr)\\[3pt]
\times\uptheta\Bigl( \tau - \alpha \mu^{-1}(x_1+y_1) -2\alpha \mu^{-\frac{1}{2}}h^{\frac{1}{2}}\eta -\alpha^2\mu^{-2}-2m \mu h\Bigr) e^{i\mu^{-\frac{1}{2}}h^{-\frac{1}{2}} (x_2-y_2)\eta} \,d\eta
\label{27-A-30}
\end{multline}
where we also replaced $H_{(2)}$ by $H_{(2)}-\mu h$ and $\tau$ by $\mu h +\tau$. Since
\begin{equation}
e(x,y,\tau)=e_{(2)}(x',y',.)*_\tau e_{(1)}(x_3,y_3,.)
\label{27-A-31}
\end{equation}
where $e_{(1)}(x_3,y_3,\tau)$ is a Schwartz kernel of the spectral projector of $1$-dimensional operator
\begin{equation}
\mu^2 h^2 D_3^2 - 2\beta \mu^{-1}x_3
\label{27-A-32}
\end{equation}
we conclude that
\begin{multline}
e (x,y,\tau)\\
=(2\pi)^{-1}\mu h^{-1}\sum_{m\in \bZ^+}\int
\upsilon_m \bigl(\eta+\mu^{-\frac{1}{2}}h^{-\frac{1}{2}}(x_1-y_1)\bigr)
\upsilon_m \bigl(\eta-\mu^{-\frac{1}{2}}h^{-\frac{1}{2}}(x_1-y_1)\bigr)\\[3pt]
\times e_{(1)}\Bigl(x_3,y_3, \tau - \alpha \mu^{-1}(x_1+y_1) -2\alpha \mu^{-\frac{1}{2}}h^{\frac{1}{2}}\eta -\alpha^2\mu^{-2}-2m \mu h\Bigr) e^{i\mu^{-\frac{1}{2}}h^{-\frac{1}{2}} (x_2-y_2)\eta} \,d\eta.
\label{27-A-33}
\end{multline}
Then
\begin{multline}
(\mu hD_{x_2}- x_1)e (x,y,\tau)\bigl|_{x=y=0}\\
=(2\pi)^{-1}\mu^{\frac{3}{2}} h^{-\frac{1}{2}}\sum_{m\in \bZ^+}\int
\upsilon_m ^2(\eta)\eta
\times e_{(1)}\Bigl(0,0, \tau - 2\alpha\mu^{-\frac{1}{2}}h^{\frac{1}{2}}\eta -\alpha^2\mu^{-2}-2m\mu h\Bigr) \,d\eta
\label{27-A-34}
\end{multline}
and since $\upsilon_m(.)$ is an even (odd) function for even (odd) $m$ we can replace
$e_{(1)}(0,0,\tau' -2\alpha\mu^{-\frac{1}{2}}h^{\frac{1}{2}}\eta)$ by \begin{equation}
e_{(1)}(0,0,\tau' -2\alpha\mu^{-\frac{1}{2}}h^{\frac{1}{2}}\eta)-
e_{(1)}(0,0,\tau' +2\alpha\mu^{-\frac{1}{2}}h^{\frac{1}{2}}\eta).
\label{27-A-35}
\end{equation}
In virtue of Subsubsection~5.2.1.3~\emph{\nameref{sect-5-2-1-3}} we know that an absolute value of this expression does not exceed $Ch^{-\frac{1}{2}}\alpha^{\frac{1}{2}}$\,\footnote{\label{foot-27-40} Only in the worst case when $ |\tau -2m\mu h|$ is not disjoint from $0$.}
we arrive to estimate\footnote{\label{foot-27-41} In the non-rescaled coordinates.}
\begin{equation}
| (hD_{x_2}- \mu x_1)e (x,y,\tau)\bigl|_{x=y=0}|\le
C\mu^{\frac{3}{2}} h^{-1}\alpha^{\frac{1}{2}}.
\label{27-A-36}
\end{equation}

Further,
\begin{multline}
\mu hD_{x_1} e (x,y,\tau)\bigl|_{x=y=0}\\
=i(2\pi)^{-1} \alpha\mu \sum_{m\in \bZ^+}\int
\upsilon_m ^2(\eta)
\times \partial_\tau e_{(1)}\bigl(0,0, \tau - 2\alpha \mu^{-\frac{1}{2}}h^{\frac{1}{2}}\eta -\alpha^2\mu^{-2}-2m \mu h\bigr) \,d\eta\\
=i(2\pi)^{-1} \mu^{\frac{3}{2}}h^{-\frac{1}{2}} \sum_{m\in \bZ^+}\int
\upsilon_m (\eta)\upsilon'_m(\eta)
\times e_{(1)}\bigl(0,0, \tau - 2\alpha \mu^{-\frac{1}{2}}h^{\frac{1}{2}}\eta -\alpha^2\mu^{-2}-2m \mu h\bigr) \,d\eta
\label{27-A-37}
\end{multline}
and using the same arguments we arrive to estimate\footref{foot-27-41}
\begin{equation}
| hD_{x_1} e (x,y,\tau)\bigl|_{x=y=0}|\le
C\mu^{\frac{3}{2}} h^{-1}\alpha^{\frac{1}{2}}.
\label{27-A-38}
\end{equation}

Finally,
\begin{multline}
\mu hD_{x_3} e (x,y,\tau)\bigl|_{x=y=0}=
(2\pi)^{-1}\mu h^{-1}\sum_{m\in \bZ^+}\int
\upsilon_m ^2(\eta)\\
\times \underbracket{\mu hD_{x_3} e_{(1)}\bigl(x_3,y_3, \tau
-2\alpha \mu^{-\frac{1}{2}}h^{\frac{1}{2}}\eta -\alpha^2\mu^{-2}-2m \mu h\bigr) \bigr|_{x_3=y_3=0} }\,d\eta
\label{27-A-39}
\end{multline}
and again in virtue of Subsubsection~5.2.1.3~\emph{\nameref{sect-5-2-1-3}} we know that an absolute value of selected expression does not exceed $Ch^{-\frac{1}{2}}\beta^{\frac{1}{2}}$\,\footnote{\label{foot-27-42} Without applying $(hD_j-\mu A_j(x))$ but it does not matter.} and we arrive to
estimate\footref{foot-27-41}
\begin{equation}
|hD_{x_3}e (x,y,\tau)\bigl|_{x=y=0}|\le
C\mu h^{-\frac{3}{2}}\beta^{\frac{1}{2}}.
\label{27-A-40}
\end{equation}

Therefore we have proven
\begin{proposition}\label{prop-27-A-6}
Let $\mu h\ge 1$, $\alpha\le 1$, $\beta\le 1$, $|\tau|\le c_01$. Then for operator $(H-\mu h)$ estimates \textup{(\ref{27-A-36})}, \textup{(\ref{27-A-38})} and \textup{(\ref{27-A-40})} hold.
\end{proposition}

\subsection{Tauberian estimates}
\label{sect-27-A-3-5}

\begin{remark}\label{rem-27-A-7}
Assume now that all assumptions are fulfilled only in $B(0,\ell)$ rather than in $\bR^3$. Then there is also a Tauberian estimate which should be added to Weyl estimate. This Tauberian estimate (the same for all $j=1,2,3$) coincides with the Tauberian estimate was calculated in Chapter~\ref{sect-16}. Namely

\begin{enumerate}[label=(\roman*), fullwidth]
\item\label{rem-27-A-7-i}
As $\mu h\le 1$, $\ell \ge C_0\mu^{-1}$ this Tauberian estimate was calculated in Proposition~\ref{prop-16-5-2}(ii)\footref{foot-27-41};

\item\label{rem-27-A-7-ii}
As $\mu h \ge 1$, $\ell \ge C_0h$ this Tauberian estimate was calculated in Proposition~\ref{prop-16-5-7}(i) and corollary~\ref{cor-16-5-8}(i)\footref{foot-27-41} and it does not exceed
$C\mu h^{-\frac{3}{2}}\ell^{-1}$.
\end{enumerate}
\end{remark}

\end{subappendices}


\end{document}